\documentclass{article}
\usepackage[utf8]{inputenc}

\usepackage{amsmath, amsthm, amssymb, latexsym,epsfig,amsthm,enumerate,multicol,wasysym}
\usepackage{color}
\usepackage{graphicx}
\usepackage{nccrules}
\usepackage{xfrac}
\usepackage{hyperref}
\usepackage{textcomp}
\usepackage{tikz}
\usepackage{xcolor}
\usepackage{comment}
\usepackage[pagewise]{lineno}
\usepackage{mathtools}

\usepackage[inline]{enumitem}

\newcommand{\R}{\mathbb R}
\newcommand{\N}{\mathbb N}
\newcommand{\W}{\mathcal W}

\newcommand{\K}{\mathcal K}

\newcommand{\E}{\mathcal E}
\newcommand{\B}{\mathcal B}
\newcommand{\oz}{\odot}

\newcommand{\Po}{\mathcal P}

\newcommand{\C}{\mathcal C}
\newcommand{\Vp}{V^\perp}
\newcommand{\Wp}{W^\perp}

\newcommand{\cC}{\mathcal{C}}
\newcommand{\cR}{\mathcal{R}}
\newcommand{\cP}{\mathcal{P}}
\newcommand{\cB}{\mathcal{B}}
\newcommand{\cW}{\mathcal{W}}
\newcommand{\cK}{\mathcal{K}}
\newcommand{\cA}{\mathcal{A}}
\newcommand{\cM}{\mathcal{M}}
\newcommand{\cX}{\mathcal{X}}
\newcommand{\cY}{\mathcal{Y}}
\newcommand{\cF}{\mathcal{F}}
\newcommand{\cH}{\mathcal{H}}
\newcommand{\cQ}{\mathcal{Q}}
\newcommand{\Z}{\mathbb{Z}}

\def \old {{\text{\tiny old}}}

\newcommand{\sgn}{\operatorname{sgn}}

\newcommand{\dist}{\operatorname{dist}}

\newcommand{\supp}{\operatorname{supp}}
\newcommand{\diam}{\operatorname{diam}}
\newcommand{\poly}{\operatorname{poly}}

\newcommand{\spn}{\operatorname{span}}

\newtheoremstyle{mytheor}
    {1ex}{1ex}{\normalfont}{0pt}{\scshape}{.}{1ex}
    {{\thmname{#1 }}{\thmnumber{#2}}{\thmnote{ (#3)}}}

\newtheorem{thm}{Theorem}[section]
\newtheorem{defi}[thm]{Definition}
\newtheorem{lemma}[thm]{Lemma}

\newtheorem{proposition}[thm]{Proposition}
\newtheorem{cor}[thm]{Corollary}

\newtheorem{rmk}[thm]{Remark}

\newtheorem{problem}{Problem}

\begin{document}
\title{The norm of linear extension operators for $C^{m-1,1}(\R^n)$}
\date{}
\author{J. Carruth, A. Frei-Pearson, A. Israel}
\maketitle

\begin{abstract}
    Fix integers $m\ge 2$, $n\ge 1$. We prove the existence of a bounded linear extension operator for $C^{m-1,1}(\R^n)$ with operator norm at most $\exp(\gamma D^k)$, where $D := \binom{m+n-1}{n}$ is the number of multiindices of length $n$ and order at most $m-1$, and $\gamma,k > 0$ are absolute constants (independent of $m,n,E$). Upper bounds on the norm of this operator are relevant to basic questions about fitting a smooth function to data. Our results improve on a previous construction of extension operators of norm at most $\exp(\gamma D^k 2^D)$. Along the way, we establish a finiteness theorem for $C^{m-1,1}(\R^n)$ with improved bounds on the involved constants.
\end{abstract}

\section{Introduction}
Fix $m\ge 1$, $n\ge 1$. We let $C^m(\R^n)$ denote the Banach space of all $m$-times continuously differentiable functions $F : \R^n \rightarrow \R$ whose partial derivatives up to order $m$ are bounded functions on $\R^n$. We equip $C^m(\R^n)$ with a standard norm:
\[
\|F \|_{C^m(\R^n)} := \sup_{x \in \R^n} \max_{|\alpha| \leq m} | \partial^\alpha F(x)|.
\]
Here, for a multiindex $\alpha = (\alpha_1,\dots,\alpha_n) \in \mathbb{N}^n$, we write $|\alpha| := \sum_{j} \alpha_j$ to denote the order of $\alpha$. We write $\partial^\alpha F(x)  = \partial^{\alpha_1}_1 \cdots \partial^{\alpha_n}_n F(x)$ for the $\alpha^{\text{th}}$ partial derivative of a function $F \in C^m_{loc}(\R^n)$. We also define $\alpha! := \prod_{i=1}^n \alpha_i !$.

The following problem goes back to Whitney \cite{whitney1934analytic,Whitney1934, whitney1934functions}. Let $E$ be an arbitrary subset of $\R^n$. Given a function $f : E \rightarrow \R$, determine whether there exists a function $F \in C^m(\R^n)$ with $F=f$ on $E$. 

Whitney's problem was solved by C. Fefferman in 2006 \cite{F4}.\footnote{The Whitney problem has a long history with contributions by many authors; below, we discuss some of the most relevant to our work. For a more complete history see \cite{FI-book} and the references therein.} In a remarkable series of papers, Fefferman posed and solved a variety of related problems. In three of these papers \cite{F5,FK0, FK1}, two of them joint with B. Klartag, the authors connected this work to the practical problem of computing a $C^m$ interpolant for a given set of data. 

Suppose now that $E$ is a finite subset of $\R^n$. We define the trace norm of a function $f:E\rightarrow \R$ by
\[
||f||_{C^m(E)}:= \inf \{ ||F||_{C^m(\R^n)} : F= f \text{ on } E\}. 
\]
A function $F : \mathbb{R}^n \rightarrow \R$ is an \emph{interpolant} of $f$ if $F = f$ on $E$. Given $C \geq 1$, a function $F \in C^m(\R^n)$ is a \emph{$C$-optimal interpolant} of $f$ provided that $F = f$ on $E$ and $||F||_{C^m(\R^n)}\le C ||f||_{C^m(E)}$. That is, $F$ is an interpolant of $f$ with $C^m$ norm that is within a factor of $C$ of the optimal value. In \cite{F5, FK0, FK1}, Fefferman and Klartag proved the following theorem.

\begin{thm}\label{thm: FK}
Fix $m\ge 1$, $n\ge 1$. Let $E\subseteq \R^n$ be a finite set with cardinality  $\#(E)=N$ and fix $f:E\rightarrow\R$. There exists an algorithm that computes a $C$-optimal interpolant $F \in C^m(\R^n)$ of $f$. Specifically, the algorithm takes as input $(E,f,m)$ and performs $C_1 N \log N$ units of one-time work, on an idealized (von Neumann) computer with $C_2 N$ units of memory. Given $x \in \R^n$, the computer responds to a query by returning the values of $\partial^\alpha F(x)$ for all $\alpha$ with $|\alpha| \leq m$, where $F$ is a $C$-optimal interpolant of $f$. The algorithm requires $C_3 \log N$ computer operations to answer a query. The constants $C, C_1, C_2, C_3$ depend only on $m$ and $n$.
\end{thm}

For details on the model of computation, including an explanation of the terms ``one-time work'', ``query'', or what it means to ``compute'' a function on $\R^n$, see \cite{F5, FK0, FK1}.

We note that (1) the running time of the algorithm in Theorem \ref{thm: FK} likely has optimal dependence on $N = \#(E)$ and (2) this is the only known algorithm for solving the $C^m$ interpolation problem for \underline{arbitrary} finite sets efficiently in $N$. Therefore, at least in theory, this algorithm could have widespread practical application.

Unfortunately, the constant $C$ in Theorem \ref{thm: FK} grows rapidly with $m$ and $n$, rendering the algorithm impractical for real-world applications. While $C$ is not computed explicitly in \cite{F5, FK0, FK1}, an examination of the arguments in those papers shows that one must take $C$ to have order of magnitude at least $\exp(\gamma D^k 2^D)$ for some real number $\gamma > 0$ and integer $k>0$; here $D:= {m+n-1\choose n}$ denotes the dimension of the vector space of polynomials in $n$ variables of degree at most $m-1$. In other words, the optimality guarantees on the interpolant produced by this algorithm deteriorate rapidly as $n$ and $m$ grow. Any practical version of Theorem \ref{thm: FK} will have to address this issue. There is considerable interest in finding such an algorithm; see \cite{fefferman2009whitney}.

The proof of Theorem \ref{thm: FK} is based on a \emph{finiteness theorem} for  $C^{m-1,1}(\R^n)$. This theorem is the source of the double exponential dependence on $D$ of the constant $C$ in Theorem \ref{thm: FK}.  Next, we state this result.

We let $C^{m-1,1}(\R^n)$ denote the space of all $(m-1)$-times differentiable functions $F:\R^n\rightarrow \R$ whose $(m-1)^{\text{rst}}$ order partial derivatives are Lipschitz continuous on $\R^n$. We equip this space with a seminorm:
\[
||F||_{C^{m-1,1}(\R^n)}:=\sup_{x,y\in\R^n}\left(\sum_{|\alpha|=m-1}\frac{(\partial^\alpha F(x) - \partial^\alpha F(y))^2}{|x-y|^2}\right)^{1/2}.
\]
Given a ball $B \subseteq \R^n$, we write $C^{m-1,1}(B)$ for the corresponding space of $C^{m-1,1}$ functions $F :B \rightarrow \R$.

\begin{thm}\label{thm: bsf fp}[Finiteness theorem for $C^{m-1,1}(\R^n)$ -- see \cite{F1}]

Let $m\ge 2, n\ge 1$. There exist constants $k^\#, C^\#$ depending on $m$ and $n$ such that the following holds.

Let $f : E \rightarrow \R$,  $E \subseteq \R^n$ an arbitrary set. Suppose that for every finite subset $S \subseteq E$ with cardinality $\#(S) \leq k^\#$ there exists a function $F^S \in C^{m-1,1}(\R^n)$ satisfying $F^S = f$ on $S$ and $\| F^S \|_{C^{m-1,1}(\R^n)} \leq 1$.

Then there exists a function $F \in C^{m-1,1}(\R^n)$ with $F = f$ on $E$ and $\| F \|_{C^{m-1,1}(\R^n)} \leq C^\#$.
\end{thm}

The finiteness theorem was first proved in the case $m=2$, $n \ge 1$ by Shvartsman \cite{Shvartsman1988}; in this case, it was shown that one can take $k^\# = 3 \cdot 2^{n-1}$ and $C^\# = A \exp(\gamma n)$, where $A, \gamma > 0$ are absolute constants (independent of $n$). Further, Shvartsman \cite{shvartsman1986lipschitz} proves that the value $k^\# = 3\cdot 2^{n-1}$ is the smallest possible when $m=2$. In other words, if $k^\# < 3 \cdot 2^{n-1}$ then the finiteness theorem fails to hold for any $C^\# > 1$.

Theorem \ref{thm: bsf fp} was conjectured to hold for any $m \ge 2$, $n \geq 1$ by Brudnyi and Shvartsman in \cite{Brudnyi1994}.

In \cite{F1}, Fefferman proved the conjecture of Brudnyi and Shvartsman. He showed that Theorem \ref{thm: bsf fp} holds for any $m \ge 2$, $n \ge 1$ with $k^\# \le (D+1)^{3\cdot 2^D}$. He did not state an explicit bound on the value of $C^\#$, but one can check that his proof gives $C^\# \le \exp(\gamma D^k 2^D)$ for absolute constants $\gamma, k>0$ (independent of $m,n$).

Note that in the case $m=2$, Fefferman's result implies Shvartsman's with the caveat that Shvartman's result holds for smaller $k^\#$, $C^\#$. Indeed, if $m =2$, then $D = (n+1)$; therefore Shvartsman's result implies that the finiteness theorem holds with $k^\# = 3\cdot 2^{D-2}$ and $C^\# = A \exp(\gamma D)$.

The constant $C$ in Theorem \ref{thm: FK} inherits its double exponential dependence on $D$ from the constant $C^\#$ in Theorem \ref{thm: bsf fp}. This leads us to pose the following problem.

\begin{problem}\label{problem: c sharp}
Is it possible to improve the dependence of the constant $C^\#$ in Theorem \ref{thm: bsf fp} on $D = \binom{m+n-1}{n}$?
\end{problem}

Progress on Problem \ref{problem: c sharp} is not possible by optimizing the constants in each line of Fefferman's proof of Theorem \ref{thm: bsf fp}. Without going into detail, his proof is by induction, and it produces a $C^\#$ which is exponential in the number of induction steps. The number of induction steps is equal to $2^D$, leading to the double exponential dependence of $C^\#$ on $D$. Thus, lowering the constant $C^\#$ requires new ideas.

In a joint work \cite{coordinateFree} with B. Klartag, we gave a new proof of Theorem \ref{thm: bsf fp} which avoided Fefferman's induction scheme. Our proof relied on semialgebraic geometry and compactness arguments, however, and therefore it did not give an effective bound on $C^\#$. In this paper, we replace the qualitative arguments of \cite{coordinateFree} with quantitative ones and improve the dependence of $C^\#$ on $D$ in Theorem \ref{thm: bsf fp} to exponential in a power of $D$. Specifically, we prove the following theorem.

\begin{thm}\label{thm: new c sharp}
There exist absolute constants $\gamma > 0$ and $k \ge 1$, independent of $m$ and $n$, such that the finiteness theorem for $C^{m-1,1}(\R^n)$ (Theorem \ref{thm: bsf fp}) holds with $C^\# = \exp(\gamma D^{k})$ and $k^\# = \exp(\gamma D^{k})$.
\end{thm}

In \cite{F2}, Fefferman showed that his proof of Theorem \ref{thm: bsf fp} can be modified to produce a $C^\#$-optimal interpolant $F$ that depends linearly on the data $f$. This property is crucial in getting from Theorem \ref{thm: bsf fp} to the algorithm in Theorem \ref{thm: FK}. Our proof also has this property. Specifically, the next theorem is a byproduct of the proof of Theorem \ref{thm: new c sharp}.

Given an arbitrary set $E \subseteq \R^n$ (not necessarily finite), we let $C^{m-1,1}(E)$ denote the space of all restrictions to $E$ of functions in $C^{m-1,1}(\R^n)$, equipped with the standard trace seminorm:
\[
\| f \|_{C^{m-1,1}(E)} := \inf \{ \| F \|_{C^{m-1,1}(\R^n)} : F = f \mbox{ on } E \} \quad (f \in C^{m-1,1}(E)).
\]

\begin{thm}\label{thm: lin op}
There exist absolute constants $\gamma>0$ and $k \ge 1$, independent of $m$ and $n$, such that the following holds. Given $E\subset\R^n$, there exists a linear map $T:C^{m-1,1}(E)\rightarrow C^{m-1,1}(\R^n)$ satisfying $Tf |_E = f$ and $||Tf||_{C^{m-1,1}(\R^n)} \le C^\# ||f||_{C^{m-1,1}(E)}$ for all $f \in C^{m-1,1}(E)$, where $C^\# = \exp(\gamma D^k)$.
\end{thm}

While the constant $C^\#$ in Theorems \ref{thm: new c sharp} and \ref{thm: lin op} is still too large to give rise to a practical algorithm for $C^m$ interpolation, this marks the first progress on Problem \ref{problem: c sharp} since Fefferman's proof of Theorem \ref{thm: bsf fp}.

Theorem \ref{thm: new c sharp} shows that the constant $C^\#$ in the finiteness theorem can be taken to be exponential in a power of $D$. We do not know whether this is the optimal dependence---little is known about lower bounds for $C^\#$. Trivially one has the lower bound $C^\# \ge 1$. One might hope that for any $C^\# > 1$ there exists some $k^\#$ sufficiently large depending on $C^\#$ such that the Finiteness Theorem holds. This is true when $m = 1$ (see \cite{prescribedjets1}), but not in general. In \cite{fefferman2009example}, Fefferman and Klartag show that there exists a constant $c_0>0$ such that Theorem \ref{thm: bsf fp} does not hold for $C^\# < 1+c_0$ for any $k^\#$ when $m=n=2$. It would be interesting to obtain a lower bound on $C^\#$ that grows with $n$ or $m$.

A loose inspection of our proof indicates that it is sufficient to take the power $k=8$ in Theorem \ref{thm: new c sharp}. In the case $m=2$ we know that this is not sharp---Shvartsman's work shows that Theorem \ref{thm: new c sharp} holds with $k=1$ when $m=2$ (see the discussion of Theorem \ref{thm: bsf fp} above).

While this paper is concerned with upper bounds on the constant $C^\#$, there is also interest in understanding the dependence of the constant $k^\#$ on $m$ and $n$. Bierstone and Milman, in \cite{Bierstone2007}, and Shvartsman, in \cite{shvartJetSelect}, independently showed that the Finiteness Theorem holds with $k^\# = 2^D$ and $C^\#$ as in Fefferman's proof of Theorem \ref{thm: bsf fp}, i.e. $C^\# = \exp(\gamma D^k 2^D)$ for absolute constants $\gamma, k > 0$. Our proof gives $k^\#, C^\# \leq \exp( \hat{\gamma} D^{\hat{k}})$ for absolute constants $\hat{\gamma}, \hat{k}$. We would be interested to know whether the Finiteness Theorem holds with $k^\# = 2^D$ and $C^\# \le \exp(\hat{\gamma} D^{\hat{k}})$ simultaneously.

We remark that, by standard arguments, Theorem \ref{thm: lin op} implies the analogous theorem for $C^m(\R^n)$ when $E$ is a \emph{finite} subset of $\R^n$. Fefferman proved the analogue of Theorem \ref{thm: lin op} for $C^m(\R^n)$ when $E$ is compact; the argument is significantly more complicated (see \cite{F7}). It would be interesting to understand the norm of linear extension operators $T : C^m(E) \rightarrow C^m(\R^n)$ for $E$ compact.

We will now sketch the proof of Theorem \ref{thm: new c sharp}, highlighting the new ideas in the argument. Small modifications to this argument enable us to obtain the existence of a linear extension operator  $T : C^{m-1,1}(E) \rightarrow C^{m-1,1}(\R^n)$ with improved bounds on the operator norm, as in Theorem \ref{thm: lin op}.

By a compactness argument, it suffices to prove the finiteness theorem for a finite set $E$ in $\R^n$. Note that the constants $C^\#$ and $k^\#$ in the finiteness theorem are to be chosen independent of $E$. In the following, constants written $C$, $C^\#$, etc., are assumed to depend only on $m$ and $n$.  We write $\| \varphi \| = \| \varphi  \|_{C^{m-1,1}(\R^n)}$ for the $C^{m-1,1}$ seminorm of a function $\varphi \in C^{m-1,1}(\R^n)$.

Fix a finite set $E \subseteq \R^n$ and function $f : E \rightarrow \R$. We assume the data $(E,f)$ satisfies the hypotheses of the finiteness theorem; namely, we assume the following \emph{finiteness hypothesis} is valid:
\[
\mathcal{(FH)} \quad \left\{
\begin{aligned}
&\mbox{for any subset } S \subseteq E \mbox{ with } \#(S) \leq k^\# \\
&\mbox{there exists a } C^{m-1,1} \mbox{ function } F^S : \R^n \rightarrow \R \\
&\mbox{satisfying } F^S =f \mbox{ on } S \mbox{ and } \| F^S \| \leq 1.
\end{aligned}
\right.
\]
We assume $k^\#$ in the finiteness hypothesis is a sufficiently large constant determined by $m$ and $n$. 

To prove the finiteness theorem, we will construct an $F \in C^{m-1,1}(\R^n)$ satisfying $F = f$ on $E$ and $\| F \| \leq C^\#$ for a constant $C^\#$ determined by $m$ and $n$. That is, we will construct an interpolant $F \in C^{m-1,1}(\R^n)$ of $f$ with $C^{m-1,1}$-seminorm at most $C^\#$.

Let $\Po$ be the vector space of real-valued polynomials on $\R^n$ of degree $\le m-1$. Write $J_x(\varphi)$ to denote the $(m-1)^{\text{rst}}$ order Taylor polynomial at $x$ of a function $\varphi \in C^{m-1,1}(\R^n)$, defined by
\[
J_x(\varphi)(z) := \sum_{|\alpha| \leq m-1} ( \partial^\alpha \varphi(x)/\alpha!) (z-x)^\alpha.
\]
We call $J_x(\varphi) \in \Po$ the $(m-1)$-jet of $\varphi$ at $x$. We define a ring product $\odot_x$ on $\Po$ by defining $P \odot_x Q = J_x( PQ)$ for $P,Q \in \Po$. We write $\cR_x$ for the ring $(\Po,\odot_x)$.

Fefferman's papers on the Whitney extension problem (e.g., \cite{F3,F7,F8,F2,F1,F4}) introduce and make extensive use of a family of convex subsets $\sigma(x) \subseteq \cP$, indexed by $x \in E$. Informally, the set $\sigma(x)$ measures the freedom in choosing the $(m-1)$-jet $J_x(F)$ for an interpolant $F \in C^{m-1,1}(\R^n)$ of $f$. Let
\[
\sigma(x) := \{ J_x (\varphi) : \varphi|_E = 0, \; \| \varphi \| \leq 1\} \subseteq \Po.
\]
Note that if $J_x(F_1) = P_1$ and $J_x(F_2) = P_2$ for two different interpolants $F_1,F_2$ of $f$, and if $\| F_1 \| \leq M$ and $\| F_2 \| \leq M$ for some $M > 0$, then $P_1 - P_2$ belongs to $2M \sigma(x)$. Indeed, $\varphi := F_1 - F_2$ satisfies $\varphi|_E = 0 $ and $\| \varphi \| \leq 2M$; hence, $P_1 - P_2 = J_x ( \varphi) \in 2M \sigma(x)$. Thus, the (dilates) of $\sigma(x)$ can be used to control the freedom in the choice of $J_x(F)$ for an interpolant $F$ of $f$ on $E$ of bounded seminorm.

A key idea in Fefferman's proof of the finiteness theorem is to index an interpolation problem by a \emph{label}\footnote{A label is a multi-index set $\cA = \{ \alpha_1,\cdots, \alpha_L \}$ with each $\alpha_i$ a multiindex of order at most $m-1$.} $\mathcal{A}$ which records information on the ``large coordinate directions'' in the set $\sigma(x)$. Fefferman introduces an order relation $<$ on labels, which can be used to sort interpolation problems according to their ``difficulty''. By a divide and conquer approach, he decomposes an interpolation problem with a given label $\cA$ into a family of easier interpolation problems with smaller labels $\cA' < \cA$. The proof is organized as an induction on the label assigned to a given interpolation problem. For details, see \cite{F1}. 

In a joint work \cite{coordinateFree} with B. Klartag, we gave a coordinate-free proof of the finiteness theorem. To accomplish this we explained how to replace the notion of a label in Fefferman's inductive scheme by the notion of a \emph{DTI subspace}. We record information on the large directions in $\sigma(x)$ by specifying that a DTI subspace is \emph{transverse} to $\sigma(x)$. We mimic Fefferman's divide and conquer strategy. However, one crucial difference is that our proof is organized as an induction with respect to an integer-valued quantity called the \emph{complexity} of $E$. Roughly speaking, the complexity of $E$ measures how often the geometry of the set $\sigma(x)$ changes dramatically as one applies a rescaling transformation about a fixed point $x \in E$.

Let $V$ be a subspace of $\Po$. We say that $V$ is \emph{dilation-and-translation-invariant}, or \emph{DTI}, provided that (1) $V$ is \emph{dilation-invariant}, i.e., $P(\cdot / \delta) \in V$ for all $P \in V$, $\delta >0$ and (2) $V$ is \emph{translation-invariant}, i.e., $P(\cdot - h) \in V$ for all $P \in V$, $h \in \R^n$. These conditions on $V$ can be reformulated as follows: A subspace $V$ is dilation-invariant provided that $ V = \bigoplus_{i=0}^{m-1} V_i$, where $V_i \subseteq \Po_i := \mathrm{span} \{ x^\alpha : |\alpha| = i \} $ is a \emph{homogeneous subspace} of $\Po$, for $i=0,1,2,\dots,m-1$. Further, a subspace $V$ is translation-invariant if and only if the orthogonal complement $V^\perp$ of $V$ with respect to a natural inner product\footnote{This claim is valid, e.g., for the inner product $\langle P, Q \rangle' := \sum_{|\alpha| \leq m-1} \frac{1}{\alpha!} \partial^\alpha P(0) \partial^\alpha Q(0)$ for $P,Q \in \Po$; see Lemma 3.11 of \cite{coordinateFree}. We will make use of another inner product $\langle \cdot, \cdot \rangle$ on $\Po$ later in the paper.} on $\Po$ satisfies that $V^\perp$ is an ideal in the ring of $(m-1)$-jets $\cR_0 = (\Po, \odot_0)$ based at $x=0$. It follows that the DTI subspaces $V$ are orthogonal to those ideals $I$ in $\cR_0$ which admit a direct sum decomposition into homogeneous subspaces.

We assign a \emph{DTI label} $V$ to the set $E$ at position $x \in E$ and scale $\delta > 0$ provided that $V$ is a DTI subspace of $\Po$, while $\sigma(x)$ and $V$ satisfy a quantitative \emph{transversality condition} at $(x,\delta)$. Roughly speaking, the transversality condition states that the ``big directions'' in $\sigma(x)$ do not make a small angle with $V$, and the intersection $V \cap \sigma(x)$ is suitably small. Here, to make sense of angles, we equip the vector space $\Po$ with a suitable inner product $\langle \cdot, \cdot \rangle_{x,\delta}$. See Definition \ref{def:trans2} for the precise statement of the transversality condition.

We associate to a point $x \in E$ a sequence of DTI subspaces
\[
V_1, V_2, \dots, V_L
\]
and lengthscales
\[
\delta_1 > \delta_2 > \dots > \delta_L
\]
such that $V_\ell$ is a DTI label assigned to $E$ at position $x$ and scale $\delta_\ell$ $(\ell \leq L)$, and $V_\ell$ is not a DTI label assigned to $E$ at position $x$ and scale $\delta_{\ell+1} < \delta_\ell$ ($\ell < L$). We denote by $\mathcal{C}(E)$ the supremal length of any such sequence associated to any $x \in E$. By convention, $\cC(E) = 0$ if $E = \emptyset$. Borrowing notation from our earlier work \cite{coordinateFree}, we refer to the quantity $\mathcal{C}(E)$ as the \emph{complexity} of $E$. It is evident from the definition that complexity is locally monotone with respect to inclusion, in the sense that $\mathcal{C}(E \cap B) \geq \mathcal{C} (E \cap B')$ whenever $B' \subseteq B \subseteq \R^n$. To construct an extension $F$ of $f$ of bounded $C^{m-1,1}$ norm, we proceed by induction on $\mathcal{C}(E)$. 

The base case of the induction corresponds to the case $\mathcal{C}(E) = 0$. If $\cC(E) = 0$ it easily follows that $E$ is the empty set, whence it is trivially true that there exists an extension of $f$ on $E$ of bounded $C^{m-1,1}$ seminorm. 

For the induction step, we assume the induction hypothesis that the finiteness theorem is true for any data $(\widetilde{E},\widetilde{f})$ satisfying that $\mathcal{C}(\widetilde{E})  < L_0$ for fixed $L_0 \geq 1$. We then fix data $(E,f)$ satisfying the hypotheses of the finiteness theorem, with $\mathcal{C}(E) = L_0$. To complete the induction step we must construct an interpolant $F$ of $f$ with $\| F \| \leq C$. 

Fix a closed ball $B_0 \subseteq \R^n$ with $E \subseteq B_0$ and $\diam(B_0) = \diam(E)$. We define a cover of $B_0$ by a family $\cW$ of closed balls in $\mathbb{R}^n$; thus, $B_0 \subseteq \bigcup_{B \in \cW} B$. We construct the cover $\cW$ to have the following properties: First, $ \mathcal{C}(E \cap B) < \mathcal{C}(E) = L_0$ for all $B \in \cW$. On the other hand, $\mathcal{C}(E \cap 100B ) = \mathcal{C}(E) = L_0$ for all $B \in \cW$. Finally, the cover $\cW$ has \emph{good geometry} in the sense that for every $B \in \cW$ we have $B \cap  B' \neq \emptyset$ for at most $C$ balls $B' \in \cW$; also, if $B \cap B' \neq \emptyset$  for $B, B' \in \cW$ then $\diam(B)$ and $\diam(B')$ differ by a factor of at most $K$. Here, $C=C(n)$ and $K=K(n)$ are appropriate dimensional constants. 

Evidently, it is sufficient to construct an interpolant $F$ of $f$ on $B_0$, satisfying $\| F \|_{C^{m-1,1}(B_0)} \leq C$. For then, it is trivial to extend $F$ to all of $\R^n$, while not increasing the $C^{m-1,1}$-seminorm by more than a constant factor.

By the induction hypothesis applied to the set $\widetilde{E} = E \cap B$, for each $B \in \cW$ there exists a \emph{local interpolant} $F_B$ of $f$ on $E \cap B$ satisfying two conditions: (local interpolation) $F_B = f$ on $E \cap B$ and (bounded seminorm) $\| F_B \| \leq M$ for all $B \in \cW$. Here, $M$ will be a constant determined by $m$,$n$ and the induction index $L_0$. So $\{ F_B \}_{B \in \cW}$ is a family of local interpolants associated to the balls in the cover $\cW$. We define
\[
F = \sum_{B \in \cW} F_B \theta_B \mbox{ on } B_0,
\]
where $\{\theta_B\}_{B \in \cW}$ is a partition of unity on $B_0$ (thus, $\sum_B \theta_B = 1$ on $B_0$), while each $\theta_B$ is supported on $B$, $\theta_B \equiv 1$ near the center of $B$, and each partition function $\theta_B$ satisfies the derivative bounds $\| \partial^\alpha \theta_B \|_{L^\infty} \leq C \diam(B)^{-|\alpha|}$ for $|\alpha| \leq m$. Such a partition of unity is guaranteed to exist by the covering and good geometry properties of $\cW$. Evidently, since $F_B = f$ on $E \cap B$ for all $B \in \cW$, we have $F = f$ on $E$. We hope to prove that $\| F \|_{C^{m-1,1}(B_0)} \leq \widetilde{C} M$ for a constant $\widetilde{C}$ determined by $m$ and $n$. Unfortunately, there is no reason to expect this to be true, given that the $F_B$ were chosen independently of one another. By following the ideas in \cite{coordinateFree} (inspired by analogous ideas in \cite{F1}), we construct local interpolants $F_B$ which are compatible with one another -- to enforce these compatibility conditions, we modify by a small additive correction function the $F_B$ specified above. We now state the extra compatibility conditions on the $F_B$. First we establish the existence of a DTI subspace $V$ that is transverse to $\sigma(x)$ for each $x \in E$ at some scale $\delta > 0$. Then fix an appropriate jet $P_0 \in \cP$ (determined by the data $(f,E)$) and specify that $J_{x_B} F_B \in P_0 + V$ for every $B \in \cW$; here $x_B$ is a specified point of $B$. Essentially, the compatibility conditions state that $J_{x_B} F_B$ belongs to the same coset of $V$ for every $B \in \cW$. These are the extra conditions required of the local interpolants $F_B$, beyond those stated before.  For a family of local interpolants $F_B$ satisfying the aforementioned conditions, we can prove that $\| F \|_{C^{m-1,1}(B_0)} \leq \widetilde{C}(m,n)  \max_{B} \| F_B \| \leq \widetilde{C}(m,n) M$ for the $F$ defined before. Since $F$ is an interpolant of $f$, this completes the induction step. As a final remark, we note that to carry out the above modification step and prove the existence of local solutions $F_B$ satisfying the extra compatibility conditions, it is required to bring in the finiteness hypothesis $\mathcal{(FH)}$ and  certain convex sets $\Gamma_\ell(x,f,M)$ (these being sometimes referred to as $\mathcal{K}_f(x;k,M)$ in Fefferman's work). We spare the details in this sketch.
 
Thus we have shown, by induction on $\mathcal{C}(E)$, that there exists an extension $F$ of $f$ with norm at most $\widetilde{C}^{\mathcal{C}(E)}$, where $\widetilde{C}$ is a fixed constant determined by $m$ and $n$. To see this, note that the bound on the norm of the extension $F$ increases by a factor of $\widetilde{C}=\widetilde{C}(m,n)$ at each step of the induction proof.

To conclude the proof of the finiteness theorem, we must demonstrate that the complexity $\mathcal{C}(E)$ is bounded uniformly for all finite subsets $E \subseteq \R^n$. We define the worst-case complexity $L_{\max}$  by
\[
L_{\max}:= \sup_{E\subseteq \R^n} \mathcal{C}(E),
\]
where the supremum is over finite sets $E \subseteq \R^n$. In \cite{coordinateFree}, we demonstrated that $L_{\max}$ is bounded by a constant $C(D)$ determined by $D = \binom{n+m-1}{n}$. Our proof used semialgebraic geometry, resulting in poor dependence $C(D) \gtrsim \exp(\exp(D))$. Also in \cite{coordinateFree}, we conjectured that
\begin{equation}\label{eq:L_max}
    L_{\max} \lesssim \poly(D).
\end{equation}
The first main technical result of this paper, Proposition \ref{prop:tech2}, establishes the conjecture \eqref{eq:L_max}. More specifically, in Section \ref{sec:rescale_grass}, we prove that $L_{\max} \leq 4 m D^2$.

 By our discussion above, we can construct an extension $F$ of $f: E \rightarrow \R$ with $||F||_{C^{m-1,1}(\R^n)}\le \widetilde{C}^{L_{\max}}$ for any finite set $E \subseteq \R^n$. Combining this with \eqref{eq:L_max} gives $||F||_{C^{m-1,1}(\R^n)}\le \widetilde{C}^{\poly(D)}$. Therefore to establish Theorem \ref{thm: new c sharp} it just remains to show that
\begin{equation}\label{eq:Ctilde}
    \widetilde{C} \lesssim \exp(\poly(D)).
\end{equation}
Indeed, \eqref{eq:Ctilde} follows from a careful bookkeeping of various constants appearing in the proof, and our second main technical result, Proposition \ref{prop:tech1}, which we prove in Section \ref{sec:WhConv}.

This completes our sketch of the proof of Theorem \ref{thm: new c sharp}.

To establish Theorem \ref{thm: lin op}, we show that our construction can be modified so that, for a fixed set $E$, the extension $F$ depends linearly on the data $f$. 

We finish the introduction by describing the content of Sections \ref{sec:statement_mt}-\ref{sec_ind} in more detail.

Section \ref{sec:statement_mt} contains the statement of our main extension theorem for finite sets $E \subseteq \R^n$.

Section \ref{sec:basicconv} contains the definitions of the convex sets $\sigma(x)$ and their variants, and gives results on the basic properties of these sets.

Section \ref{sec:MLS} contains additional technical results (many borrowed from \cite{F2}) needed for the proof of Theorem \ref{thm: lin op}.

Sections \ref{sec:localMain}--\ref{sec_ind} contain the main analytic ingredients of the paper, including the Main Decomposition Lemma (Lemma \ref{mdl}), which is the apparatus used to decompose the extension problem for $(E,f)$ into easier subproblems.

Finally, Section \ref{sec:mainproofs} contains the proof of the extension theorem for finite $E$, and the proofs of the theorems from the introduction (Theorems  \ref{thm: new c sharp} and \ref{thm: lin op}).

The notation and terminology in the previous discussion is not necessarily used in the rest of the paper. This discussion captures the spirit of the proof of our theorems, but some of the definitions given above are simplified for ease of explanation. In particular, the phrase ``DTI label'' does not appear in the remainder of the paper, nor in our earlier work \cite{coordinateFree}. Furthermore, the definition of complexity and the description of the properties of the cover $\cW$ are presented somewhat differently than in the main body of the paper -- for instance, certain technical constants have been obscured in the above discussion to simplify the exposition.

\subsubsection{Acknowledgements}

We are grateful to the participants of the 14th Whitney Problems Workshop for their interest in this work. We are particularly grateful to Charles Fefferman and Bo'az Klartag, for providing valuable comments on an early draft of this paper. We are also grateful to the National Science Foundation and the Air Force Office of Scientific Research for their generous financial support.\footnote{The first-named author acknowledges the support of AFOSR grant FA9550-19-1-0005. The third-named author acknowledges the support of NSF grant DMS-1700404 and AFOSR grant FA9550-19-1-0005.} Last, we would like to thank the anonymous referee, whose feedback led to improvements in the paper.

\section{Notation and preliminaries}\label{sec:prelim}

Fix $m\geq 2$, $n \geq 1$ throughout the paper.  Let $D := \binom{m+n-1}{n}$.

We write $B(x,r) = \{ z \in \R^n : |z-x| \leq r\}$ for the closed ball of radius $r$ and center $x$ in $\R^n$.

Given a ball $B \subseteq \R^n$ and $\lambda > 0$, let $\lambda B$ denote the ball with the same center as $B$ and radius equal to $\lambda$ times the radius of $B$.

For any finite set $S$, write $\#(S)$ to denote the number of elements of $S$. If $S$ is infinite, we put $\#(S) = \infty$.

Let $\cM := \{ \alpha = (\alpha_1,\alpha_2,\dots,\alpha_n) : |\alpha| = \alpha_1+\alpha_2+\dots+\alpha_n \leq m-1\}$ be the set of all multiindices of length $n$ and order at most $m-1$.  Then $\#(\cM) = D$.

\subsection{Convention on constants}

By an ``absolute constant'' we mean a numerical constant whose value is independent of $m$ and $n$.

Given quantities $A,B \geq 0$, we write $A = O(B)$ to indicate that $A \leq \gamma B$ for an absolute constant $\gamma > 0$. We write $\poly(x)$ to denote a polynomial $\poly(x) = \sum_{k=0}^d a_k x^k$ with coefficients $a_k$ and maximum degree $d$ given by absolute constants. Similarly, we write $\poly(x,y)$ to denote a polynomial in two variables with  coefficients and maximum degree given by absolute constants.

We say that $C > 0$ is a \emph{controlled constant} if $C$ depends only on $m$, $n$ and both $1/C$ and $C$ are $O(\exp(\poly(D)))$. Note that the product of $O(\poly(D))$ many controlled constants is again a controlled constant. 

Provided $m \geq 2$, the binomial coefficient $D = \binom{m+n-1}{n}$ satisfies $\max\{m,n\} \leq D$. So, if both $C$ and $1/C$ are $O(\exp(\poly(m,n)))$ then $C$ is a controlled constant. 

We say that two quantities $X,Y \geq 0$ are \emph{equivalent up to a controlled constant} if $C^{-1} Y \le X \le CY$ for a controlled constant $C$.

\subsection{Function spaces $C^{m-1,1}$ and $\dot{C}^m$}

Let $G\subseteq \R^n$ be a convex domain with nonempty interior. We write $C^{m-1,1}(G)$ to denote the space of all $(m-1)$-times differentiable  functions $F:G\rightarrow \R$ whose $(m-1)$-st order partial derivatives are Lipschitz continuous on $G$, equipped with the seminorm
\begin{equation}\label{eqn:norm_defn}
    \|F\|_{C^{m-1,1}(G)} := \sup_{x,y\in G}\left( \sum_{|\alpha|=m-1} \frac{(\partial^\alpha F(x) - \partial^\alpha F(y))^2}{|x-y|^2}\right)^{1/2}.
\end{equation}

We define the space $\dot{C}^m(G)$ to consist of all $m$-times continuously differentiable functions $F : G \rightarrow \R$ whose $m$-th order partial derivatives are uniformly bounded on $G$, equipped with the seminorm
\begin{equation}\label{eqn:Cm_norm_defn}
    \|F\|_{\dot{C}^{m}(G)} := \sup_{z\in G}\max_{|\beta|=m} |\partial^\beta F(z)|.
\end{equation}

Let $F \in \dot{C}^m(G)$. Given a multiindex $\alpha$ with $|\alpha| = m-1$, the Mean Value Theorem implies that the difference quotient $| \partial^\alpha F(x) - \partial^\alpha F(y)|/|x-y|$ is bounded by $\sup_{z \in [x,y]} |\nabla \partial^\alpha F(z)|$, where $[x,y]$ is the line segment connecting $x$ and $y$ (contained in $G$). The latter quantity is bounded by $\sqrt{n} \cdot \| F \|_{\dot{C}^m(G)}$. Therefore, if  $F \in \dot{C}^m(G)$ then $F \in C^{m-1,1}(G)$ and
\begin{equation}\label{eqn:Cm_norm_bd}
\| F \|_{C^{m-1,1}(G)} \leq C \| F \|_{\dot{C}^m(G)},
\end{equation}
for a controlled constant $C$. 

We write $C^{m-1}_{loc}(\R^n)$ to denote the space of all functions $F:\R^n \rightarrow \R$ such that $F\in \dot{C}^{m-1}(B(0,R))$ for any $ R> 0$. 

\subsection{Jet space}

Let $\Po$ denote the vector space of all polynomials on $\R^n$ of degree at most $m-1$. Then $\Po$ admits a basis of monomials, $\mathfrak{V}_x := \{ m_{\alpha,x}(z) := (z-x)^\alpha : \alpha \in \cM \}$ for any $x \in \R^n$. In particular, $\dim(\Po) = \#(\cM) = D$.

Given $x \in \R^n$ and $F \in C^{m-1}_{loc}(\R^n)$, let $J_x(F) \in \Po$  denote the $(m-1)$-jet of $F$ at $x$, given by
\[
J_x(F)(z) := \sum_{|\alpha| \leq m-1} ( \partial^\alpha F(x)/\alpha!) \cdot (z-x)^\alpha.
\]
We endow $\Po$ with a product $\odot_{x}$ (``jet multiplication at $x$'') defined by $P \odot_{x} Q = J_{x}(P\cdot Q)$ for $P,Q \in \Po$. We write $\cR_x$ to denote the ring $(\Po,\odot_x)$ of $(m-1)$-jets at $x$. We write $\odot = \odot_0$ for the jet product at $x=0$.

Note that if $F,G \in C^{m-1}_{loc}(\R^n)$ then $J_x(F \cdot G) = J_x(F) \odot_x J_x(G)$. That is, $J_x : C^{m-1}_{loc}(\R^n) \rightarrow \cR_x$ is a ring isomorphism.

We often use the notation $\Po$ and $\cR_x$  interchangeably. We shall use $\Po$ when the ring structure of the jet space is irrelevant to the intended application.


\subsubsection{Translations and dilations}

The jet space $\Po$ inherits the structure of translations and dilations from $\R^n$. Specifically, we let $\tau^h : \Po \rightarrow \Po$ ($h \in \R^n$) and $\tau_{x,\delta} : \Po \rightarrow \Po$ ($x \in \R^n$, $\delta > 0$) be translation and dilation operators defined by 
\begin{equation}\label{eqn:trans_dil}
\begin{aligned}
&\tau^h(P)(z) := P(z-h), \mbox{ and}\\ &\tau_{x,\delta}(P)(z) := \delta^{-m} P(x + \delta \cdot (z-x)) \qquad (P \in \Po).
\end{aligned}
\end{equation}

\subsubsection{Inner products and norms}\label{sec:inner_prod_norm}
Let $x \in \R^n$. We define the inner product $\langle P,Q\rangle_x$ of $P,Q \in \Po$ by
\[
\langle P,Q \rangle_x := \sum_{|\alpha| \leq m-1}  \partial^\alpha P(x) \partial^\alpha Q(x) / (\alpha !)^2.
\]
The corresponding norm $|P|_x$ of $P \in \Po$ is given by
\[
|P|_x := \sqrt{\langle P, P \rangle_x} = \sqrt{\sum_{|\alpha| \leq m-1 } (\partial^\alpha P(x))^2/(\alpha!)^2}.
\]
The purpose of the $1/(\alpha!)^2$ factor in the above expressions is to ensure the monomials $m_{\alpha,x}(z) := (z-x)^\alpha$ have unit length, i.e., $|m_{\alpha,x} |_x = 1$ for $|\alpha| \leq m-1$.

For $x \in \R^n$, $\delta > 0$, we define the \emph{scaled inner product} $\langle P, Q \rangle_{x,\delta}$ of $P,Q \in \Po$ by
\[
\begin{aligned}
\langle P, Q \rangle_{x,\delta} &:= \langle \tau_{x,\delta} (P), \tau_{x,\delta}
(Q) \rangle_x \\
&=  \sum_{|\alpha| \leq m-1} \frac{1}{(\alpha!)^2} \delta^{2 (|\alpha| - m)} \partial^\alpha P(x) \cdot \partial^\alpha Q(x).
\end{aligned}
\]
The associated \emph{scaled norm} $|P|_{x,\delta}$ of $P \in \Po$ is
\[
|P|_{x,\delta} := \sqrt{ \langle P, P \rangle_{x,\delta}} = \biggl( \sum_{|\alpha| \leq m-1} \frac{1}{(\alpha !)^2} \cdot ( \delta^{|\alpha| - m} \cdot \partial^{\alpha} P(x) )^2 \biggr)^{\frac{1}{2}}.
\]
The closed unit ball for the scaled norm $| \cdot |_{x,\delta}$ is denoted by
\[ 
\B_{x,\delta} := \biggl\{ P :  \; |P|_{x,\delta}  \leq 1 \biggr\} \subseteq  \Po.
\]

For fixed $x$ the monomial basis $\mathfrak{V}_x := \{ m_{\alpha,x}: |\alpha| \leq m-1\}$ is orthogonal in $\Po$, with respect to the scaled inner product $\langle \cdot, \cdot \rangle_{x,\delta}$ for any $\delta > 0$. The monomial basis $\mathfrak{V}_x$ is orthonormal in $\Po$ only for the inner product $\langle \cdot, \cdot \rangle_x = \langle \cdot, \cdot \rangle_{x,1}$.

For any $\delta \ge \rho >0$, and $P \in \cP$,
\begin{equation}\label{eqn:norm_scale}
    \left(\frac{\rho}{\delta}\right)^m\cdot|P|_{x,\rho}\le|P|_{x,\delta}\le \left(\frac{\rho}{\delta}\right)\cdot|P|_{x,\rho}.
\end{equation}
Therefore,
\begin{equation}\label{eqn:ball_scale}
   \left(\frac{\delta}{\rho}\right) \cB_{x,\rho} \subseteq \cB_{x,\delta} \subseteq \left(\frac{\delta}{\rho}\right)^m \cB_{x,\rho}.
\end{equation}
In particular, 
\begin{equation}\label{eqn:ball_inc}
| P |_{x,\delta} \leq |P|_{x,\rho}, \; \mbox{and} \; \cB_{x,\rho} \subseteq \cB_{x,\delta} \mbox{ for } \delta \geq \rho > 0.
\end{equation}
Observe that $|P|_{x,\delta} = |\tau_{x,\delta} P |_x$ for $P \in \Po$. It follows that
\begin{equation}\label{eqn:ball_scale_id}
\tau_{x,r} \B_{x,\delta} = \B_{x,\delta/r}.
\end{equation}

Note that $\langle \cdot, \cdot \rangle_{x,1} = \langle \cdot,\cdot \rangle_{x}$ and $| \cdot |_{x,1} = | \cdot |_x$ for $x \in \R^n$. When $x=0$, we write $\langle P,Q \rangle = \langle P,Q \rangle_{0,1}$ and $|P| = |P|_{0,1}$ for the \emph{standard inner product and norm} on $\Po$. Write $\B = \B_{0,1}$ to denote the closed unit ball for the standard norm on $\Po$. 

Unless stated otherwise, we equip $\Po$ by default with the standard norm and inner product.

We write $\Po_i = \mathrm{span} \{x^\alpha : |\alpha| = i \} \subseteq \Po$ to denote the subspace of homogeneous polynomials of degree $i$.

We require bounds on the norm of a product of polynomials. These bounds are sometimes referred to in the literature as Bombieri inequalities. Recall that $\odot$ is the jet product at $x=0$.
\begin{lemma}\label{lem:bombieri}
Let $C_b := (m+1)!$. Then
\begin{align}\label{eqn:bombieri1}
   &|P \odot Q | \le C_b |P| \cdot |Q| \quad (P,Q \in \Po) \\
   \label{eqn:bombieri2}
   &|P \odot Q | \geq C_b^{-1} |P|\cdot |Q| \quad (P \in \Po_i, Q \in \Po_j, \; i+j < m).
\end{align}
\end{lemma}
\begin{proof}

We use two inequalities from \cite{Beauzamy1990}, stated below in \eqref{eqn:bomb_norm}. Our standard norm on $\Po$ is given by $|P| = \sqrt{ \sum c_\alpha^2}$ if $P = \sum c_\alpha x^\alpha$. In  \cite{Beauzamy1990} this is called the \emph{$2$-norm} and denoted by $|P|_2$. From \cite{Beauzamy1990} (see Proposition 1.B.3 and Theorem 1.1),  the following holds: If  $P \in \Po_i$ and $Q \in \Po_j$ for $i + j < m$, then
\begin{equation}\label{eqn:bomb_norm}
((i+j)!)^{-1/2} |P| |Q| \leq |P \cdot Q| \leq 2^{(i+j)/2} |P| |Q|.
\end{equation}

Note that $P \cdot Q = P \odot Q$ if $P \in \Po_i$ and $Q \in \Po_j$ for $i + j < m$; else, if $P \in \Po_i$ and $Q \in \Po_j$ for $i + j \geq m$ then $P \odot Q = 0$.  Therefore, the left-hand inequality in \eqref{eqn:bomb_norm} implies \eqref{eqn:bombieri2}.

Now let  $P, Q \in \Po$. Write $P = \sum_{i < m} P_i$ and $Q = \sum_{i<m} Q_i$ for $P_i,Q_i \in \Po_i$. Then $|P| = \sqrt{\sum |P_i|^2}$ and $|Q| = \sqrt{\sum |Q_i|^2}$ by orthogonality of the homogeneous subspaces $\Po_i$. Also, $P \odot Q = \sum_{i + j < m} P_i \cdot Q_j$. By the triangle inequality, and the right-hand inequality in \eqref{eqn:bomb_norm},
\[
|P \odot Q| \leq \sum_{i + j < m} |P_i \cdot Q_j| \leq 2^{m/2} \cdot \sum_{i + j < m} |P_i | \cdot | Q_j| \leq 2^{m/2} \left( \sum_{i < m} |P_i| \right) \cdot \left( \sum_{j < m} |Q_j| \right).
\]
By Cauchy-Schwartz, $\sum_{i < m} |P_i|  \leq \sqrt{m} \sqrt{\sum_{i<m} |P_i|^2}$, and similarly for the $Q_i$. Hence, 
\begin{equation}\label{eqn:bomb_norm2}
|P \odot Q| \leq m 2^{m/2} \sqrt{\sum_{i<m} |P_i|^2} \sqrt{\sum_{i<m} |Q_i|^2} = m 2^{m/2} |P| |Q|.
\end{equation}
Observe that $m 2^{m/2} \leq (m+1) !$. Thus, \eqref{eqn:bomb_norm2} implies \eqref{eqn:bombieri1}.
\end{proof}

\begin{proposition}[Taylor's Theorem]\label{lem:TT}
Let $G$ be a convex domain with nonempty interior. There exists a controlled constant $C_T \geq 1$ such that, for all $F \in C^{m-1,1}(G)$, $x,y\in G$, and $\delta \ge |x-y|$,
\begin{equation}\label{eqn:Taylor}
    |J_x F - J_yF|_{x,\delta} \le C_T \|F\|_{C^{m-1,1}(G)}.
\end{equation}
\end{proposition}
\begin{proof}
Taylor's theorem implies that if  $F\in C^{m-1,1}(G)$, $x,y \in G$, and $|\beta| \leq m -1 $ then
\[
|\partial^\beta (J_x F- J_yF)(x)|\le C \cdot \| F \|_{C^{m-1,1}(G)}  \cdot |x-y|^{m-|\beta|}.
\]
for a controlled constant $C$. Thus, for $\delta \ge |x-y|$, we obtain:
\[
\delta^{|\beta| - m} |\partial^\beta (J_x F- J_yF)(x)|\le C \cdot \|F\|_{C^{m-1,1}(G)}.
\]
Now square both sides of the above inequality, divide by $(\beta!)^2$, sum over $\beta$ with $|\beta| \leq m-1$, and take the square root, to obtain \eqref{eqn:Taylor}.
\end{proof}

\subsubsection{The Classical Whitney Extension Theorem}

We make use of the classical Whitney Extension Theorem for $(m-1)$-jets. We state the result here in a convenient form for later use.

Let $E \subseteq \R^n$. Suppose we are given a family of polynomials $P_x \in \Po$, indexed by $x \in E$.
We use the notation $P_\bullet : E \to \cP$ to denote the polynomial-valued map  $P_\bullet: x \mapsto P_x$. We refer to $P_\bullet$ as a \emph{Whitney field} on $E$. Endow the space of Whitney fields with a seminorm $\|P_\bullet\|_{\cP(E)} := \sup\{|P_x - P_y|_{x, |x-y|} : x, y \in E, \; x \neq y\}$. We let $\cP(E) := \{ P_{\bullet} : E \rightarrow \Po : \| P_\bullet \|_{\cP(E)} < \infty \}$.

\begin{proposition}[Classical Whitney Extension Theorem]\label{prop:cwet}
There exists a linear map $T: \cP(E) \to C^{m-1 ,1}(\R^n)$ such that $\|T ( P_\bullet)\|_{C^{m-1,1}(\R^n)} \leq C_{Wh} \|P_\bullet\|_{\Po(E)}$, and $J_x T (P_\bullet) = P_x$ for all $x \in E$, and all $P_\bullet \in \Po(E)$. Here,  $C_{Wh}$ is a controlled constant.
\end{proposition}
We refer the reader to \cite{Chang2017}, where it is proven that the classical Whitney extension theorem holds with the constant $C_{Wh} = C_m n^{5m/2}$, for a constant $C_m$ determined by $m$. The proof in \cite{Chang2017} does not give an explicit bound on $C_m$, but by inspection of the proof one can see that $C_m$ is a polynomial function of $m!$. Therefore, $C_{Wh}$ is controlled.

We now state an elementary consequence of the Whitney extension theorem: We can extend a $C^{m-1,1}$ function on a convex domain $G \subseteq \R^n$ to all of $\R^n$, with control on the $C^{m-1,1}$ seminorm of the extension.

\begin{lemma}\label{lem:dom_ext}Let $G$ be a convex domain in $\R^n$ with nonempty interior. Let $F \in C^{m-1,1}(G)$. Then there exists a function $\widehat{F} \in C^{m-1,1}(\R^n)$ with $\widehat{F}|_G = F$ and $\| \widehat{F} \|_{C^{m-1,1}(\R^n)} \leq C \| F \|_{C^{m-1,1}(G)}$, for a controlled constant $C \geq 1$. Furthermore, $\widehat{F}$ can be taken to depend linearly on $F$.
\end{lemma}
\begin{proof}
Given $F \in C^{m-1,1}(G)$, define a Whitney field $P_\bullet \in \Po(G)$ by $P_x = J_x F$ for $x \in G$ (note that $J_xF$ is well-defined for $x \in G$ by the hypothesis that $G$ has nonempty interior). By Taylor's theorem (Proposition \ref{lem:TT}),  $\| P_\bullet \|_{\Po(G)} \leq C_T \| F \|_{C^{m-1,1}(G)}$. Let $T: \cP(G) \to C^{m-1 ,1}(\R^n)$ be as in the classical Whitney extension theorem, and set $\widehat{F} := T(P_\bullet)$. Then $\widehat{F}$ depends linearly on $F$. Because $J_x \widehat{F} = P_x = J_x F$ for all $x \in G$, we have $\widehat{F}|_G = F$. Furthermore,
\[
\| \widehat{F} \|_{C^{m-1,1}(\R^n)} \leq C_{Wh} \| P_\bullet \|_{\Po(G)} \leq C_{Wh} C_T \| F \|_{C^{m-1,1}(G)}.
\]
This completes the proof of the lemma with $C = C_T C_{Wh}$.
\end{proof}

\subsubsection{Graded decomposition of the jet space}

Given $x \in \R^n$, the jet space $\cR_x \simeq \Po$ admits a \emph{graded decomposition} into homogeneous vector subspaces. Specifically, 
\[
\cR_x = \bigoplus_{i=0}^{m-1} \cR_{x}^i,  \quad \mbox{where} \; \cR_x^{i} := \spn \{ m_{x,\alpha}(z) := (z-x)^\alpha : |\alpha| = i \}.
\]
Note that $\tau_{x,\delta}(P) = \delta^{i - m} P$ for $P \in \cR_x^i$ -- thus, $\cR_x^i$ is homogeneous of order $i - m$ with respect to the dilations $\tau_{x,\delta}$ ($\delta > 0$). The subspaces $\cR_{x}^i$ are pairwise orthogonal with respect to the inner product $\langle \cdot, \cdot \rangle_{x,\delta}$ (any $\delta > 0$). Furthermore, $\spn(\cR_i^x \odot_x \cR_j^x) = \cR_{i+j}^x$ if $i + j < m$, and $\cR_i^x \odot_x \cR_j^x = \{ 0 \}$ if $i + j \geq m$.

\subsubsection{Dilation and translation invariant subspaces}

Let $V$ be a subspace of $\Po$. We say $V$ is \emph{translation invariant} if $\tau^h(P) \in V$ for all $P \in V$, $h \in \R^n$.  Let $x_0 \in \R^n$. We say $V$ is \emph{dilation invariant at $x_0$} if $\tau_{x_0,\delta}(P) \in V$ for all $P \in V$, $\delta > 0$.  For the definitions of the translations $\tau^h$ and dilations $\tau_{x,\delta}$, see \eqref{eqn:trans_dil}.

Note that $V$ is dilation invariant at $x_0$ if and only if $V$ admits a decomposition
\[
V = \bigoplus_{i=0}^{m-1} V_i^{x_0},
\]
for subspaces $V_i^{x_0} \subseteq \cR_{x_0}^i$ ($0 \leq i \leq m-1$). 

We say $V$ is \emph{DTI} (\emph{dilation-and-translation-invariant}) if $V$ is both translation invariant and dilation invariant at $x_0$ for some $x_0 \in \R^n$. If $V$ is DTI then $V$ is dilation invariant at $x$ for all $x \in \R^n$, due to the identity $\tau_{x,\delta} = \tau^{x-x_0} \tau_{x_0,\delta} \tau^{x_0-x}$.

A special class of DTI subspaces arises by looking at the span of monomials in $\Po$. Given $\cA \subseteq \cM$, let $V_\cA := \spn \{ x^\alpha : \alpha \in \cA\}$. 

\begin{defi}\label{def:mon}
A set $\cA \subseteq \cM$ is  \emph{monotonic} provided that if  $\alpha \in \cA$, $\beta \in \cM$, and $\alpha + \beta \in \cM$, then $\alpha + \beta \in \cA$. 
\end{defi}

\begin{lemma}\label{lem:mon}
 Let $\cA \subseteq \cM$. Then the following are equivalent:
 \begin{itemize}
     \item[(i)] $\cA$ is monotonic.
     \item[(ii)] $V_\cA$ is an ideal in the ring $\cR_0 = (\Po, \odot)$.
     \item[(iii)]  $V_{\cM \setminus \cA}$ is a DTI subspace.
 \end{itemize}
\end{lemma}
\begin{proof}
Recall that $\odot = \odot_0$ is the ``jet product at $x=0$''. Note that $V_{\cA}$ is an ideal in $\cR_0$ if and only if $x^\beta \odot P \in V_{\cA}$ for every polynomial $P$ in a basis for $V_{\cA}$ and every $\beta \in \cM$. Thus, $V_{\cA}$ is an ideal if and only if $x^\beta \odot x^\alpha \in V_{\cA}$ for all $\beta \in \cM$ and $\alpha \in \cA$. Observe that $x^\beta \odot x^\alpha = 0$ if $|\beta| + |\alpha| \geq m$, and else, $x^\beta \odot x^\alpha = x^{\alpha + \beta}$ if $|\alpha| + |\beta| \leq m-1$. Thus, $V_{\cA}$ is an ideal if and only if $\alpha + \beta \in \cA$ whenever $\alpha \in \cA$, $\beta \in \cM$, $|\alpha| + |\beta| \leq m-1$. Therefore, $V_\cA$ is an ideal if and only if $\cA$ is monotonic, establishing the equivalence of (i) and (ii).

It remains to establish the equivalence of (i) and (iii). Evidently,  $V = V_{\cM\setminus \cA}$ is dilation invariant at $x_0 = 0$ due to the fact that $V$ is spanned by monomials based at $x_0 = 0$. Therefore it suffices to show that $V$ is translation invariant if and only if $\cA$ is monotonic. 

Suppose $\cA$ is monotonic. By linearity it suffices to show that $\tau^h P \in V$ for any element $P$ in the basis $ \{x^\gamma \}_{\gamma \in \cM \setminus \cA}$ for $V$. Fix $\gamma \in \cM \setminus \cA$ and $h \in \R^n$, and use the binomial identity to write
\[
\tau^h [ x^\gamma] = (x-h)^\gamma = \sum_{\substack{ \gamma_1,\gamma_2 \in \cM \\  \gamma_1 + \gamma_2 = \gamma}} c_{\gamma_1 \gamma_2} x^{\gamma_1} h^{\gamma_2}.
\]
Since $\cA$ is monotonic, and $\gamma \in  \cM \setminus \cA$, we have $\gamma_1 \in \cM \setminus \cA$ if $\gamma = \gamma_1 + \gamma_2$. Consequently, each term $c_{\gamma_1 \gamma_2} x^{\gamma_1} h^{\gamma_2}$ in the above sum belongs to $V$. By linearity, $\tau^h [ x^\gamma] \in V$ for any $\gamma \in \cM \setminus \cA$. Thus, $V$ is translation invariant.

Next we suppose $V$ is translation invariant and show that $\cA$ is monotonic. Note the identity $\partial_{x_i} P = \lim_{h \rightarrow 0} h^{-1}( P - \tau^{h e_i}(P))$ where $e_i \in \R^n$ is the $i$'th coordinate vector. Because $V$ is translation invariant, this identity implies that $\partial_{x_i} P \in V$ for any $P \in V$. Therefore, $\partial^\beta P \in V$ for $P \in V$ and any multiindex $\beta$. For sake of contradiction suppose that $\cA$ is not monotonic. Then there exist $\alpha \in \cA$, $\beta \in \cM$  with $\alpha + \beta \in \cM \setminus \cA$. Thus, $x^{\alpha + \beta} \in V$. Consequently,  $\partial^\beta x^{\alpha + \beta} \in V$. Note that $\partial^\beta x^{\alpha + \beta} = c x^\alpha$ for $c \in \R$, $c \neq 0$. Thus, $x^\alpha \in V$, implying that $\alpha \in \cM \setminus \cA$, a contradiction.

This completes the proof of the lemma.

\end{proof}

\subsubsection{Whitney convexity}

A subset $\Omega$ of a vector space is \emph{symmetric} provided that $v \in \Omega \implies - v \in \Omega$.

Given $x \in \R^n$, we denote $X \odot_x Y := \{ P \odot_x Q : P \in X, \; Q \in Y\}$ for subsets $X,Y \subseteq \cR_x$. 

The next definition plays a key role in the theory of $C^{m-1,1}$ extension.

\begin{defi}[Whitney convexity] \label{wc_def} Let $x \in \R^n$, and let  $\Omega \subseteq \cR_x$ be a closed symmetric convex set. We say that $\Omega$ is $A$-Whitney convex at $x$ if $(\Omega \cap \B_{x,\delta}) \odot_x  \B_{x,\delta} \subseteq A \delta^m \Omega$ for all $\delta > 0$. If $\Omega$ is $A$-Whitney convex at $x$ for some $A < \infty$, then we say that $\Omega$ is Whitney convex at $x$.

The Whitney coefficient $w_x(\Omega)$ of $\Omega$ at $x$ is the infimum of all $A > 0$ such that $\Omega$ is $A$-Whitney convex at $x$. If no finite $A$ exists, then $w_x(\Omega) := + \infty$. 
\end{defi}

\subsection{Main technical results}

Here, we state the new technical results of this paper. The second result will be used to affirm a conjecture from the introduction of \cite{coordinateFree}. Sections  \ref{sec:GeomGrass}, \ref{sec:rescale_grass} and \ref{sec:WhConv} are dedicated to the proofs of these results,

Fix $x \in \R^n$. We equip the jet space $\cR_x = (\Po,\odot_x)$ with the inner product $\langle \cdot, \cdot \rangle_x$ and norm $|\cdot|_x$; see Section \ref{sec:inner_prod_norm}. Then $\cR_x$ is a finite-dimensional Hilbert space, with $\dim(\cR_x) = D = \binom{m+n-1}{n}$. Let $\cB_x$ be the unit ball of $\cR_x$. We let $\Pi_V: \cR_x \rightarrow V$ denote the orthogonal projection map on a subspace $V \subseteq \cR_x$.

\begin{defi}\label{defn:trans1}
Let $V$ be a subspace of $\cR_x$, let $\Omega$ be a closed symmetric convex subset of $\cR_x$, and let $R \geq 1$. Say that $\Omega$ is $R$-transverse to $V$ at $x$ if $\Omega \cap V \subseteq R \cB_x$ and $\Pi_{V^\perp}(\Omega \cap \cB_x) \supseteq R^{-1} \cB_x \cap V^\perp$. Here, $V^\perp$ is the orthogonal complement of $V$ with respect to the inner product $\langle \cdot, \cdot \rangle_x$ on $\cR_x$.
\end{defi}
Obviously, we can state a corresponding definition of transversalty in a general finite-dimensional Hilbert space. We do so in  Definition \ref{def:trans_hilbert}.

We now note a couple of trivial properties of $R$-transversality for the unfamiliar reader.
\begin{itemize}
    \item If $\Omega$ is $R$-transverse to $V$, then $\Omega$ is $R'$-transverse to $V$ for any $R' \ge R$.
    \item If $\Omega = V^\perp$, then $\Omega$ is $R$-transverse to $V$ for any $R \ge 1$.
\end{itemize}

Our first technical result is as follows:

\begin{proposition}\label{prop:tech1}
Let $x \in \R^n$, $A \geq 1$, and $\Omega \subseteq \cR_x$ be given. Suppose that $\Omega$ is $A$-Whitney convex at $x$. Then there is a DTI subspace $V \subseteq \cR_x$ such that $\Omega$ is $R_0$-transverse to $V$ at $x$. Here, $R_0$ is a constant determined by $m$, $n$, and $A$ of the form $R_0 = \exp(\poly(D) \log(A))$.
\end{proposition}

We write $l(I) \leq r(I)$ to denote the left and right endpoints of a compact interval $I \subseteq \R$, respectively. If $I$ and $J$ are compact intervals, we write $I>J$ if $l(I) > r(J)$. We write $I>0$ if $l(I) > 0$. 

\begin{defi}\label{jet_complexity_def}
    Let $x \in \R^n$. Given a closed symmetric convex set $\Omega \subseteq \cR_x$, $\delta > 0$, and real numbers $1 < R < R^* < \infty$, we define the quantity $\C_x(\Omega,R,R^*,\delta)$ to be the supremum of all integers $K$ such that there exist subspaces $V_k \subseteq \cR_x$ and compact intervals $I_k \subseteq (0,\delta]$ ($k=1,2,\dots,K$) such that the following conditions hold:
    \begin{itemize}
        \item $I_1 > I_2 > I_3 > \cdots > I_K > 0$
        \item For all $k$, $\tau_{x,r(I_k)} \Omega$ is $R$-transverse to $V_k$ at $x$.
        \item For all $k$, $\tau_{x,l(I_k)} \Omega$ is not $R^*$-transverse to $V_k$ at $x$.
        \item For all $k$, $V_k$ is dilation invariant at $x$.
    \end{itemize}
    We refer to $\C_x(\Omega,R,R^*, \delta)$ as the \textbf{pointwise complexity} of $\Omega$ at $x$ at scale below $\delta$ with parameters $(R,R^*)$. 
    
    If $\delta = \infty$, we set $\C_x(\Omega,R,R^*) = \C_x(\Omega,R,R^*,\infty)$, which we refer to as the pointwise complexity of $\Omega$ at $x$ with parameters $(R,R^*)$.
\end{defi}

Our second technical result provides a bound on the pointwise complexity of a general closed symmetric convex subset of $\mathcal{R}_x$.

\begin{proposition}\label{prop:tech2}
    Let $x \in \R^n$, $\delta > 0$, $R \geq 16$, and $R^* \geq D^{2D+1/2}R^{4D}$ be given. Then $\C_x(\Omega,R,R^*, \delta) \leq 4 m D^2$ for any closed symmetric convex set $\Omega \subseteq \cR_x$.
\end{proposition}

\subsection{Elementary tools and techniques}

This section contains elementary lemmas on polynomial inequalities and cutoff functions. Many of these results were proven in \cite{coordinateFree} via compactness arguments. Here we give direct proofs that yield explicit constants.

\subsubsection{Properties of polynomial norms}

We present inequalities for polynomial norms used throughout the paper. 

\begin{lemma}\label{lem:poly1}[cf. Lemma 2.1, part (i) in \cite{coordinateFree}]
Let $x,y \in \R^n$ and $\delta > 0$. Suppose $|x-y|\le \eta \delta$ for $0\le \eta \le 1$. Then for any $P\in\Po$,
\[|P|_{y,\delta}^2 \le (1+C \eta) |P|_{x,\delta}^2
\]
for a controlled constant $C$.
\end{lemma}
\begin{proof}
By Taylor's theorem, for any $\alpha$ with $|\alpha| \le m-1$ we have
\[
    \partial^\alpha P(y) = \sum_{\gamma: |\alpha + \gamma|<m} \frac{1}{\gamma !} (\partial^{\alpha+\gamma}P)(x)\cdot (y-x)^\gamma.
\]
Therefore
\begin{equation}\label{eq: close ps 1}
\begin{split}
    |P|_{y,\delta}^2 &=  \sum_{|\alpha|<m} \frac{\delta^{2(|\alpha| - m)}}{\alpha!} \left( \partial^\alpha P(x) + \sum_{\substack{\gamma>0:\\|\alpha + \gamma| <m}} \frac{1}{\gamma!}(\partial^{\alpha+\gamma}P)(x)\cdot(y-x)^{\gamma}\right)^2\\
    &= |P|_{x,\delta}^2 + (\mathrm{R})
    \end{split}
\end{equation}
where
\begin{align*}
(\mathrm{R}) &= \sum_{|\alpha|<m} \frac{\delta^{2(|\alpha| - m)}}{\alpha!} \sum_{\substack{\gamma_1 > 0, \gamma_2 > 0:\\|\alpha + \gamma_1| <m \\ |\alpha + \gamma_2|<m}} \frac{1}{\gamma_1!\gamma_2!}(\partial^{\alpha+\gamma_1}P)(x)(\partial^{\alpha+\gamma_2}P)(x)(y-x)^{\gamma_1+\gamma_2}\\
& +\sum_{|\alpha|<m} \frac{\delta^{2(|\alpha| - m)}}{\alpha!} 2(\partial^\alpha P)(x) \cdot \sum_{\substack{\gamma>0:\\|\alpha + \gamma| <m}} \frac{1}{\gamma!}(\partial^{\alpha+\gamma}P)(x)\cdot(y-x)^{\gamma} \\
&\le 2 \sum_{|\alpha|<m} \frac{\delta^{2(|\alpha| - m)}}{\alpha!} \sum_{\substack{\gamma_1 \ge 0, \gamma_2 > 0:\\|\alpha + \gamma_1| <m \\ |\alpha + \gamma_2|<m}} \frac{1}{\gamma_1!\gamma_2!}(\partial^{\alpha+\gamma_1}P)(x)(\partial^{\alpha+\gamma_2}P)(x)\cdot(y-x)^{\gamma_1+\gamma_2}.
\end{align*}
Now use the trivial bound
\[
|\partial^{\alpha+\gamma}P(x)|\le|P|_{x,\delta} \frac{(\alpha+\gamma)!}{\delta^{|\alpha+\gamma|-m}}
\]
and the hypothesis $|y-x|\le \delta\eta$ to get that
\[
|(\mathrm{R})| \le 2  \sum_{|\alpha|<m} \frac{\delta^{2(|\alpha|-m)}}{\alpha!} \sum_{\substack{\gamma_1 \ge 0, \gamma_2 > 0:\\|\alpha + \gamma_1| <m \\ |\alpha + \gamma_2|<m}} \frac{(\alpha+\gamma_1)!(\alpha+\gamma_2)!}{\gamma_1!\gamma_2!}\frac{\eta^{\gamma_1+\gamma_2}}{\delta^{2(|\alpha|-m)}} |P|_{x,\delta}^2.
\]
Using the fact that the number of multiindices appearing in each of the above sums is bounded by $D$, and the hypothesis $\eta<1$, we see that
\begin{equation}\label{eq: close ps 2}
|(\mathrm{R})| \le D^3 \left((m-1)!\right)^2\eta|P|_{x,\delta}^2.
\end{equation}
Combining \eqref{eq: close ps 1} and \eqref{eq: close ps 2} proves the lemma, with $C = D^3 \left((m-1)!\right)^2$.
\end{proof}

\begin{lemma}\label{lem: poly 2}[cf. Lemma 2.1, part (ii) in \cite{coordinateFree}]
Let $x \in \R^n$ and $0 < \rho \leq \delta$. Then there exists a controlled constant $C$ such that for any $P,Q \in \Po$,
\[
|P\odot_x Q|_{x,\rho}\le C\delta^m |P|_{x,\delta}|Q|_{x,\rho}.
\]
\end{lemma}
\begin{proof}
By translating and rescaling, we reduce matters to the case $x=0$, $\rho=1$. For $\delta \geq 1$, we have $\delta^m |P|_{0,\delta} \geq |P|_{0,1}$, by \eqref{eqn:norm_scale}. Thus, it suffices to prove the bound $|P \odot_0 Q|_{0,1} \leq C |P|_{0,1} |Q|_{0,1}$. This inequality is a consequence of Lemma \ref{lem:bombieri}.
\end{proof}

\begin{lemma}\label{lem: poly 3}[cf. Lemma 2.1, part (iii) in \cite{coordinateFree}]
Let $x,y \in \R^n$ and $\delta, \rho > 0$. Assume that $|x-y| \le \rho \le \delta$. Then there exists a controlled constant $C$ such that for any $P,Q \in \Po$,
\[
|(P\odot_y Q) - (P\odot_x Q)|_{x,\rho} \le C \delta^m |P|_{x,\delta} |Q|_{x,\delta}.
\]
\end{lemma}
\begin{proof}
By translating and rescaling, we reduce matters to the case $x=0$, $\delta=1$. Write $|\cdot| = |\cdot|_{0,1}$ for the standard norm on $\Po$. Fix $P,Q \in \Po$ with $|P| \leq 1$, $|Q| \leq 1$.  Then $P(z)= \sum_{|\alpha| \leq m-1} c_\alpha z^\alpha$ and $Q(z) = \sum_{|\alpha| \leq m-1} d_\alpha z^\alpha$, with $|c_\alpha|$, $|d_\alpha|$ each bounded by a controlled constant. Our task is to show that $|P \odot_y Q - P \odot_0 Q|_{0,\rho} \leq  C$ for $|y| \leq \rho \leq 1$.

Let $B = B(0,1)$ be the closed unit ball in $\R^n$ of radius $1$ centered at $0$. 

Let $F(z) = P(z)Q(z)$. Then $F$ is a polynomial of degree at most $2m-2$ of the form
\[
F(z) = \sum_{|\alpha| \leq 2m-2} f_\alpha z^\alpha, \;\; |f_\alpha| \leq C, \;\; C \mbox{ controlled}.
\]
Each of the monomial functions $z \mapsto z^\alpha$ is in $\dot{C}^m(B)$ with $\dot{C}^m$ seminorm bounded by a controlled constant. Thus,  $\| F \|_{\dot{C}^m(B)} \leq C'$ for a controlled constant $C'$. Using \eqref{eqn:Cm_norm_bd}, we deduce that $F$ is in $C^{m-1,1}(B)$ and $\| F \|_{C^{m-1,1}(B)} \leq C$ for a controlled constant $C$.

By Taylor's theorem \eqref{eqn:Taylor}, we have
\[
|(P\odot_0 Q) - (P\odot_y Q)|_{0,\rho} = |J_0F - J_yF|_{0,\rho}  \le C_T \|F\|_{C^{m-1,1}(B)} \le C_T C
\]
for $1 \geq \rho \geq |y|$. This completes the proof of the lemma.

\end{proof}

\begin{lemma}[cf. Lemma 2.2 of \cite{coordinateFree}] \label{lem:jetprodbd}
Fix polynomials $P_x, Q_x, R_x$ and $P_y, Q_y, R_y$ in $\Po$, for $|x-y|\le\rho\le\delta$. Suppose that $P_x,P_y\in M_0 \B_{x,\delta}$, $Q_x,Q_y \in M_1 \B_{x,\delta}$, and $R_x, R_y \in M_2 \B_{x,\delta}$. Also suppose that $P_x - P_y \in M_0 \B_{x,\rho}$, $Q_x-Q_y\in M_1\B_{x,\rho}$, and $R_x-R_y \in M_2 \B_{x,\delta}$. Then
\[
|P_x\odot_x Q_x \odot_x R_x - P_y \odot_y Q_y \odot_y R_y|_{x,\rho} \le C \delta^{2m}M_0M_1M_2,
\]
where $C$ is a controlled constant.
\end{lemma}
\begin{proof}
This lemma is identical to Lemma 2.2 in \cite{coordinateFree} with the additional claim that the constant $C$ is controlled. To see that this is true, we examine the proof of Lemma 2.2 in \cite{coordinateFree}. Note that $C$ is a product of a finite number (independent of $D$) of the constants appearing in Lemma 2.1 in \cite{coordinateFree}. Lemmas \ref{lem:poly1}, \ref{lem: poly 2}, and \ref{lem: poly 3} of this paper show that we can take these constants to be controlled.
\end{proof}

\begin{lemma}[cf. equation (2.4) of \cite{coordinateFree}]\label{lem:poly2}
If $|x-y| \le \lambda \delta$ for $\lambda \ge 1$, then for any $P\in\Po$,
\begin{equation*}
    |P|_{y,\delta}\le C' \lambda^{m-1}|P|_{x,\delta} 
\end{equation*}
for a controlled constant $C'$. Consequently,
\[
\cB_{x,\delta} \subseteq C' \lambda^{m-1} \cB_{y,\delta}.
\]
\end{lemma}
\begin{proof}
Apply \eqref{eqn:norm_scale} twice and Lemma \ref{lem:poly1} to get:
\[
|P|_{y,\delta} \le \lambda^m |P|_{y,\lambda \delta} \le (1+C) \lambda^m|P|_{x,\lambda\delta}\le (1+C) \lambda^{m-1}|P|_{x,\delta},
\]
where $C$ is the controlled constant from Lemma \ref{lem:poly1}.
\end{proof}

\subsubsection{Whitney Covers and Partitions of Unity}
\begin{lemma}\label{lem:theta_bound}
For any ball $B \subseteq \R^n$ and any $0<r<1$ there exists a cutoff function $\theta \in C^{m}(\R^n)$ with $\theta \equiv 0$ on $\R^n \backslash B$, $\theta \equiv 1$ on $(1-r)B$, $\|\partial^\alpha\theta\|_{L^\infty(\R^n)}\le C_{\theta, 1}(r) \diam(B)^{-|\alpha|}$ for any $|\alpha| \le m$, where $C_{\theta, 1}(r) := 9\frac{(4m)^{4m}}{r^m}$.
\end{lemma}

\begin{proof}
By translating and rescaling it suffices to construct $\theta$ supported on the unit ball $B= \{ x : |x| \leq 1 \}$.

Let $\psi : \R \rightarrow \R_+$ be given by $\psi(x) = e^{-x^{-1}} e^{-(1-x)^{-1}}$ for $x \in (0,1)$, and $\psi(x) = 0$ for $x \notin (0,1)$. Evidently, $\psi \in C^\infty(\R)$, and $\psi^{(k)}(0)  = \psi^{(k)}(1) = 0$ for all $k \geq 0$. By the product rule, for $x \in (0,1)$, $\psi^{(k)}(x)$ is the sum of at most $2^k$ terms of the form $f_{i,j}(x) = \frac{d^i}{dx^i}(e^{-x^{-1}}) \frac{d^j}{dx^j}(e^{-(1-x)^{-1}})$ with $i+j = k$. By induction on $i$, $\frac{d^i}{dx^i}(e^{-x^{-1}})$ is the sum of at most $2^i$ terms of the form $h_{w,r,s}(x) = w x^{-2s - r} e^{-x^{-1}}$ for integers $r,s$ with $r+s=i$, and real $w$ with $|w| \leq (2s+r)^{r}$. Using the bound $t^K e^{-t} \leq K^K$ ($t,K>0$), we find $|h_{w,r,s}(x)| \leq |w| (2s+r)^{2s+r} \leq (2s+r)^{2s+2r} \leq (2i)^{2i}$, and thus $|\frac{d^i}{dx^i}(e^{-x^{-1}})| \leq 2^i (2i)^{2i} = 8^i i^{2i}$ for $x>0$. Similarly, $|\frac{d^j}{dx^j}(e^{-(1-x)^{-1}})| \leq 8^j j^{2j}$ for $x < 1$. Thus, $|f_{i,j}(x)| \leq 8^{i+j} \max\{i,j\}^{2(i+j)}$ for $x \in (0,1)$. We deduce that  $\| \psi^{(k)} \|_{L^\infty(\R)} \leq 2^k  8^k k^{2k} = (4k)^{2k}$ for $k \geq 0$. 

Note that $\gamma := \int_{-\infty}^\infty \psi(t) dt \geq  \frac{1}{3} e^{-2/3}  \geq \frac{1}{9}$. Now, let 
\[
\upsilon(x) := \gamma^{-1} \int_{-\infty}^x \psi(t) dt.
\]
Then $\upsilon(t) = 0$ for $t \leq 0$, $\upsilon(t) = 1$ for $t \geq 1$, and $\upsilon^{(k)}(0) = \upsilon^{(k)}(1) = 0$ for $k \geq 1$. Finally, $\| \upsilon^{(k)} \|_{L^\infty(\R)} \leq 9 \cdot (4k)^{2k}$ for $k \geq 0$; here, our convention is that $0^0 = 1$.

For $0 < \eta < 1$ let $\varphi_\eta:\R^+\rightarrow \R$ given by $\varphi_\eta(t) = v ((1-t)/(1-\eta))$. Then
\begin{enumerate}
    \item $\varphi_\eta(t) = 1$ for $t \leq \eta$,
    \item $\varphi_\eta(t) = 0$ for $t \geq 1$,
    \item $\| \varphi_\eta^{(k)} \|_{L^\infty(\R^+)} \leq 9 \cdot \frac{(4k)^{2k}}{(1-\eta)^k}$ for $k \geq 0$.
\end{enumerate}

Define $\theta:\R^n\rightarrow \R$ by $\theta(x) := \varphi_{(1-r)^2}(|x|^2)$. Note that $\theta(x) \equiv 0$ for $|x| \geq 1$ due to property 2 of $\varphi_\eta$. Furthermore, $\theta(x) \equiv 1$ for $|x| \leq 1-r$, by property 1 of $\varphi_\eta$. By induction on $|\alpha|$, using the product and chain rules, we establish the following claim: For $0 < |\alpha| \leq m$, the function $\partial^\alpha \theta(x)$ is a sum of at most $2^{|\alpha|-1} \leq 2^{m}$ terms of the form $h_{j,\beta}(x) = C_{j,\beta} \varphi_{(1-r)^2}^{(j)}(|x|^2) \cdot x^\beta$ for integers $j\le m$, multindices $\beta$ with $|\beta|\le m$, and constants $C_{j,\beta}$ satisfying $|C_{j,\beta}| \leq m^{|\alpha|} \leq m^m$. If $|x| \leq 1$ then $|x^\beta| \leq 1$. Property 3 of $\varphi_\eta$ implies that, for $|x| \leq 1$ and $|\alpha| \leq m$:
\[
|\partial^\alpha \theta(x)| \le 2^{m} \cdot m^{m} \cdot 9 \frac{(4m)^{2m}}{(1-(1-r)^2)^m} \leq 9 \frac{(4m)^{3m}}{(2r-r^2)^m} \leq 9 \frac{(4m)^{3m}}{r^m}.
\]
(We use $2r - r^2 \geq r$.) Because $\theta(x) \equiv 0$ for $|x| \geq 1$, we conclude that 
\[
\| \partial^\alpha \theta \|_{L^\infty(\R^n)} \leq 9 \frac{(4m)^{3m}}{r^m} \leq 9 \frac{(4m)^{3m}}{r^m} 2^m \diam(B)^{-|\alpha|} \leq 9 \frac{(4m)^{4m}}{r^m} \diam(B)^{-|\alpha|}.
\]
This completes the proof of the lemma.

\end{proof}

\begin{defi}\label{defn:whit_cover}
A finite collection  $\cW$ of closed balls is a \emph{Whitney cover} of a ball $\widehat{B}\subseteq\R^n$ if \begin{enumerate*}[label=\textnormal{(\arabic*)}] \item{$\cW$ is a cover of $\widehat{B}$}, \item{the collection of third-dilates $\{\frac{1}{3}B:B\in\cW\}$ is pairwise disjoint, and} \item{$\diam(B_1)/\diam(B_2) \in [1/8,8]$ for all balls $B_1,B_2\in\cW$ with $\frac{6}{5}B_1 \cap \frac{6}{5}B_2 \ne \emptyset$}. \end{enumerate*}
\end{defi}

\begin{lemma}[Bounded Overlap of Whitney Covers]\label{lem: bounded overlap}
If $\cW$ is a Whitney cover of $\widehat{B}$ then $\#\{B\in\cW : x \in \frac{6}{5}B\}\le 100^n$ for all $x \in \R^n$.
\end{lemma}
\begin{proof}
See Lemma 2.14 of \cite{coordinateFree} for the proof.
\end{proof}

\begin{lemma}[Partitions of Unity adapted to Whitney Covers -- cf. Lemma 2.15 of \cite{coordinateFree}]\label{lem: p of u}
If $\W$ is a Whitney cover of $\widehat{B}$, then for each $B\in\W$ there exists a non-negative $C^\infty$ function $\theta_B:\widehat{B}\rightarrow[0,\infty)$ such that
\begin{enumerate}
    \item $\theta_B = 0$ on $\widehat{B} \backslash \frac{6}{5}B$.
    \item $|\partial^\alpha \theta_B(x)|\le C \diam(B)^{-|\alpha|}$ for all $|\alpha| \le m$ and $x \in \widehat{B}$.
    \item $\sum_{B\in\W} \theta_B = 1$ on $\widehat{B}$.
\end{enumerate}
Here, $C$ is a controlled constant.
\end{lemma}

\begin{proof} Use Lemma \ref{lem:theta_bound} to obtain a function $\psi_B:\R^n\rightarrow\R$ for each $B \in \cW$ satisfying \begin{enumerate*}[label=\textnormal{(\arabic*)}]\item{$\supp(\psi_B)\subseteq \frac{6}{5}B$},\item{$\psi_B = 1$ on $B$}, and \item{$\|\partial^\alpha \psi_B\|_{L^\infty} \le C \diam(B)^{-|\alpha|}$ for all $|\alpha| \leq m$}\end{enumerate*}, for a controlled constant $C$.

Set $\Psi := \sum_{B\in\cW}\psi_B$ and define
\begin{equation}\label{eq: theta_B}
\theta_B(x):= \psi_B(x)/\Psi(x), \quad x\in\widehat{B}.
\end{equation}
Since each point in $\widehat{B}$ belongs to some $B\in\cW$, $\Psi\ge 1$ on $\widehat{B}$ and thus $\theta_B$ is well-defined on $\widehat{B}$. Property 1 follows from the fact that $\psi_B$ is supported on $\frac{6}{5}B$. Property 3 follows because $\sum_{B\in\cW} \theta_B = \sum_{\B\in\cW} \psi_B/\Psi =1$ on $\widehat{B}$.

Property 2 is valid if $x \in \widehat{B}\backslash \frac{6}{5}B$ since then $J_x(\theta_B) = 0$. Now fix $x \in \frac{6}{5}B \cap \widehat{B}$. If $\psi_{B'}(x) \ne 0$ for some $B'$, then $x \in \frac{6}{5} B'$, so $\frac{6}{5}B \cap \frac{6}{5} B' \ne \emptyset$, and hence, $\diam(B)/\diam(B') \in [\frac{1}{8}, 8]$ by definition of Whitney covers. By Lemma \ref{lem: bounded overlap}, the cardinality of $\cW_x:=\{B' \in \cW : x \in \frac{6}{5}B'\}$ is $\leq 100^n$. Therefore,
\begin{equation}\label{eq: Psi bound}
|\partial^\alpha\Psi(x)| \le \sum_{B' \in W_x} |\partial^\alpha \psi_{B'}(x)| \le \sum_{B' \in W_x} C \diam(B')^{-|\alpha|} \le C' \diam(B)^{-|\alpha|}
\end{equation}
for controlled constants $C$, $C'$. Given \eqref{eq: Psi bound} and the fact that $\Psi\ge 1$ on $\widehat{B}$, by repeated application of the quotient rule we obtain $| \partial^\gamma(1/\Psi(x))| \leq C'' \diam(B)^{-|\gamma|}$ for $|\gamma| \leq m$ for a controlled constant $C''$. By application of the product rule to \eqref{eq: theta_B}, we see that $|\partial^\alpha \theta_B(x)|$ is bounded above by a sum of $2^{|\alpha|}$ terms of the form
\[
|\partial^\beta \psi_B(x)|\cdot | \partial^\gamma(1/\Psi(x))|, \mbox{ where } \beta+\gamma = \alpha.
\]
Given $|\partial^\beta \psi_B(x)| \le C \diam(B)^{-|\beta|}$ we conclude that $|\partial^\alpha\theta_B(x)| \le C''' \diam(B)^{-|\alpha|}$ for a controlled constant $C'''$. This finishes the proof of property 2.
\end{proof}

\begin{lemma}[Gluing lemma -- cf. Lemma 2.16 of \cite{coordinateFree}]\label{lem:glue}
Fix a Whitney cover $\W$ of $\widehat{B}$, a partition of unity $\{\theta_B\}_{B\in\W}$ as in Lemma \ref{lem: p of u}, and points $x_B \in \frac{6}{5}B$ for each $B \in \W$. Suppose $\{F_B\}_{B\in\W}$ is a collection of functions in $C^{m-1,1}(\R^n)$ with the following properties:
\begin{itemize}
    \item $\|F_B\|_{C^{m-1,1}(\R^n)} \le M_0$
    \item $F_B = f$ on $E \cap \frac{6}{5}B$.
    \item $|J_{x_B}F_B - J_{x_{B'}}F_{B'}|_{x_B, \diam(B)}\le M_0$ whenever $\frac{6}{5}B \cap \frac{6}{5} B' \ne \emptyset$.
\end{itemize}
Let $F = \sum_{B\in\W} \theta_B F_B$. Then $F\in C^{m-1,1}(\widehat{B})$ with $F=f$ on $E\cap \widehat{B}$ and $\|F\|_{C^{m-1,1}(\widehat{B})} \le C M_0$, where $C$ is a controlled constant.
\end{lemma}
\begin{proof}

We sketch the proof, following the proof of Lemma 2.16 in \cite{coordinateFree}, which is identical to Lemma \ref{lem:glue} but without the claim that $C$ is a controlled constant. 

See the proof of Lemma 2.16 of \cite{coordinateFree} for verification that $F= f$ on $E \cap \widehat{B}$. 

The proof of Lemma 2.16 of \cite{coordinateFree} then goes on to show that
\begin{equation}\label{eqn:glue1}
|J_x(F) - J_y(F) |_{x,|x-y|} \leq C M_0 
\end{equation}
whenever $x,y \in \widehat{B}$ with $|x-y| \leq \delta_{\min} := \frac{1}{100} \min \{ \diam(B) : B \in \cW\}$. By definition of the $|\cdot |_{x,\delta}$-norm, \eqref{eqn:glue1} implies the local Lipschitz condition: 
\[
|\partial^\alpha F(x)  - \partial^\alpha F(y)| \leq C' M_0 |x-y| \quad \mbox{for } |\alpha| = m-1,\; x,y \in \widehat{B}, \; |x-y| \leq \delta_{\min}.
\]
Then by the triangle inequality, the Lipschitz constant of $\partial^\alpha F$ on all of $\widehat{B}$ is $ \leq C'M_0$, for each $|\alpha| = m -1$. Therefore, $\| F \|_{C^{m-1,1}(\widehat{B})} \leq C'' M_0$, as desired. 

All that remains is to show that $C$ in \eqref{eqn:glue1} is a controlled constant. From the proof in \cite{coordinateFree}, we note that $C$ is a sum or product of finitely many (independent of $m,n$) of the constants $C_T$ (appearing in Taylor's theorem), $100^n$, $4^m$, and the constants in Lemmas 2.2, 2.15, and equation (2.4) of \cite{coordinateFree}. By Lemmas \ref{lem:jetprodbd}, \ref{lem: p of u}, and \ref{lem:poly2} of the present paper, we see that each of the last three of these constants is controlled. $C_T$ is controlled by Proposition \ref{lem:TT}. Thus, $C$ is a controlled constant.
\end{proof}

\section{Geometry in the Grassmanian}\label{sec:GeomGrass}

Let $(X,\langle \cdot, \cdot \rangle)$ be a real finite-dimensional Hilbert space, and set $d := \dim X$. Denote the norm on $X$ by $| \cdot | = \sqrt{\langle \cdot,\cdot \rangle} $, and let $\cB = \{ x \in X : |x| \leq 1 \}$ be the unit ball of $X$. Write $\cK(X)$ for the collection of all closed, convex, symmetric subsets of $X$. Recall that a subset $\Omega \subseteq X$ is  symmetric if $v \in \Omega \implies -v \in \Omega$.

\subsection{Tools from linear and multilinear algebra}\label{sec:tools}

Here we present a few tools and pieces of terminology from multilinear algebra.

For $0 \leq k \leq \dim(X)$, let $\bigwedge^k X$ be the $k$'th exterior power of $X$. We refer to  elements of $\bigwedge^k X$ as tensors. If $v_1,v_2,\dots,v_k \in X$ then $v_1 \wedge v_2 \wedge \dots \wedge v_k \in \bigwedge^k X$ is called a \emph{pure tensor}. Every tensor is a finite linear combination of pure tensors. We specify a Hilbert space structure on $\bigwedge^k X$ as follows. Let $e_1,\dots,e_d$ be an orthonormal basis for $X$. For $1 \leq i_1 <  \dots < i_k \leq d$, $1 \leq j_1  < \dots < j_k \leq d$, let $\langle  \bigwedge_{\ell=1}^k e_{i_\ell} ,  \bigwedge_{\ell=1}^k e_{j_\ell} \rangle$ be $1$ if $i_\ell = j_\ell$ for all $\ell$, and $0$ otherwise. We extend this inner product to all of $\bigwedge^k X$ by bilinearity. Then $\{ \bigwedge_{\ell=1}^k e_{i_\ell} : 1 \leq i_1  < \dots < i_k \leq d\}$ is an orthonormal basis for $\bigwedge^k X$.  Write $\langle \cdot, \cdot \rangle$ and $| \cdot |$ for the inner product and associated norm on $\bigwedge^k X$. This inner product can be defined in a basis-independent manner as the unique bilinear mapping obeying the identity
\[
\left\langle \bigwedge_{i=1}^k v_i, \bigwedge_{i=1}^k w_i \right\rangle = \det (\langle v_i, w_j \rangle)_{1 \leq i,j \leq k}, \mbox{ for all } v_1,\dots,v_k, w_1,\dots,w_k \in X.
\]
 In particular, the Hilbert space structure on $\bigwedge^k X$ is independent of the choice of orthonormal basis for $X$.

Let $V$ be a $k$-dimensional subspace of $X$, and fix  a basis $\{v_j\}_{1 \leq j \leq k}$ for $V$. We set $\omega_V := v_1 \wedge v_2 \wedge \dots \wedge v_k \in \bigwedge^k X$. We call $\omega_V$ a \emph{representative form} for $V$. The next remark implies that the representative forms associated to different choices of basis for $V$ are scalar multiples of one another.

\begin{rmk}\label{rem:basis_independence}
If $\{\hat{v}_j \}_{1 \leq j \leq k}$ and $\{ v_j \}_{1 \leq j \leq k}$ are two bases for $V$ then $\hat{v}_1 \wedge \hat{v}_2 \wedge \dots \wedge \hat{v}_k = \det(A) \cdot v_1 \wedge v_2 \wedge \dots \wedge v_k$, where $A=(A_{ij}) \in \R^{k \times k}$ is the change-of-basis matrix defined by the relations $\hat{v}_i = \sum_j A_{ij} v_j$ ($i=1,2,\dots,k$).
\end{rmk}

The eigenvalues of a self-adjoint operator $T : X \rightarrow X$ will be written in descending order: $ \lambda_1(T) \geq \lambda_2(T) \geq \cdots \geq \lambda_d(T)$ ($d=\dim X$). 

Let $X_0$, $X_1$ be $k$-dimensional Hilbert spaces. We denote the singular values of a linear transformation $T : X_0 \rightarrow X_1$ by $ \sigma_1(T)\geq \sigma_2(T) \geq \dots \geq \sigma_k(T) \geq 0$. The squared singular values of $T$ are eigenvalues of $T^* T$, or $T T^*$ (equivalently), i.e., 
\[
\sigma_\ell(T) =  \sqrt{\lambda_\ell(T^* T)} = \sqrt{\lambda_\ell(T T^*)} \mbox{ for } \ell=1,2,\dots,k.
\]
The extremal singular values $\sigma_1(T)$ and $\sigma_k(T)$ are related to the operator norms of $T$ and $T^{-1}$. First, $\sigma_1(T) = \| T \|_{op}$. Also,  $\sigma_k(T) > 0$ if and only if $T : X_0 \rightarrow X_1$ is invertible, and then $\sigma_k(T) = \| T^{-1} \|_{op}^{-1}$. This implies the following description:
\begin{equation}
    \label{eqn:smallest_sv1}
    \sigma_k(T) = \sup \{ \eta \geq 0 : \| Tx \|_{X_1} \geq \eta \| x \|_{X_0} \mbox{ for all } x \in X_0 \}.
\end{equation}
Finally, $\sigma_k(T)$ has a description in terms of the images of balls under $T$. Let $\cB_{X_j} := \{ x \in X_j : \| x \|_{X_j} \leq 1\}$ be the unit ball of $X_j$ ($j \in \{0,1\}$). Then
\begin{equation}
    \label{eqn:smallest_sv2}
    \sigma_k(T) = \sup \{ \eta \geq 0 : T(\cB_{X_0}) \supseteq \eta \cB_{X_1} \}.
\end{equation}

\subsection{Angles between subspaces}\label{sec: angles}
Let $G(k,X)$ be the Grassmannian of $k$-dimensional subspaces of $X$ ($1 \leq k \leq d$). Note that $G(k,X) \subseteq \cK(X)$.

Given $V, W \in G(k,X)$,  the \emph{maximum principal angle} $\theta_{\max}(V,W) \in [0,\frac{\pi}{2}]$ between $V$ and $W$ is defined by
\begin{equation}\label{eq:mpa1}
   \theta_{\max}(V,W) := \arccos\left(\inf\left\{\frac{|\Pi_W v|}{|v|} : v \in V, v \neq 0 \right\}\right).
\end{equation}
Here and below, we write $\Pi_W : X \rightarrow W $ to denote the orthogonal projection operator on a subspace $W$ of $X$. Below we show that $\theta_{\max}(\cdot,\cdot)$ is symmetric. In fact,  $\theta_{\max}(\cdot,\cdot)$ is a metric on the Grassmanian $G(k,X)$. For a further discussion of principal angles, see \cite{Qiu2005} and the references therein.

Given $V,W \in G(k,X)$, let 
\begin{equation}\label{eqn:def_angle1}
\angle(V,W) = \arccos \left( \frac{|\langle \omega_V,\omega_W\rangle|}{|\omega_V| \cdot | \omega_W|} \right).
\end{equation}
By Remark \ref{rem:basis_independence}, the quantity $\angle(V,W)$ is independent of the choice of representative forms for $V$ and $W$. 

The quantities $\angle(V,W)$ and $\theta_{\max}(V,W)$ are related to singular values of the projection operator $T_{V \rightarrow W} := \Pi_W |_V : V \rightarrow W$. In fact, we have the identities
\begin{equation}
    \label{eqn:sv1}
    \begin{aligned}
    &\cos(\angle(V,W)) = \sigma_1(T_{V \rightarrow W})  \sigma_2(T_{V \rightarrow W}) \dots \sigma_k(T_{V \rightarrow W}), \\
    & \cos (\theta_{\max}(V,W)) = \sigma_k(T_{V \rightarrow W}).
    \end{aligned}
\end{equation}
The first identity can easily be seen to be true by computing \eqref{eqn:def_angle1} using principal vectors for $V$ and $W$, and using the singular value characterization of principal angles; see \cite{galantai2006jordan} for details. The second identity follows from \eqref{eqn:smallest_sv1} and \eqref{eq:mpa1}.

\begin{lemma}
\label{lem:angles_projections}
Fix $\eta >0$. If $X_0$ and $H$ are subspaces of $X$ such that $|\Pi_{H} x| \geq \eta |x|$ for all $x \in X_0$, then $\dim(X_0) = \dim(\Pi_{H}X_0)$ and $\cos(\theta_{\max}(X_0,\Pi_{H}X_0)) \geq \eta$.
\end{lemma}
\begin{proof}
Condition $\dim(X_0) = \dim(\Pi_{H}X_0)$ holds as $\Pi_{H}|_{X_0} : X_0 \rightarrow X$ is injective. As $\Pi_{\Pi_{H}X_0}= \Pi_{H}$ on $X_0$,  $|\Pi_{\Pi_{H}X_0}(x)| \geq \eta |x|$ for $x \in X_0$ by the lemma's hypothesis. The bound $\cos(\theta_{\max}(X_0,\Pi_{H}X_0)) \geq \eta$ is  a consequenece of the definition \eqref{eq:mpa1}.
\end{proof}

\begin{lemma}\label{lem:mpa1}
    Let $W$ and $V$ be subspaces of $X$ of equal dimension. Then the following conditions are equivalent.
    \begin{enumerate}
        \item $\cos(\theta_{\max}(V,W)) \geq \eta$
        \item $|\Pi_W(v)| \geq \eta |v| $ for all $v \in V$.
        \item $| \Pi_V(w)| \geq \eta |w|$ for all $w \in W$.
        \item $\cos(\theta_{\max}(W,V)) \geq \eta$
    \end{enumerate}
\end{lemma}
\begin{proof}
The equivalence of conditions 1 and 2 is immediate from the definition \eqref{eq:mpa1}. The equivalence of conditions 3 and 4 follows for the same reason.

We prove the equivalence of conditions 2 and 3 by duality. Let $T_{V \rightarrow W} := \Pi_W|_V$ and $T_{W \rightarrow V} := \Pi_V|_W$. Condition 2 is equivalent to the claim that $T_{V \rightarrow W}$ is invertible and $\| T_{V \rightarrow W}^{-1}\|_{op} \leq \eta^{-1}$. Similarly, condition 3 is equivalent to the claim that $T_{W \rightarrow V}$ is invertible and $\| T_{W \rightarrow V}^{-1}\|_{op} \leq \eta^{-1}$. Since $T_{V \rightarrow W}$ is the adjoint of $T_{W \rightarrow V}$, we obtain the equivalence of conditions 2 and 3.
\end{proof}
From the equivalence of conditions 1 and 4 of Lemma \ref{lem:mpa1}, we learn that 
\begin{equation}\label{eqn:maxangle_sym}
\theta_{\max}(V,W) = \theta_{\max}(W,V) \quad\mbox{for } V,W \in G(k,X).
\end{equation}

See Section 2 of \cite{Qiu2005} for a proof of the following result.

\begin{lemma}
\label{lem:mpa2}
If $V$ and $W$ are subspaces of $X$ of equal dimension then $\cos (\theta_{\max}(V,W)) = \cos ( \theta_{\max}(V^\perp, W^\perp))$.
\end{lemma}

Thanks to \eqref{eqn:sv1}, we have the following result.
\begin{lemma}\label{lem:ineq1} Let $V,W \in G(k,X)$. Then
\begin{equation}\label{eqn:ineq1}
    \cos(\theta_{\max}(V,W))^k \leq  \cos(\angle(V,W)) \leq \cos(\theta_{\max}(V,W)).
\end{equation}
\end{lemma}

\subsection{Transversality}

Recall that $\cK(X)$ denotes the set of all closed, convex, symmetric subsets of $X$. Given $\Omega \in \cK(X)$ and $a > 0$, let $a \cdot \Omega := \{ a \cdot x : x \in \Omega\}$. Given a function $T : X \rightarrow X$, let $T(\Omega) := \{ T(x) : x \in \Omega\}$.

We start with an elementary lemma. Given $A \subseteq X$ and a subspace $V$ in $X$, let $A/V$ denote the subset $\{a + V : a \in A \}$ of the quotient space $X/V = \{ x + V : x \in X \}$.
\begin{lemma}\label{lem:mod_perp}
Let $A,B \subseteq X$. Then  $A / V \subseteq B/V$ if and only if $\Pi_{V^\perp} A \subseteq \Pi_{V^\perp} B$.
\end{lemma}
\begin{proof}
Note that $A /V \subseteq B/V$ if and only if for every $a \in A$ there exists $b \in B$ such that $a-b\in V$.

Suppose $A / V \subseteq B / V$. Fix an arbitrary $x \in \Pi_{V^\perp} A$. Let $a \in A$ with $x = \Pi_{V^\perp} a$. Because $A / V \subseteq B/V$, there exists $b \in B$ so that $a-b \in V$. Then $\Pi_{V^\perp} b = \Pi_{V^\perp} a$. Thus, $x \in \Pi_{V^\perp} B$. So, we've shown $\Pi_{V^\perp} A \subseteq \Pi_{V^\perp} B$.

Conversely, suppose $\Pi_{V^\perp} A \subseteq \Pi_{V^\perp} B$. Fix $a \in A$. Let $x = \Pi_{V^\perp} a$. Because $\Pi_{V^\perp} A \subseteq \Pi_{V^\perp} B$, there exists $b \in B$ with $x = \Pi_{V^\perp} b$. But then $\Pi_{V^\perp}(a-b) = x - x = 0$. So, $a-b \in V$. This proves $A/V \subseteq B/V$.
\end{proof}

We now introduce the concept of transversality in the Hilbert space $X$.
\begin{defi}\label{def:trans_hilbert}
Let $\Omega \in \cK(X)$, let $V \subseteq X$ be a subspace, and let $R \geq 1$. Then $\Omega$ is \emph{$R$-transverse} to $V$ if 
\begin{align}
&\label{eqn:Rtrans1a}\Omega \cap V \subseteq R \cdot \B \\
&\label{eqn:Rtrans1b}\Pi_{V^\perp}(\Omega \cap \B) \supseteq R^{-1}\cdot  \B \cap V^\perp.
\end{align}
\end{defi}
In particular, if $X = \cR_x$, then $\Omega \subseteq X$ is $R$-transverse to $V$ if and only if it is $R$-transverse to $V$ at $x$ (see Definition \ref{defn:trans1}).

Using Lemma \ref{lem:mod_perp}, we obtain an equivalent formulation of transversality used in our previous work~\cite{coordinateFree}. This will allow us to later borrow results from~\cite{coordinateFree}.
\begin{cor}\label{cor:eq_trans}
Let $\Omega \in \cK(X)$ and let $V$ be a subspace of $X$. Then $\Omega$ is $R$-transverse to $V$ if and only if (A) $\Omega \cap V \subseteq R \cdot \cB$, and (B) $(\Omega \cap \cB) / V \supseteq R^{-1} \cdot \cB/V$. 
\end{cor}

The notion of transversality between a pair of subspaces (i.e., when $\Omega$ is a subspace) admits an equivalent formulation in terms of principal angles.

\begin{lemma}\label{lem:trans1}
    Let $W, V$ be subspaces of $X$, and $R \geq 1$. Then $W$ is $R$-transverse to $V$ if and only if $\dim(W) = \dim(V^\perp)$ and $\cos(\theta_{\max}(W,V^\perp)) \geq R^{-1}$.
\end{lemma}
\begin{proof}
When $\Omega = W$ is a subspace, condition \eqref{eqn:Rtrans1a} is equivalent to the assertion that $W \cap V = \{0\}$. Thus, from \eqref{eqn:Rtrans1a}, \eqref{eqn:Rtrans1b}, $W$ is $R$-transverse to $V$ if and only if (a) $W \cap V = \{0\}$ and (b) $\Pi_{V^\perp}(W \cap \cB) \supseteq R^{-1} \cdot \cB \cap V^\perp$.

Note that condition (a) implies $\dim(W) \leq \dim(V^\perp)$.

Note that condition (b) implies that $T_{W \rightarrow V^\perp} := \Pi_{V^\perp}|_W : W \rightarrow V^\perp$ is surjective. Hence, condition (b) implies $\dim(W) \geq \dim(V^\perp)$.

Hence, if $W$ is $R$-transverse to $V$ then $\dim(W) = \dim(V^\perp)$, and condition (b) is then equivalent to the inequality $\sigma_k(T_{W \rightarrow V^\perp}) \geq R^{-1}$ (see \eqref{eqn:smallest_sv2}), which is equivalent to the inequality $\cos(\theta_{\max}(W,V^\perp)) \geq R^{-1}$ (see \eqref{eqn:sv1}).

On the other hand, suppose $\dim(W) = \dim(V^\perp)$ and $ \cos(\theta_{\max}(W,V^\perp)) \geq R^{-1}$. Thus, $\sigma_k(T_{W \rightarrow V^\perp}) \geq R^{-1}$, which implies condition (b) above (again, see \eqref{eqn:sv1} and \eqref{eqn:smallest_sv2}). In particular, $T_{W \rightarrow V^\perp} : W \rightarrow V^\perp$ is surjective. As $\dim(W) = \dim(V^\perp)$, we have that $T_{W \rightarrow V^\perp}$ is injective. Thus, $\{0\} = \ker( \Pi_{V^\perp}|_W) = W \cap V$, which gives condition (a) above. So $W$ is $R$-transverse to $V$. 

This completes the proof of the lemma.
\end{proof}

By Lemma \ref{lem:trans1} and Lemma \ref{lem:mpa2}, we have the following result.

\begin{cor}\label{cor:trans2}
Let $W$ and $V$ be subspaces of $X$, and let $R \geq 1$. Then $W$ is $R$-transverse to $V$ if and only if $V$ is $R$-transverse to $W$.
\end{cor}

\begin{lemma}
\label{lem:basic_trans_result}
Let $W,V$ be subspaces of $X$, and let $r > 0$ and $R \geq 1$. If $W$ is $R$-transverse to $V$ then $(W + r \B) \cap V \subseteq R r \B$.
\end{lemma}
\begin{proof}
Fix $x \in (W + r \B) \cap V$. As $x \in W + r \B$, we have $| \Pi_{W^\perp} x | = \dist(x, W) \leq r$. Observe that $| \Pi_{W^\perp} x | \geq R^{-1} |x|$ due to the condition $\theta_{\max}(V, W^\perp) \geq R^{-1}$ and since $x \in V$. Thus, $|x| \leq Rr$, so $x \in Rr \B$, as desired.
\end{proof}

\begin{lemma}
\label{lem:basic_trans_result2}
Let $\Omega \in \cK(X)$ and let $V$ be a subspace of $X$. Let $T : X \rightarrow X$ be an invertible linear transformation satisfying either (A) $|x | \leq |Tx| \leq M |x|$ for all $x \in X$, or (B) $M^{-1} |x| \leq |Tx| \leq |x|$ for all $x \in X$. 

If $\Omega$ is $R$-transverse to $V$ then $T(\Omega)$ is $M R$-transverse to $T (V)$.
\end{lemma}
\begin{proof}
We suppose $\Omega$ is $R$-transverse to $V$, so (a) $\Omega \cap V \subseteq R \B$ and (b) $R^{-1} \B/ V \subseteq (\Omega \cap \B)/V$. Here we use the formulation of transversality given in Corollary \ref{cor:eq_trans}. 

If $T$ satisfies condition (A) then $\| T^{-1}\|_{op} \leq 1$ and $\|T \|_{op} \leq M$, implying the set inclusions $\B \subseteq T(\B)$ and $T(\B) \subseteq M \B$. By (a),
\[
T(\Omega) \cap T(V) = T(\Omega \cap V) \subseteq T(R \B) \subseteq MR \B,
\]
and (b) implies that
\[
\begin{aligned}
R^{-1} \B/T(V) &\subseteq R^{-1} T(\B)/T(V) \subseteq T(\Omega \cap \B)/T(V) = (T(\Omega) \cap T(\B))/T(V) \\
&\subseteq (T(\Omega) \cap M \B)/T(V) \subseteq M (T(\Omega) \cap \B)/T(V).
\end{aligned}
\]
Thus, $(M R)^{-1} \B/T(V) \subseteq (T(\Omega) \cap \B)/T(V))$. We deduce from the previous inclusions that $T(\Omega)$ is $MR$-transverse to $V$.

If $T$ satisfies condition (B) then $T(\B) \subseteq \B$ and $M^{-1}\B \subseteq  T(\B)$, thus by (a),
\[
T(\Omega) \cap T(V) = T(\Omega \cap V) \subseteq T(R \B) \subseteq R \B \subseteq M R \B,
\]
and by (b),
\[
\begin{aligned}
M^{-1}R^{-1} \B/T(V) \subseteq R^{-1} T(\B)/T(V) \subseteq T(\Omega \cap \B)/T(V) &= (T(\Omega) \cap T(\B))/T(V) \\
&\subseteq (T(\Omega) \cap  \B)/T(V),
\end{aligned}
\]
so $T(\Omega)$ is $M R$-transverse to $T(V)$.
\end{proof}

\section{Rescaling dynamics}\label{sec:rescale_grass}

Let $(X, \langle \cdot, \cdot \rangle)$ be a real Hilbert space of finite dimension $d := \dim(X)<\infty$. Write $| x | = \sqrt{\langle x,x \rangle} $ for the norm of a vector $x \in X$. Let $\tau_\delta$ be a $1$-parameter family of linear operators on $X$ of the following form. Fix $m \geq 1$. Suppose that $X$ admits a direct sum decomposition 
\begin{equation}\label{eqn1:sec3}
    X = \bigoplus\limits_{\nu=1}^m X_\nu,
\end{equation}
for pairwise orthogonal subspaces $X_\nu \subseteq X$. Let $\tau_\delta : X \rightarrow X$ satisfy
\begin{equation}\label{eqn2:sec3}
\tau_\delta |_{X_\nu} = \delta^{-\nu} \cdot \mathrm{id}|_{X_\nu} \quad (\delta > 0).
\end{equation}
In this description of $\tau_\delta$ we allow that $ X_\nu = \{0\}$ for certain $\nu$.

\begin{defi}\label{defi:hilbert_system}
We refer to a tuple $\cX= (X,\tau_\delta)_{\delta>0}$ satisfying \eqref{eqn1:sec3}, \eqref{eqn2:sec3} as a \emph{Hilbert dilation system}. A Hilbert dilation system $\cX$ is said to be \emph{simple} provided that $\dim(X_\nu) \in \{0,1\}$ for all $\nu=1,2,\dots,m$.

\end{defi}

\begin{defi}\label{defi:di_subspace}
A subspace $V \subseteq X$ is \emph{dilation-invariant}, or DI, if $\tau_\delta V = V$ for all $\delta>0$. If $V$ is DI then
\[
    V = \bigoplus\limits_{\nu=1}^m V_\nu, \quad \mbox{with }  V_\nu = V \cap X_\nu  \subseteq X_\nu.
\]
The \emph{signature} of a DI subspace $V \subseteq X$ is the quantity 
\[
\sgn(V) = \sum_{\nu=1}^m \nu \cdot \dim (V \cap X_\nu).
\]
\end{defi}
We note that 
\begin{equation}
\label{eqn3:sec3}
\left\{
\begin{aligned}
&V \mbox{ dilation-invariant} \implies V^\perp \mbox{ dilation-invariant, and} \\
& \sgn(V^\perp) = \Sigma_0 - \sgn(V), \;\; \Sigma_0 := \sum_{\nu=1}^m \nu \cdot \dim(X_\nu).
    \end{aligned}
\right.
\end{equation}

We study the behavior of orbits of $\tau_\delta$ acting on the Grassmanian $G(k,X)$. To do so, we will pass to the action of $\tau_\delta$ on the $k$-fold exterior product $\bigwedge^k X$.

The linear transformation $\tau_\delta : X \rightarrow X$ induces a linear transformation $\tau_\delta^* : \bigwedge^k X \rightarrow \bigwedge^k X$ defined by its action on the pure tensors: 
\[
\tau_\delta^*(v_1 \wedge v_2 \wedge \dots \wedge v_k) = \tau_\delta(v_1) \wedge \tau_\delta(v_2) \wedge \dots \wedge \tau_\delta(v_k).
\]

If $V$ is a DI subspace of $X$ of dimension $k$, and $\omega_V \in \bigwedge^k X$ is a representative form for $V$ (i.e., $\omega_V$ is the tensor product of a basis for $V$), then 
\begin{equation}\label{eq:DI_id}
\tau_\delta^*(\omega_V) = \delta^{-\sgn(V)} \omega_V.
\end{equation}
Because all representative forms of $V$ are equivalent up to a scalar multiple, it suffices to verify \eqref{eq:DI_id} for the form associated to a particular basis for $V$. Because $V$ is DI it admits a basis of the form $\{e_j\}_{j=1}^k$, with $e_j \in X_{i_j}$ for each $j$. Consider the representative form $\omega_V = e_1 \wedge  \dots \wedge e_k$ for $V$. Note that $\tau_\delta(e_j) = \delta^{- i_j} e_j$ for all $j$, and $\sgn(V) = \sum_{j=1}^k i_j$. Thus, by definition of $\tau_\delta^*$, 
\[
\tau_\delta^*(\omega_V) = \tau_\delta(e_1) \wedge \dots \wedge \tau_\delta(e_k) = \delta^{- \sgn(V)} e_1 \wedge \dots \wedge e_k,
\]
giving \eqref{eq:DI_id}.

\subsection{Quantitative stabilization for the action of a simple Hilbert dilation system on the Grassmanian}\label{sec:4.1}

Let $H \in G(k,X)$. The parametrized family of subspaces $(\tau_\delta H)_{\delta > 0}$ is an \emph{orbit} of $\tau_\delta$ in $G(k,X)$. The orbit $\tau_\delta H$ converges to a subspace $H_0 \in G(k,X)$ in the Grassmanian topology in the limit as $\delta \rightarrow 0^+$, and furthermore the limit subspace $H_0$ is dilation invariant (see the proof of Lemma 3.12 in~\cite{coordinateFree}).

The main result of this section is a quantitative bound on the distance of the orbit $\tau_\delta H$ to the set of dilation invariant subspaces when $\delta$ varies in a compact interval $I$. Specifically, we have:

\begin{proposition}\label{prop:subspace_dil}
Let $(X,\tau_\delta)_{\delta > 0}$ be a simple Hilbert dilation system. Let $H \in G(k,X)$, $1 \leq k \leq d = \dim(X)$. Fix $\eta \in (0,1/2)$ and a compact interval $I  \subseteq (0,\infty)$ with $\frac{r(I)}{l(I)} \geq \left(\frac{2^d}{\eta} \right)^{dk+2}$. There exist $\delta \in I$ and a dilation invariant subspace $V \in G(k,X)$ satisfying $\cos( \theta_{\max}(\tau_\delta H, V)) \geq 1 - \eta$.
\end{proposition}
The rest of Section \ref{sec:4.1} is devoted to proving Proposition \ref{prop:subspace_dil}. The restriction that $(X,\tau_\delta)_{\delta > 0}$ is simple ($\dim X_\nu \in \{0,1\}$ $\forall \nu$) will be in place for the rest of this section. We expect it is possible to prove a variant of Proposition \ref{prop:subspace_dil} without this restriction, but the arguments are likely more involved and the constants are slightly worse. Anyway, the above version is sufficient for the needed application in Section \ref{sec:WhConv}. 

We introduce notation to be used in the proof. We order the indices $\nu$ for which $\dim X_\nu = 1$ in an increasing sequence: $1 \leq \nu_1 < \nu_2 < \dots < \nu_d \leq m$. For $j=1,2,\dots,d$, let $e_j \in X$ be a unit vector spanning $X_{\nu_j}$. Then:
\begin{equation}\label{eqn:assump3.2}
\begin{aligned}
&X \mbox{ admits an orthonormal basis }\{ e_{1},e_{2},\dots, e_{d} \} \mbox{ with} \\
&\tau_\delta(e_{j}) = \delta^{-\nu_j} e_{j}\qquad (j=1,2,\dots,d, \; \delta > 0), \mbox{ and}\\
& \nu_1,\dots,\nu_d \in \mathbb{N}, \;\; 1 \leq \nu_1 < \dots < \nu_d \leq m.
\end{aligned}
\end{equation}

Let $[d] := \{1,2,\dots,d\}$.  Given $S \subseteq [d]$, let $V_S := \spn \{ e_j : j \in S \}$. Note that a subspace $V \subseteq X$ is dilation invariant if and only if $V=V_S$ for some $S \subseteq [d]$. 

For $S \subseteq [d]$ with $\#(S) = k$, let $\omega_{S} :=  \bigwedge_{j \in S} e_j \in \bigwedge^k X$ be a representative form for $V_S \subseteq X$. Note that $\{ \omega_{S} : S \subseteq [d], \; \#(S)=k\}$ is an orthonormal basis for $\bigwedge^k X$. See Section \ref{sec:tools} for a discussion of the Hilbert space structure on $\bigwedge^k X$.

Given $S \subseteq [d]$, define $c(S) :=( c_1(S),c_2(S),\dots, c_d(S)) \in \{0,1,2,\dots,d\}^d$ by
\[
c_\ell(S) := \# \{ j \in S : j \leq \ell \} \qquad (\ell=1,2,\dots,d).
\]
By definition,
\begin{equation}\label{prop:c_ell2}
 c_\ell(S) = \dim(V_S \cap X_{\leq \ell})\mbox{, where } X_{\leq \ell} := \spn \{ e_1,e_2,\dots,e_\ell\}.    
\end{equation}
Also, observe that
\begin{equation}\label{prop:c_ell}
    S \neq S' \implies c_\ell(S) \neq c_\ell (S') \mbox{ for some } \ell=1,2,\dots,d.
\end{equation}
Let $A,B: X \rightarrow X$ be linear operators. Then we write $A \ge B$ to mean that $(A-B)$ is positive semidefinite.

\begin{lemma}\label{lem:multi2}
    Let $H \in G(k,X)$ for $1 \leq k \leq d$. Suppose $\epsilon \in (0, 1/2)$, $\delta, \delta' >0$ and $S, S' \subseteq [d]$ satisfy $\#(S) = \#(S') = k$, $\delta \geq \frac{1}{\epsilon^2} \delta' $,
    \[
        \frac{|\langle \omega_{\tau_\delta H}, \omega_{S} \rangle|}{|\omega_{\tau_{\delta} H}|} \geq \epsilon, \mbox{  and  } \;\;
        \frac{|\langle \omega_{\tau_{\delta'} H}, \omega_{S'} \rangle|}{|\omega_{\tau_{\delta'} H}|} \geq \epsilon.
    \]
    Then $c_\ell(S) \geq c_\ell(S')$ for $\ell=1,2,\dots,d$. 
\end{lemma}
    
\begin{proof}
  For sake of contradiction, let $H,\delta,\delta',S,S'$ be as in the hypotheses of the lemma, and suppose that there exists $\ell \in [d]$ with $c_\ell(S) < c_\ell(S')$. Without loss of generality,  $\delta' = 1$. Then $\delta \geq \frac{1}{\epsilon^2}$, $| \langle \omega_{\tau_\delta H}, \omega_S \rangle| \geq \epsilon |\omega_{\tau_\delta H}|$, and  $|\langle \omega_H, \omega_{S'} \rangle| \geq \epsilon |\omega_H|$. Thanks to \eqref{eqn:def_angle1} and \eqref{eqn:ineq1}, we have
    \begin{equation}\label{eqn:bb1}
    \cos (\theta_{\max}( \tau_\delta H,V_S))  \geq \cos (\angle(\tau_\delta H, V_S)) = \frac{|\langle \omega_{\tau_\delta H}, \omega_S \rangle|}{|\omega_{\tau_\delta H}|} \geq \epsilon,
    \end{equation}
    and similarly
    \begin{equation}\label{eqn:bb2}
    \cos (\theta_{\max}(H,V_{S'})) \geq \epsilon.
    \end{equation}

    Consider the orthogonal subspaces
    \[
    X_{\leq \ell} := \spn\{ e_j : j \leq \ell\} ,\;\;\; X_{> \ell} := \spn \{e_j : j  > \ell\}.
    \]
    Then $X = X_{\leq \ell} \oplus X_{>\ell}$. If $\ell=d$, by convention $X_{>\ell} = \{ 0\}$. By \eqref{prop:c_ell2},
    \begin{align}
    &\dim(V_S \cap X_{\leq \ell}) = c_\ell(S), \label{eqqq1}\\
    &\dim(V_{S'} \cap X_{\leq \ell}) = c_\ell(S'). \label{eqqq2}
    \end{align}
    Let $\Pi_{\leq \ell} := \Pi_{X_{\leq \ell}}$ and $\Pi_{> \ell} := \Pi_{X_{>\ell}}$ be the orthogonal projection operators associated to $X_{\leq \ell}$ and $X_{>\ell}$, respectively.
    
    From \eqref{eqn:bb2} and Lemma \ref{lem:mpa1}, $|\Pi_H(x)| \geq \epsilon |x|$ for all $x \in V_{S'}$. Set 
    \[
    \widetilde{H} = \Pi_{H} (V_{S'} \cap X_{\leq \ell}) \subseteq H.
    \]
    Applying Lemma \ref{lem:angles_projections} to the subspace $X_0 = V_{S'} \cap X_{\leq \ell}$ gives
    \begin{equation}\label{eqqq2a}
    \dim ( \widetilde{H}) = \dim(V_{S'} \cap X_{\leq \ell})
    \end{equation}
    and $\cos( \theta_{\max}( \widetilde{H}, V_{S'} \cap  X_{\leq \ell})) \geq \epsilon$. The prior inequality implies, by Lemma \ref{lem:mpa1},
    \begin{equation}\label{eqqq3}
    |\Pi_{ \leq \ell} x| \geq | \Pi_{ X_{\leq \ell} \cap V_{S'}} x| \geq \epsilon |x| \;\mbox{ for all } x \in \widetilde{H}.
    \end{equation}
    By the Pythagorean theorem and \eqref{eqqq3},
    \begin{equation}\label{eqqq4}
    |\Pi_{>\ell} x| = \sqrt{|x|^2 - |\Pi_{ \leq \ell} x|^2 } \leq \sqrt{1-\epsilon^2} |x| \;\mbox{ for all } x \in \widetilde{H}.
    \end{equation}

    By the form of $\tau_\delta$ (see \eqref{eqn:assump3.2}) and because $\delta \geq 1$, we have $\tau_\delta|_{X_{\leq \ell}} \geq \delta^{-\nu_\ell} \cdot \mathrm{id}|_{X_{\leq \ell}}$ and $\tau_\delta|_{X_{> \ell}} \leq \delta^{-\nu_\ell-1} \cdot \mathrm{id}|_{X_{> \ell}}$. Therefore, for $x \in \widetilde{H}$,  \eqref{eqqq3} gives
    \[
        |\tau_{\delta} \Pi_{ \leq \ell} x| \geq \delta^{-\nu_\ell} |\Pi_{ \leq \ell} x| \geq \delta^{-\nu_\ell} \epsilon |x|,
    \]
    and \eqref{eqqq4} gives
    \[
        |\tau_{\delta} \Pi_{> \ell} x | \leq \delta^{-\nu_\ell-1} |\Pi_{> \ell} x| \leq \delta^{-\nu_\ell-1} \sqrt{1 - \epsilon^2} |x|.
    \]
    Because $\tau_{\delta}$ fixes $X_{> \ell}$ and $X_{\leq \ell}$, the operators $\tau_{\delta}$, $\Pi_{> \ell}$, $\Pi_{ \leq \ell}$ all commute. Thus, combining the above inequalities gives
    \begin{equation}\label{eqn:delta_bd}
        \frac{|\Pi_{>\ell} \tau_{\delta} x|}{|\tau_{\delta} x |} \leq \frac{|\tau_{\delta} \Pi_{>\ell} x |}{|\tau_{\delta} \Pi_{ \leq \ell} x|} \leq \frac{1}{\delta} \frac{\sqrt{1-\epsilon^2}}{\epsilon} < \epsilon \quad \mbox{for all } x \in \widetilde{H} \setminus \{0\}.
    \end{equation}
    where the last inequality uses the assumption that $\delta \geq \frac{1}{\epsilon^2}$.
    
    From \eqref{eqqq1}, \eqref{eqqq2}, \eqref{eqqq2a}, and the assumption $c_\ell(S) < c_\ell(S')$, we have that $\dim(\widetilde{H}) > \dim(V_S \cap X_{ \leq \ell})$. Thus, we can find an $x \in \widetilde{H} \cap (V_S \cap X_{ \leq \ell})^\perp$ with $x \neq 0$. Note that $(V_S \cap X_{\leq \ell})^\perp$ is spanned by a subcollection of the basis $\{e_j \}$, and each $e_j$ is an eigenvector of $\tau_\delta$. Thus, since $x \in (V_S \cap X_{\leq \ell})^\perp$, we have $\tau_{\delta} x \in (V_S \cap X_{\leq \ell})^\perp$. Therefore, due to the orthogonal decomposition $V_S = (V_S \cap X_{\leq \ell}) \oplus (V_S \cap X_{> \ell})$, we have $\Pi_{V_S} \tau_{\delta} x = \Pi_{V_S \cap X_{> \ell}} \tau_\delta x$. We deduce that
    \[
    |\Pi_{V_S} \tau_{\delta} x| = |\Pi_{V_S \cap X_{> \ell}} \tau_\delta x| \leq | \Pi_{> \ell} \tau_\delta x|.
    \]
    Since $x \in \widetilde{H} \setminus \{0\}$,  \eqref{eqn:delta_bd} implies that $|\Pi_{> \ell} \tau_{\delta} x| < \epsilon | \tau_\delta x|$. Thus, $|\Pi_{V_S} \tau_{\delta} x| < \epsilon |\tau_{\delta} x|$ and $\tau_\delta x \in \tau_\delta \widetilde{H} \subseteq \tau_\delta H$, which implies $\cos (\theta_{\max}(\tau_{\delta} H, V_S)) < \epsilon$. This contradicts \eqref{eqn:bb1}, completing the proof of Lemma \ref{lem:multi2}.
\end{proof}

\subsubsection{Proof of Proposition \ref{prop:subspace_dil}}

 Fix $H \in G(k,X)$, $\eta \in (0,1/2)$.
 
 Fix a compact interval $I \subseteq (0,\infty)$ with $\frac{r(I)}{l(I)} \geq \left(\frac{2^d}{\eta} \right)^{dk+2}$. 

For ease of notation, let $H_\delta = \tau_\delta H$ and $\omega_\delta = \omega_{H_\delta} = \tau_\delta^* \omega_H$ for $\delta >0$.

We aim to show that there exist $\delta \in I$ and a $k$-dimensional dilation invariant subspace $V \subseteq X$ with $\cos (\theta_{\max}(H_\delta, V)) \geq 1 -\eta$.

Recall that every $k$-dimensional dilation invariant subspace $V \subseteq X$ has the form $V = V_S = \spn \{ e_j : j \in S \}$ for some $S \subseteq [d]$ with $\#(S) = k$, and that $\omega_S =  \bigwedge_{j \in S} e_j \in \bigwedge^k X$ is a representative form for $V_S$. 

By \eqref{eqn:def_angle1}, \eqref{eqn:ineq1}, it is enough to show that there exists $\delta \in I$ such that 
\begin{equation}\label{eqn:enough1}
   \cos(\angle( H_\delta, V_S)) = \frac{|\langle \omega_{\delta} , \omega_S \rangle|}{|\omega_\delta|} \geq 1 - \eta \quad \mbox{ for some } S \subseteq [d], \; \#(S) = k.
\end{equation}

Let 
\begin{equation}
    \label{eqn:def_eps}
\epsilon = \sqrt{\eta/2^d}.
\end{equation}

Observe that if $k  = \dim(H) = d$ then \eqref{eqn:enough1} is true with $S = [d]$ for any $\delta \in I$. That's because $\bigwedge^d X$ is one-dimensional, hence, $\omega_\delta \in \spn \{ \omega_{[d]} \}$.

We may thus assume $d \geq 2$ and $1 \leq k < d$.

We will then prove \eqref{eqn:enough1} by contradiction. For sake of contradiction, suppose \eqref{eqn:enough1} fails for every $\delta \in I$. 

Recall $\l(I)$ and $r(I)$ are the left and right endpoints of $I$. For $j \geq 0$, let $\delta_j := \epsilon^{-2j} \cdot l(I)$. Let $J$ be the largest positive integer such that $\delta_J \in I$. By assumption, $r(I)/l(I) \geq \left(\frac{2^d}{\eta} \right)^{dk+2} = \epsilon^{-2(dk+2)}$, thus
\begin{equation}\label{eqn:Jbd}
J \geq dk+2.
\end{equation}

For each $j=0,1,\dots,J$ we claim that there exist distinct subsets $S_{j,1}$ and $S_{j,2}$ of $[d]$ of cardinality $k$, such that
\begin{equation}\label{eqn:enough2}
    \frac{| \langle \omega_{\delta_j}, \omega_{S_{j,\mu}}  \rangle|}{|\omega_{\delta_j}|} \geq \epsilon \qquad (\mu=1,2).
\end{equation}
To see this, order the subsets of $[d]$ of cardinality $k$ in a sequence, $S_{j,1}, S_{j,2}, \dots,S_{j,L}$, $L = \binom{d}{k}$, so that
\[
\ell  \mapsto | \langle \omega_{\delta_j}, \omega_{S_{j,\ell}} \rangle| \mbox{ is non-increasing (for fixed } j \mbox{)}.
\]
Set $a_{j,\ell} := | \langle \omega_{\delta_j}, \omega_{S_{j,\ell}} \rangle|/|\omega_{\delta_j}|$ for $\ell=1,\dots,L$. By assumption, \eqref{eqn:enough1} fails for $\delta = \delta_j$, thus
\[
a_{j,\ell} < 1 - \eta \quad \mbox{for all } \ell=1,\dots,L.
\]
Because $\{ \omega_S : S \subseteq [d], \; \#(S) = k \} = \{\omega_{S_{j,\ell}} : \ell=1,\dots, L \}$ is an orthonormal basis for $\bigwedge^k X$, we have  $\sum_{\ell=1}^L a_{j,\ell}^2 = 1$. Since $\ell \mapsto a_{j,\ell}$ is non-increasing,
\[
a_{j,1} \geq \sqrt{1/L} \geq \sqrt{1/2^d} > \epsilon.
\]
Since $a_{j,1} \leq 1-\eta$, we have
\[
\sum_{\ell=2}^L a_{j,\ell}^2 = 1 - a_{j,1}^2 \geq 1 - (1 - \eta)^2 = 2\eta - \eta^2 \geq \eta.
\]
Thus, because $\ell \mapsto a_{j,\ell}$ is non-increasing, we have
\[
a_{j,2} \geq \sqrt{\eta/(L-1)} > \sqrt{\eta/L} > \sqrt{\eta/2^d} = \epsilon.
\]
As $a_{j,1}, a_{j,2} \geq \epsilon$, we complete the proof of \eqref{eqn:enough2}.

Let $\mu_0 = 1$, and for $1 \leq j \leq J$, let $\mu_j \in \{1,2\}$ be such that $S_{j,\mu_j} \neq S_{j-1,\mu_{j-1}}$. 
By definition, note that $\delta_{j} = \delta_{j-1}/\epsilon^2$ for $j \geq 1$. Thus, using \eqref{eqn:enough2}, we may apply Lemma \ref{lem:multi2} to deduce that $c_\ell(S_{j,\mu_j}) \geq c_\ell(S_{j-1,\mu_{j-1}})$ for every $\ell=1,2,\dots,d$ and $j=1,2,\dots, J$. Further, since $S_{j,\mu_j} \neq S_{j-1,\mu_{j-1}}$, for each $j$ this inequality is strict for some $\ell$ (see \eqref{prop:c_ell}). It follows that
\begin{equation}\label{eqn:mon1}
    \psi_j := \sum\limits_{\ell=1}^d c_\ell(S_{j,\mu_j}) > \psi_{j-1} \qquad (j =1,2,\dots,J).
\end{equation}
But note that 
\[
0 \leq c_\ell(S) = \# \{ j \in S : j \leq \ell\} \leq  k
\]
for all $\ell = 1,2,\dots, d$ and all $S \subseteq [d]$ with $\#(S)=k$. Thus,
\[
0 \leq \psi_j \leq dk \qquad (j=1,2,\dots,J).
\]
From this and \eqref{eqn:mon1} we deduce that $J \leq dk+1$. But this contradicts \eqref{eqn:Jbd}. This completes the proof of \eqref{eqn:enough1} and finishes the proof of Proposition \ref{prop:subspace_dil}.

\subsection{Monotonicity of the orbits of a Hilbert dilation system on the Grassmanian.}

Fix a Hilbert dilation system $\mathcal{X} = (X,\tau_\delta)_{\delta >0}$. We drop the assumption that $\mathcal{X}$ is simple. Thus, $X = \bigoplus_{\nu=1}^m X_\nu$ and $\tau_\delta : X \rightarrow X$ has the form $\tau_\delta|_{X_\nu} = \delta^{-\nu} \cdot \mathrm{id}|_{X_\nu}$, as in \eqref{eqn1:sec3}, \eqref{eqn2:sec3}. Our next result describes a qualitative property of the functions $\delta \mapsto \angle(\tau_\delta H,V)$ (for fixed $H,V$) that will enter into the proof of Proposition \ref{prop:tech2}. See \eqref{eqn:def_angle1} for the definition of the quantity $\angle(V,W)$.

\begin{lemma}\label{lem:multi3} Let $H, V \subseteq X$ be subspaces with $\dim(H) = \dim(V) \geq 1$,  such that $V$ is dilation invariant. Then the map $f(\delta) = \cos(\angle(\tau_\delta H, V))$ is unimodal: if $a<b<c$ and $f(b)<f(c)$, then $f(a)<f(b)$. 
\end{lemma}
\begin{proof}
    Let $l = \dim(H) = \dim(V) \geq 1$. Fix representative forms $\omega_H, \omega_{V} \in \bigwedge^l X$ for $H$, $V$, respectively, with unit norm, $|\omega_H| = |\omega_V| = 1$. Then $\tau_\delta^* \omega_H$ is a representative form for $\tau_\delta H$. So we have
    \begin{equation} \label{eq:frac}
       f(\delta) = \cos(\angle(\tau_\delta H,V)) = \frac{|\langle \tau_\delta^* \omega_H , \omega_{V} \rangle|}{|\tau_\delta^* \omega_H| |\omega_{V}|} = \frac{| \langle \omega_H , \tau_\delta^* \omega_{V} \rangle|}{|\tau_\delta^* \omega_H |}.
    \end{equation}
    Since $V$ is dilation invariant, $\tau_\delta^* \omega_{V} = \delta^{- \sgn(V)} \omega_V$ (see \eqref{eq:DI_id}). So the numerator of \eqref{eq:frac} is
    \begin{equation}\label{eq:num}
    | \langle \omega_H , \tau_\delta^* \omega_{V} \rangle| = \alpha \cdot \delta^{-\sgn(V)} \mbox{, for } \alpha := |\langle \omega_H, \omega_V \rangle| \geq 0.
    \end{equation}
    
    To compute the denominator of \eqref{eq:frac}, we fix a basis for $\bigwedge^l X$. Fix a family of dilation invariant subspaces $U_1, U_2, \dots,U_M$, such that the associated unit-norm representative forms $\omega_{U_1}, \omega_{U_2}, \dots,\omega_{U_M}$ give an orthonormal basis for $\bigwedge^l X$ ($M = \binom{d}{l}$). Then $\tau_\delta^*\omega_{U_i} = \delta^{- \sgn(U_i)} \omega_{U_i}$ by \eqref{eq:DI_id}. So the denominator of \eqref{eq:frac} is
    \begin{equation}\label{eq:den}
    \begin{aligned}
        &| \tau_\delta^* \omega_H | = \sqrt{\sum_{i=1}^M \langle \tau_{\delta}^*\omega_H, \omega_{U_i} \rangle^2} = \sqrt{\sum_{i=1}^M \langle \omega_H, \tau_{\delta}^*\omega_{U_i} \rangle^2}\\
        &=\sqrt{\sum_{i=1}^M \langle \omega_H , \omega_{U_i} \rangle^2 \delta^{-2 \sgn(U_i)}} = \sqrt{\sum_{p=1}^{dm} \alpha_p \delta^{-2p}}
    \end{aligned}
    \end{equation}
    for constants
    \[
    \alpha_p = \sum_{i \in [M], \sgn(U_i) = p} \langle \omega_H, \omega_{U_i} \rangle^2 \geq 0, \qquad  1 \leq p \leq dm.
    \]
    Here, we used that $1 \leq \sgn(U) \leq dm$ for any dilation invariant subspace $U \subseteq X$ with $\dim(U) \geq 1$. Not all of the coefficients $\alpha_p$ are equal to zero, because $\{ \omega_{U_i} \}$ is an orthonormal basis for $\bigwedge^\ell X$. Combining \eqref{eq:num} and \eqref{eq:den}, we have
    \begin{equation}\label{eq:func}
    f(\delta) = \frac{\alpha \delta^{- \sgn(V)}}{\sqrt{\sum\limits_{p=1}^{dm} \alpha_p \delta^{-2p}}}.
    \end{equation}
    
    If $\alpha = 0$ then $f= 0$, and we obtain the desired conclusion because constant functions are unimodal.
    
    Now suppose $\alpha > 0$, so that $f(\delta) > 0$ for all $\delta$. Define $g(\delta) =  \log(f(\delta^{-1}))$. Then compute
    \[
    \begin{aligned}
        &g'(\delta) = \sgn(V)\frac{1}{\delta} - \frac{\sum\limits_{p=1}^{dm} p \alpha_p \delta^{2p-1}}{\sum\limits_{p=1}^{dm} \alpha_p \delta^{2p}} = \frac{P(\delta)}{\sum\limits_{p=1}^{dm} \alpha_p \delta^{2p}}, \\
        &P(\delta) = \sum\limits_{p=1}^{dm} \alpha_p ( \sgn(V) - p) \delta^{2p-1}.
    \end{aligned}
    \]
    We now split the proof into two cases.
    
    Case 1: $\alpha_p = 0$ for all $p \neq \sgn(V)$. Then from the above identities, $P \equiv 0$, and so $g' \equiv 0$. So $g$ is constant, and thus $f$ is constant, giving the desired result.
    
    Case 2: $\alpha_p \neq 0$ for some $p \neq \sgn(V)$. If there exist $r,q$ with $\alpha_r>0$, $\alpha_q >0$, $r < \sgn(V) < q$, then the signs of the coefficients of $P(\delta)$ change exactly once; otherwise they change $0$ times. By Descartes' rule of signs, there is at most one value of  $\delta > 0$ with $P(\delta) = 0$, so at most one value of $\delta > 0$ with $g'(\delta) = 0$. This leaves three options: $g$ is monotone, $g$ has one interior maximum and no interior minima, and $g$ has one interior minimum and no interior maxima. The first two options imply that $g$ is unimodal, hence $f$ is unimodal. The third option is impossible. To see this, we exploit the assumption that $\alpha_p \neq 0$ for some $p \neq \sgn(V)$. Therefore, from \eqref{eq:func}, either $\lim_{\delta \rightarrow \infty} f(\delta) = 0$ or $\lim_{\delta \rightarrow 0} f(\delta) = 0$. Therefore, $g(\delta) \to -\infty$ for at least one of $\delta \to 0$ or $\delta \to \infty$, ruling out that $g$ has one interior minimum and no interior maxima.
    
    This completes the proof of Lemma \ref{lem:multi3}.
\end{proof}

\subsection{Rescaling dynamics on the space of ellipsoids}\label{sec:rescale_ellip}

We present further preparatory results to be used in the proofs of  Propositions \ref{prop:tech1} and \ref{prop:tech2}.

Let $(X,\tau_\delta)_{\delta > 0}$ be a Hilbert dilation system. So, $X$ is a real Hilbert space of dimension $d$ and  $\tau_\delta : X \rightarrow X$ are linear operators of the form \eqref{eqn1:sec3},\eqref{eqn2:sec3}. 

Given a set $\Omega \subseteq X$ and $T : X \rightarrow X$, we denote $T \Omega = \{ T(x) : x \in \Omega\}$.

A (centered) \emph{ellipsoid} $\E \subseteq X$ is a set of the form
\begin{equation}\label{eqn:ellips1}
\E = \left\{ \sum_{i=1}^d c_i \sigma_i v_i : \sum_{i=1}^d c_i^2 \leq 1 \right\},
\end{equation}
where $\sigma_1 \geq  \dots \geq \sigma_d \geq 0$ and $\{v_1,\dots,v_d\}$ is an orthonormal basis for $X$. We call $v_1,\dots,v_d$ (normalized) principal axis directions of $\E$, and $\sigma_1,  \dots,\sigma_d$ the principal axis lengths of $\E$. Denote $\sigma_j(\E) = \sigma_j$ for the $j$'th principal axis length of $\E$. Principal axis lengths (but not directions) are uniquely determined by $\E$.

Note that the intersection of an ellipsoid and a subspace is also an ellipsoid. Further, the image of an ellipsoid under a linear transformation is an ellipsoid.

Let $\cB$ be the closed unit ball of $X$. If $A : X \rightarrow X$ is a linear transformation then $A \cB$ is an ellipsoid in $X$. Let $ \sigma_1 \geq  \dots \geq \sigma_d \geq 0$ be the singular values of $A$, let $\{ v_1,\dots, v_d \}$ be left singular vectors of $A$, and let $\{w_1,\dots,w_d\}$ be right singular vectors of $A$. That is, $\{ v_i\}$ and $\{w_i\}$ are orthonormal bases for $X$, and $A w_i = \sigma_i v_i$ for all $i$. We express $\cB$ in the form $\{ \sum_i c_i w_i : \sum_i c_i^2 \leq 1 \}$. Then 
\begin{equation} \label{eqn:ellips2}
  A \B = \left\{ \sum_{i=1}^d c_i \sigma_i v_i : \sum_{i=1}^d c_i^2 \leq 1 \right\}.  
\end{equation}
So,  the principal axis lengths of $A \B$ are the singular values of $A$, and the principal axis directions of $A \cB$ are corresponding left singular vectors of $A$.

In particular, every ellipsoid $\E$ can be written as $\E = A \B$ for some linear transformation $A : X \rightarrow X$.

Given an ellipsoid $\E \subseteq X$, let $\E_\delta := \tau_\delta \E$ for $\delta > 0$. Then $(\E_\delta)_{\delta >0}$ is an \emph{orbit} of $\tau_\delta$ in the space of ellipsoids. Our next result, Lemma \ref{lem:help2}, states that this orbit can be approximated by an orbit in the Grassmanian $G(k,X)$ if a condition on the $\E_\delta$ is met.

\begin{defi}
Let $\epsilon \in (0,1/2)$, and let $\E \subseteq X$ be an ellipoid. Say that $\E$ is $\epsilon$-degenerate if $\sigma_j(\E) \notin [\epsilon, \epsilon^{-1}]$ for all $j$. In other words, $\E$ is $\epsilon$-degenerate if the length of every principal axis of $\E$ is either less than $\epsilon$ or greater than $\epsilon^{-1}$. 
\end{defi}

\begin{lemma}\label{lem:help2}
    Let $\E$ be an ellipsoid in $X$, let $\epsilon \in (0,1/2)$, and let $I \subseteq (0,\infty)$ be a compact interval. Let $\E_\delta := \tau_\delta \E$ for $\delta > 0$. Suppose that $\E_\delta$ is $\epsilon$-degenerate for all $\delta \in I$. Then there exists a subspace $H \subseteq X$ such that, for all $\delta \in I$,
    \begin{enumerate}
        \item[(a)] $\E_\delta \subseteq   \tau_\delta H + \epsilon \B$, and
        \item[(b)] $\tau_\delta H \cap (\frac{1}{2 \epsilon} \B) \subseteq \E_\delta$.
    \end{enumerate}
\end{lemma}

\begin{proof} 
By rescaling, we may assume that $I$ has the form $I=[1, T]$ for $T \geq 1$. Write $\E = A \B$ for a linear transformation $A : X \rightarrow X$. Then $\E_\delta = A_\delta \B$, with $A_\delta := \tau_\delta A$. For $\delta > 0$, consider the singular values of $A_\delta$:
\begin{equation}\label{eqn:sv2}
 \sigma_1(\delta) \geq \sigma_2(\delta) \geq \dots \geq \sigma_d(\delta) \geq 0,
\end{equation}
and let $\{ v_1(\delta),\dots,v_d(\delta) \}$ be the associated left singular vectors of $A_\delta$, which form an orthonormal basis for $X$. By \eqref{eqn:ellips2}, 
\begin{equation}\label{eqn:ellips3}
\E_\delta = \left\{ \sum_{i=1}^d c_i \sigma_i(\delta) v_i(\delta) : \sum_{i=1}^d c_i^2 \leq 1 \right\}.
\end{equation}
The singular values of $A_\delta$ are the square roots of eigenvalues of $A_\delta A_\delta^*$:
\[
\sigma_j(\delta)  = \sqrt{\lambda_j(A_\delta A_\delta^*)} = \sqrt{\lambda_j(\tau_\delta A A^* \tau_\delta)}.
\]
The ordered tuple of eigenvalues $(\lambda_1(B),\dots,\lambda_d(B)) \in \R^d$ of a symmetric matrix $B \in \R^{d \times d}$ is a continuous function of the entries of $B$. It follows that $\delta \mapsto \sigma_j(\delta)$ is continuous for each $j$. By the intermediate value theorem, and the assumption that $\E_\delta$ is $\epsilon$-degenerate for each $\delta \in I$, there exists  $k \in \{0,1,\dots, d\}$ so that
\begin{align}\label{eqn:lem4a}
&\sigma_j(\delta) > \epsilon^{-1} \mbox{ for } 1 \leq j \leq k,\\
\label{eqn:lem4b}
& \sigma_{j}(\delta) < \epsilon \mbox{ for } k < j \leq d \quad (\mbox{all } \delta \in I).
\end{align}
Let $H = \spn\{ v_j(1) : 1 \leq j \leq k \}$. Thus, $H \in G(k,X)$ is spanned by the $k$ longest principle axes of $\E_1$, and $H^\perp = \spn\{ v_j(1) : k < j \leq d \}$ is spanned by the $(d-k)$ shortest principal axes of $\E_1$. 

Evidently, by \eqref{eqn:ellips3}, $\Pi_H \E_1 = \E_1 \cap H$ and $\Pi_{H^\perp} \E_1 = \E_1 \cap H^\perp$. By the second identity, a general element $x$ of $\Pi_{H^\perp} \E_1$ has the form $x = \sum_{i > k} c_i \sigma_i(1) v_i(1)$, for coefficients $c_i$ with $\sum_i c_i^2 \leq 1$.
By  \eqref{eqn:lem4b} for $\delta = 1$, the fact that $|v_j(1)| = 1$ for all $j$, and the Pythagorean theorem, we deduce that $|x| \leq \epsilon$ for any $x \in \Pi_{H^\perp} \E_1$. Thus, $\Pi_{H^\perp} \E_1 \subseteq \epsilon \cB$. Thus, given that $\E = \E_1$, we obtain 
\[
\E \subseteq \Pi_H \E + \Pi_{H^\perp} \E = (\E \cap H) + (\E \cap H^\perp) \subseteq (\E \cap H) + \epsilon \cB.
\]
Thus, for $\delta \geq 1$, 
    \begin{equation}\label{eqn:cont0}
        \tau_\delta \E \subseteq \tau_{\delta} ( (\E \cap  H) + \epsilon \B) =  \tau_\delta \E \cap \tau_\delta H + \epsilon \tau_{\delta} \B \subseteq  \tau_\delta \E \cap \tau_\delta H + \epsilon \B,
    \end{equation}
    where the last inclusion uses that $\| \tau_{\delta}\|_{op} \leq 1$ for $\delta \geq 1$.
    
    Note that \eqref{eqn:cont0} implies $\tau_\delta \E  \subseteq \tau_\delta H + \epsilon \cB$ for $\delta \geq 1$. This implies (a).
    
    We next establish (b). For contradiction, suppose there exists $\delta \in [1,T]$ with
    \begin{equation}\label{eqn:cont1}
        \tau_\delta H \cap \left((2 \epsilon)^{-1} \B \right) \not\subseteq \tau_\delta \E.
    \end{equation}
     We regard $\tau_\delta \E \cap \tau_\delta H$ as an ellipsoid in the vector space $\tau_\delta H$. Let $\sigma \geq 0$ be the shortest principal axis length of $\tau_\delta \E \cap \tau_\delta H$ in $\tau_\delta H$, and let $v \in \tau_\delta H$ be an associated unit-norm principal axis direction. Then $\pm \sigma  v \in \tau_\delta \E \cap \tau_\delta H$, and by \eqref{eqn:cont1}, $\sigma<\frac{1}{2 \epsilon}$. Thus, if $U := \tau_\delta H \cap v^\perp$, then $\tau_\delta \E \cap \tau_\delta H \subseteq U + \frac{1}{2 \epsilon} \B$. By \eqref{eqn:cont0}, 
    \begin{equation}\label{eqn:lem4c}
    \tau_\delta \E \subseteq (\tau_\delta \E \cap \tau_\delta H) + \epsilon \cB  \subseteq U + \left((2\epsilon)^{-1}+ \epsilon \right) \B \subseteq U + (3/4) \epsilon^{-1} \B.
    \end{equation}

    Given $\dim(\tau_\delta H)=k$, and $U $ has codimension $1$ in $\tau_\delta H$, then $\dim(U)=k-1$. From \eqref{eqn:ellips3} and \eqref{eqn:lem4a}, $\tau_\delta \E$ contains a $k$-dimensional disk of radius $\epsilon^{-1}$. Together with \eqref{eqn:lem4c}, these remarks lead to a contradiction.
\end{proof}

Let $\E$ be an ellipsoid in $X$. The next lemma guarantees that $\tau_\delta \E$ is $\epsilon$-degenerate for ``most'' $\delta \in (0,\infty)$. We write $r(I)$ and $l(I)$ for the right and left endpoints of an interval $I \subseteq (0,\infty)$, respectively.

\begin{lemma}\label{lem:help3}
    
    Let $d = \dim X$. Let $\E \subseteq X$ be an ellipsoid and let $\epsilon \in (0,1/2)$. There exists a collection of closed intervals $J_1,J_2,\dots,J_d \subseteq (0,\infty)$ such that $\tau_\delta \E$ is $\epsilon$-degenerate for all $\delta \notin \bigcup_{p=1}^d J_p$, and such that $r(J_p)/l(J_p) \leq \frac{1}{\epsilon^2}$  for all $p$.
\end{lemma}

\begin{proof}
Write $\E = A \B$ for a linear transformation $A : X \rightarrow X$. For $\delta > 0$, let $\E_\delta = \tau_\delta \E = A_\delta \B$, with $A_\delta = \tau_{\delta} A$. Let $\sigma_1(\delta) \geq \sigma_2(\delta) \geq \dots \geq \sigma_d(\delta) \geq 0$ be the principal axis lengths of $\E_\delta$, given by the singular values of $A_\delta$.

The $j$-th singular value $\sigma_j(\delta)$ of $A_\delta$ is given by $\sigma_j(\delta) = \sqrt{ \lambda_j(\delta)}$, where $\lambda_j(\delta)$ is the $j$-th eigenvalue of $A_\delta A_\delta^*$, i.e.,
\[
\lambda_j(\delta) = \lambda_j(\tau_\delta A A^* \tau_\delta).
\]
Here, we write the eigenvalues of $A_\delta A_\delta^*$ in decreasing order, $\lambda_1(\delta) \geq \lambda_2(\delta) \geq \dots \geq \lambda_d(\delta) \geq 0$, for each $\delta$. Let $\delta_* > 0$. We claim that
\begin{equation}\label{eqn:claim1}
\lambda_j(\delta) \leq \left(\delta_*/\delta \right)^2 \cdot \lambda_j(\delta_*) \quad  \mbox{for} \; j=1,2,\dots,d, \;\; \delta \geq \delta_*.
\end{equation}
Using that $A_\delta = (A_{\delta_*})_{\delta/\delta_*}$, we make the substitution $A \leftarrow A_{\delta_*}$ and $\delta \leftarrow \delta/\delta_*$ and reduce the proof of \eqref{eqn:claim1} to the case $\delta_* = 1$. By the min-max characterization of eigenvalues, for any $\delta \geq 1$, with $B = A A^*$, we have
    \[
    \begin{aligned}
        \lambda_{j}(\delta) &= \sup\limits_{V \in G(j,X)} \inf\limits_{x \in V \setminus \{0\}} \frac{\langle \tau_{\delta} B \tau_{\delta} x , x \rangle}{|x|^2}\\
    &= \sup\limits_{V \in G(j,X)} \inf\limits_{x \in V \setminus \{0\}} \frac{|\tau_{\delta} x|^2}{|x|^2} \frac{\langle B \tau_{\delta} x , \tau_{\delta} x \rangle}{|\tau_{\delta} x|^2} \\
    &= \sup\limits_{\hat{V} \in G(j,X)} \inf\limits_{\hat{x} \in \hat{V} \setminus \{0\}} \frac{|\hat{x}|^2}{|\tau_{\delta^{-1}} \hat{x}|^2} \frac{\langle B \hat{x} , \hat{x} \rangle}{|\hat{x}|^2} \leq \delta^{-2} \lambda_{j}(1).
    \end{aligned}
    \]
The last equality above makes use of the substitution $\hat{V} = \tau_{\delta} V$ and $\hat{x} = \tau_{\delta}x$. The last inequality holds because $|\tau_a y| \geq a^{-1} |y|$ for $a \leq 1$, and by the min-max characterization of the eigenvalue $\lambda_j(1) = \lambda_j(B)$. We have proven \eqref{eqn:claim1}.

Let $J_p$ be the closure of the set $\{\delta \in (0,\infty) : \sigma_p(\delta) \in [\epsilon,\epsilon^{-1}]\}$ for $p=1,2,\dots,d$. From \eqref{eqn:claim1} and  $\sigma_p(\delta) = \sqrt{\lambda_p(\delta)}$, we have $\sigma_p(\delta) \leq (\delta_*/\delta) \sigma_p(\delta_*)$ for $\delta \geq \delta_*$. Thus, $\sigma_p$ is a decreasing function of $\delta$, and if $\delta > \epsilon^{-2} \delta_*$ then $\sigma_p(\delta) < \epsilon^2 \sigma_p(\delta_*)$. It follows that $J_p$ is an interval and $r(J_p)/l(J_p) \leq \epsilon^{-2}$.

Finally, note, for $\delta \notin \bigcup_p J_p$, that $\sigma_p(\delta) \notin [\epsilon,\epsilon^{-1}]$ for all $p$ (by definition of the intervals $J_p$), thus, $\tau_\delta \E$ is $\epsilon$-degenerate.
\end{proof}

\subsection{Complexity}\label{sec:comp}

Given a Hilbert space $X$, we let $\K(X)$ denote the collection of all closed, convex, symmetric subsets of $X$. Let $\B \in \cK(X)$ denote the unit ball of $X$. Given $\Omega \in \K(X)$, $V$ a subspace of $X$, and $R \geq 1$, recall that $\Omega$ is $R$-transverse to $V$ if (a) $\Omega \cap V \subseteq R \B$, and (b) $\Pi_{V^\perp}(\Omega \cap \B) \supseteq R^{-1} \B \cap V^\perp$ (see Definition \ref{def:trans_hilbert}). 

For an interval $I$, let $l(I)$ and $r(I)$ denote the left and right endpoints of $I$, respectively. We say $I>J$ if $l(I) > r(J)$, and $I>0$ if $l(I)>0$. 

\begin{defi}\label{defn:comp2}
    Let $\cX = (X,\tau_\delta)_{\delta > 0}$ be a Hilbert dilation system. For $\Omega \in \K(X)$, $R \in [1,\infty)$, $R^{*} \in (R,\infty)$, the {\em complexity} of $\Omega$ with respect to $\cX$ with parameters $(R,R^*)$, written $\C_{\cX}(\Omega,R,R^{*}) = \C(\Omega,R,R^*)$, is the largest positive integer $K$ such that there exist compact intervals $I_1 > I_2 > \dots > I_K > 0$ in $(0,\infty)$ and dilation invariant subspaces $V_1,V_2,\dots,V_K \subseteq X$ such that, for every $j$, $\tau_{r(I_j)} \Omega$ is $R$-transverse to $V_j$, and $\tau_{l(I_j)} \Omega$ is not $R^{*}$-transverse to $V_j$.
\end{defi}

Fix a Hilbert dilation system $(X,\tau_\delta)_{\delta > 0}$. Thus, $X = \bigoplus_{\nu=1}^m X_\nu$ and $\tau_\delta : X \rightarrow X$ is given by $\tau_\delta|_{X_\nu} = \delta^{-\nu} \mathrm{id}|_{X_\nu}$. Let $d := \dim(X)$. Let $V$ be a dilation-invariant (DI) subspace of $X$ (see Definition \ref{defi:di_subspace}). Then $V$ has the form
\[
    V = \bigoplus\limits_{\nu=1}^m V \cap X_\nu.
\]
Recall that the \emph{signature} of $V$ is defined by $\sgn(V) = \sum\limits_{\nu=1}^m \nu \cdot \dim (V \cap X_\nu)$. Note that $0 \leq \sgn(V) \leq md$ for any dilation-invariant subspace $V$.

If $\Omega_1, \Omega_2 \in \K(X)$ satisfy $\lambda^{-1}\Omega_2 \subseteq  \Omega_1 \subseteq \lambda \Omega_2$ for $\lambda \geq 1$ then we say that $\Omega_1$ and $\Omega_2$ are $\lambda$-equivalent, and we write $\Omega_1 \sim_\lambda \Omega_2$.

We now rephrase a classical theorem of F. John (see \cite{ball}) in terms of the definitions just provided.

\begin{proposition}[John's theorem]\label{prop:john}
Given a compact $\Omega \in \cK(X)$, there exists an ellipsoid $\E \subseteq X$ such that $\Omega$ and $\E$ are $\sqrt{d}$-equivalent.
\end{proposition}

\begin{rmk}\label{rem:1}
If $\Omega_1$ is $R$-transverse to $V$ and $\Omega_1 \sim_\lambda \Omega_2$ then $\Omega_2$ is $\lambda R$-transverse to $V$. It follows that if $\Omega_1 \sim_\lambda \Omega_2$ then $\C(\Omega_1,R,R^*) \leq \C(\Omega_2,\lambda R, \lambda^{-1}R^*)$ provided that $R^* > \lambda^2 R$ so the right-hand-side is well-defined.
\end{rmk}

\begin{lemma}\label{lem:stability} 
Fix $\xi > R \geq 1$. Suppose $\Omega_1$ is $R$-transverse to $V$ and $\Omega_1 \cap \xi \B =  \Omega_2 \cap \xi \B$. Then $\Omega_2$ is $R$-transverse to $V$.
\end{lemma}
\begin{proof}
Given that $\Omega_1$ is $R$-transverse to $V$ and $\Omega_1 \cap \xi \B =  \Omega_2 \cap \xi \B$, we have
\[
\Omega_2 \cap V \cap \xi \B = \Omega_1 \cap V \cap \xi \B \subseteq R \B.
\]
Since $\xi > R$, we deduce that $\Omega_2 \cap V \subseteq R \B$. 

Since $\Omega_1 \cap \xi \B =  \Omega_2 \cap \xi \B$ for $\xi > 1$, we have $\Omega_1 \cap \B = \Omega_2 \cap \B$, thus
\[
R^{-1} \B \cap V^\perp \subseteq \Pi_{V^\perp} ( \Omega_1 \cap \B) = \Pi_{V^\perp} ( \Omega_2 \cap \B).
\]
So, $\Omega_2$ is $R$-transverse to $V$.
\end{proof}

The remainder of this section is devoted to the proof of the next result.

\begin{proposition} \label{prop:comp_bound}
   For any $\Omega \in \K(X)$, $\C(\Omega,R_1,R_2) \leq 4 m d^2$
   provided that $R_1 \geq 16$ and $R_2 \geq  \max \{ (\sqrt{d})^{4m+1} R_1^{4m}, (\sqrt{d})^{3d+1} R_1^{3d}\}$.
\end{proposition}

Using John's theorem, we shall reduce Proposition \ref{prop:comp_bound} to the following:
\begin{proposition}\label{prop:main_goal} For any ellipsoid $\E \subseteq X$, $R \geq 16$ and $R^* \geq \max \{R^{4m}, R^{3d}\}$,
\[
    \C(\E,R,R^*) \leq 4md^2.
\]
\end{proposition}
We will later give details on the reduction of Proposition \ref{prop:comp_bound} to Proposition \ref{prop:main_goal}. Next we make preparations for the proof of Proposition \ref{prop:main_goal}.  Fix $R, R^*$ and $\epsilon > 0$ such that
\begin{equation}\label{eqn:const_assump}
\begin{aligned}
    &16 \leq R \leq \max\{R^{3d},R^{4m}\} \leq R^{*},\\
    &\epsilon \leq 1/(4R) \mbox{ and } R/R^* \leq \epsilon^{2m}.
\end{aligned}
\end{equation}
Note \eqref{eqn:const_assump} is satisfied if $\epsilon = \frac{1}{4R}$, as then $\frac{R}{R^*} \leq R^{1-4m} \leq R^{-3m} \leq (4R)^{-2m} = \epsilon^{2m}$.

The following result is the key ingredient in the proof of Proposition \ref{prop:main_goal}.

\begin{proposition}\label{prop:comp_mon}
   Let $R$, $R^*$, $\epsilon$ be as in \eqref{eqn:const_assump}. Let $\E$ be an ellipsoid in $X$, and let $I = [\delta_{\min}, \delta_{\max}] \subseteq (0,\infty)$. Suppose that $\tau_\delta \E$ is $\epsilon$-degenerate for all $\delta \in I$. 
   
   If there exist $\delta_* \in I$  and dilation invariant subspaces $V, W \subseteq X$ such that
\begin{enumerate}
    \item $\tau_{\delta_{\max}}\E$ is $R$-transverse to $V,$
    \item $\tau_{\delta_*} \E$ is not $R^{*}$-transverse to $V,$ and
    \item $\tau_{\delta_{\min}} \E$ is $R$-transverse to $W$,
\end{enumerate}
then $\sgn(V) > \sgn(W)$.
\end{proposition}

Before the proof of Proposition \ref{prop:comp_mon}, we present two preparatory lemmas.

\begin{lemma}\label{lem:AKT}
If $A,K,T \in \K(X)$, and $K \subseteq T$, then $(A+K) \cap T \subseteq ( A \cap 2T) + K$.
\end{lemma}
\begin{proof}
Fix $x \in (A+K) \cap T$. Then $x = a + k$ for $a \in A$, $k \in K$. Note that $a = x-k \in T + K \subseteq 2T$. Hence, $x = a + k \in (A \cap 2T) + K$.
\end{proof}

\begin{lemma}\label{lem:help4}
    Under the hypotheses of Proposition \ref{prop:comp_mon}, there exists a subspace $H \subseteq X$ such that Conditions 1,2,3 of Proposition \ref{prop:comp_mon} hold with $H$ and $4R$ in place of $\E$ and $R$, respectively.
\end{lemma}

\begin{proof}[Proof of Lemma \ref{lem:help4}]
    By  Lemma \ref{lem:help2}, there exists a subspace $H \subseteq X$ such that for all $\delta \in I$,
    \begin{enumerate}
    \item[(a)] $\tau_\delta \E \subseteq \tau_\delta H + \epsilon \B$
    \item[(b)] $\tau_\delta H \cap ( \frac{1}{2 \epsilon} \B) \subseteq \tau_\delta \E$.
    \end{enumerate}
    Using (b) for $\delta = \delta_{\max}$, the inequality $R \leq \frac{1}{4 \epsilon}$ (see \eqref{eqn:const_assump}), and the condition that $\tau_{\delta_{\max}}\E$ is $R$-transverse to $V$,
    \[
        (\tau_{\delta_{\max}} H \cap (2 R \B)) \cap V \subseteq \tau_{\delta_{\max}}\E \cap V \subseteq R \B,
    \]
    which implies that $\tau_{\delta_{\max}}H \cap V \subseteq R \B$.
    
    Using the condition that $\tau_{\delta_{\max}}\E$ is $R$-transverse to $V$, and (a) for $\delta = \delta_{\max}$, 
    \[
    \begin{aligned}
        R^{-1} \B \cap V^\perp &\subseteq \Pi_{V^\perp} (\tau_{\delta_{\max}}\E \cap \B)) \\
        &\subseteq \Pi_{V^\perp}((\tau_{\delta_{\max}}H + \epsilon \B) \cap \B)) \\
        &\subseteq \Pi_{V^\perp}( 2 (\tau_{\delta_{\max}}H \cap \B) + \epsilon \B),
    \end{aligned}
    \]
    where we used Lemma \ref{lem:AKT} for the last inclusion. Because $\epsilon \leq \frac{1}{2R}$ and $\Pi_{V^\perp} \cB = \cB \cap V^\perp$, it follows that
    \[
    R^{-1} \cB \cap V^\perp \subseteq 2 \Pi_{V^\perp} (\tau_{\delta_{\max}}H \cap \B) + (1/2) R^{-1} \B \cap V^\perp.
    \]
    We deduce that $\frac{1}{4} R^{-1} \B \cap V^\perp \subseteq \Pi_{V^\perp}(\tau_{\delta_{\max}} H \cap \B)$.
    
    Therefore, we see that $\tau_{\delta_{\max}}H$ is $4R$-transverse to $V$. 
    
  Repeating the previous argument, using that $\tau_{\delta_{\min}} \E$ is $R$-transverse to $W$, and (a), (b) for $\delta = \delta_{\min}$, we see that $\tau_{\delta_{\min}} H$ is $4R$-transverse to $W$.
    
    Assume for sake of contradiction that $\tau_{\delta_*} H$ is $R^{*}$-transverse to $V$.  By Lemma \ref{lem:basic_trans_result}, $(\tau_{\delta_*} H + \epsilon \B) \cap V \subseteq R^{*} \epsilon \B \subseteq R^{*} \B$. Thus, by condition (a) for $\delta = \delta_*$, 
    \[
        \tau_{\delta_*}\E \cap V \subseteq (\tau_{\delta_*} H + \epsilon \B) \cap V \subseteq R^* \B.
    \]
   Condition (b) for $\delta = \delta_*$ implies that $\tau_{\delta_*} H \cap \B \subseteq \tau_{\delta_*} \E \cap \B$. Thus,
    \[
        (R^*)^{-1} \B \cap V^\perp \subseteq \Pi_{V^\perp} (\tau_{\delta_*} H \cap \B) \subseteq \Pi_{V^\perp}(\tau_{\delta_*} \E \cap \B),
    \]
    where the first inclusion uses the assumption that $\tau_{\delta_*} H$ is $R^{*}$-transverse to $V$. Thus, $\tau_{\delta_*} \E$ is $R^*$-transverse to $V$, contradicting the hypotheses  on $\E$ and $V$.
\end{proof}

\begin{proof}[Proof of Proposition \ref{prop:comp_mon}]
    By Lemma \ref{lem:help4}, there exists a subspace $H \subseteq X$ such that $\tau_{\delta_{\max}}H$ is $4R$-transverse to $V$, $\tau_{\delta_{*}}H$ is not $R^*$-transverse to $V$, and $\tau_{\delta_{\min}}H$ is $4R$-transverse to $W$, where $\delta_{\min} \leq \delta_* \leq \delta_{\max}$. According to Lemma \ref{lem:trans1},
    \[
    \begin{aligned}
    &\cos( \theta_{\max}(\tau_{\delta_{\max}} H, V^\perp)) \geq (4R)^{-1}\\
    &\cos (\theta_{\max}( \tau_{\delta_*} H, V^\perp)) \leq (R^*)^{-1},\\
    &\cos (\theta_{\max} ( \tau_{\delta_{\min}} H, W^\perp)) \geq (4R)^{-1},
    \end{aligned}
    \]
    with $\dim(V^\perp)=\dim(W^\perp)= \ell$, where $\ell := \dim(H)$. 
    
    By Lemma \ref{lem:ineq1}, we then have
    \begin{align}
    \label{eq:111}
        &\cos(\angle(\tau_{\delta_{\max}}H,V^\perp)) \geq (4R)^{-\ell}\\
    \label{eq:112}
        &\cos(\angle(\tau_{\delta_*} H, V^\perp)) \leq (R^*)^{-1},\\
    \label{eq:113}
        &\cos(\angle(\tau_{\delta_{\min}} H,W^\perp)) \geq (4R)^{-\ell}.
    \end{align}
    
    Suppose for contradiction that $\sgn(V) \leq \sgn(W)$. Then, by \eqref{eqn3:sec3}, we have $\sgn(V^\perp) \geq \sgn(W^\perp)$. Now, let $\alpha(\delta) = \frac{\cos(\angle(\tau_\delta H,V^\perp))}{\cos(\angle(\tau_\delta H,W^\perp))}$. Let $\omega_H$, $\omega_{\Vp}$,  $\omega_{\Wp}$ be representative forms for $H$, $\Vp$, and $\Wp$, respectively. We then write
    \[
    \begin{aligned}
        \alpha(\delta) &= \frac{|\omega_{\Wp}|\cdot |\langle \tau_\delta^* \omega_H , \omega_{\Vp} \rangle|}{|\omega_{\Vp}| \cdot |\langle \tau_\delta^* \omega_H, \omega_{\Wp} \rangle|}\\
        &= \frac{|\omega_{\Wp}| \cdot |\langle \omega_H, \tau_\delta^* \omega_{\Vp} \rangle|}{|\omega_{\Vp}| \cdot |\langle \omega_H, \tau^*_\delta \omega_{\Wp} \rangle |} = \frac{|\omega_{\Wp}| \cdot |\langle \omega_H , \omega_{\Vp} \rangle| \cdot \delta^{-\sgn(\Vp)}}{|\omega_{\Vp}| \cdot |\langle \omega_H , \omega_{\Wp} \rangle| \cdot \delta^{-\sgn(\Wp)}}.
    \end{aligned}
    \]
    By assumption, $\sgn(\Wp) \leq \sgn(\Vp)$, so $\delta \mapsto \alpha(\delta)$ is non-increasing. 
    
    By \eqref{eq:111} and $\cos (\angle (\tau_{\delta_{\max}}H,W^\perp)) \leq 1$, we have $\alpha(\delta_{\max}) \geq (4R)^{-\ell}$. Because $\delta \mapsto \alpha(\delta)$ is non-increasing, $\alpha(\delta_{\min}) \geq \alpha(\delta_{\max}) \geq (4R)^{-\ell}$. 
    
    From \eqref{eq:111} and \eqref{eq:112}, we have $\cos ( \angle (\tau_{\delta_*} H , \Vp)) \leq \cos (\angle (\tau_{\delta_{\max}}H , \Vp))$, so long as $R^{*} \geq (4R)^\ell$. Thus, by Lemma \ref{lem:multi3}, we have 
    \[
    \cos (\angle( \tau_{\delta_{\min}} H , V^\perp)) \leq \cos (\angle (\tau_{\delta_*} H , V^\perp)) \leq (R^*)^{-1}.
    \]
    Thus, using \eqref{eq:113}, $\alpha(\delta_{\min}) \leq \frac{(4R)^\ell}{R^{*}}$. This yields a contradiction for $R^{*} > (4R)^{2\ell}$, which is implied by our assumptions $R^* \geq R^{3d}$ and $R \geq 16$ (see \eqref{eqn:const_assump}).
\end{proof}


\begin{proof}[Proof of Proposition \ref{prop:main_goal}]

Let $\E$ be an ellipsoid in $X$, let $R \geq 16$ and $R^* \geq \max\{R^{4m},R^{3d}\}$. Recall that we have chosen a constant $\epsilon \in (0,1/4R]$ with $R/R^* < \epsilon^{2m}$; see \eqref{eqn:const_assump}. To prove the result that $\C(\E,R,R^*) \leq 4 m d^2$ we will show that, for $\delta$ in the complement of a controlled number of intervals, the principal axis lengths of $\tau_\delta \E$ avoid values $\sim 1$, and within a connected component of this complementary region we may apply Proposition \ref{prop:comp_mon} to prove monotonicity of the sequence of signatures of the DI subspaces that arise in the definition of the complexity of $\E$.

To prove that $\C(\E,R,R^*) \leq 4 m d^2$, we must demonstrate that $K \leq 4 m d^2$ whenever $\{I_k\}_{k=1}^K$ is a sequence of intervals and $\{V_k\}_{k=1}^K$ is a sequence of DI subspaces such that 
\begin{align}
\label{eqn:comp_bd1}
&\tau_{r(I_k)}\E \mbox{ is } R \mbox{-transverse to }V_k, \mbox{ and} \\
\label{eqn:comp_bd2}
&\tau_{l(I_k)} \E \mbox{ is not } R^*\mbox{-transverse to } V_k\mbox{ for every }k.
\end{align}

By the form of $\tau_\delta$ in \eqref{eqn1:sec3}, \eqref{eqn2:sec3}, we have $|x| \leq | \tau_a x| \leq a^{-m} |x|$ for $a < 1$. Note that $\tau_{l(I_k)} \E = \tau_{a_k} \tau_{r(I_k)} \E$ with $a_k = l(I_k)/r(I_k) < 1$. Note  $V_k$ is dilation invariant, so $\tau_{a_k} V_k = V_k$. By Lemma \ref{lem:basic_trans_result2} and \eqref{eqn:comp_bd1}, we deduce that $\tau_{l(I_k)} \E$ is $(r(I_k)/l(I_k))^{m} R$-transverse to $V_k$. Thus, by \eqref{eqn:comp_bd2}, $(r(I_k)/l(I_k))^{m} R \geq R^*$, hence
\begin{equation}\label{eqn:bca}
r(I_k)/l(I_k) \geq  (R^{*}/R)^{1/m} \geq \epsilon^{-2} \qquad (k=1,2,\dots,K).
\end{equation}
Here we use that $\frac{R^{*}}{R} \geq \epsilon^{-2m}$ (see \eqref{eqn:const_assump}). 

Apply Lemma \ref{lem:help3} to $\E$ and $\epsilon$ to find intervals $ J_1,\dots,J_d \subseteq (0,\infty)$ such that $\tau_\delta \E$ is $\epsilon$-degenerate for all $\delta \notin \cup_{p=1}^d J_p$ and such that $r(J_p)/l(J_p) \leq \frac{1}{\epsilon^2}$ for all $p$. Given the $I_k$ are disjoint, and by \eqref{eqn:bca}, at most two of the $I_k$ can intersect each $J_p$. Thus, $\#\{k : I_k \cap J_p \neq \emptyset \text{ for some } p =1,2,\dots, d\} \leq 2d$. If $L$ is a component interval of $(0,\infty) \setminus \bigcup\limits_{p=1}^d J_p$ then $\tau_\delta \E$ is $\epsilon$-degenerate for all $\delta \in L$, by Lemma \ref{lem:help3}. Thus, by \eqref{eqn:comp_bd1} and \eqref{eqn:comp_bd2}, Proposition \ref{prop:comp_mon} implies that the number of $I_k$ contained in $L$ is at most the number of signatures of subspaces of the same dimension. It is easily checked that this number is at most $md+1$. Furthermore, the number of component intervals $L$ is at most $d+1$. Putting this together, we learn that $K \leq 2d+(md+1)(d+1)$. If $m \geq 2$ and $d \geq 1$ or $m \geq 1$ and $d \geq 2$, then $K \leq 4 m d^2$, as desired. Else, if $m=d=1$ then it is easily verified that $\cC(\E,R,R^*) \leq 2$ for all ellipsoids $\E \subseteq X \simeq \R$.

\end{proof}

\begin{proof}[Proof of Proposition \ref{prop:comp_bound}]

Our task is to show that $\C(\Omega,R_1,R_2) \leq 4 m d^2$ whenever $\Omega \in \cK(X)$, and 
\begin{equation}\label{eqn:const_assump2}
R_1 \geq 16, \quad R_2 \geq  \max \{ ( \sqrt{d})^{4m+1} R_1^{4m}, ( \sqrt{d})^{3d+1} R_1^{3d}\}.
\end{equation}

We claim it is sufficient to show that $\C(\Omega', R_1, R_2) \leq 4 m d^2$ for all compact $\Omega' \in \cK(X)$. We  check that this result implies Proposition \ref{prop:comp_bound}, by contrapositive. Suppose that there exists $\Omega \in \K(X)$ such that $\C(\Omega, R_1, R_2) > 4 m d^2$. Then, for $K = 4 m d^2+1$, there exist compact intervals $\{I_j\}_{j=1}^K$ and dilation invariant subspaces $\{V_j\}_{j=1}^K$ such that
\begin{itemize}
    \item $I_j > I_{j+1} > 0$ for each $j < K$,
    \item $T_{r(I_j)} \Omega$ is $R_1$-transverse to $V_j$ for each $j \leq K$, and
    \item $T_{l(I_j)} \Omega$ is not $R_2$-transverse to $V_j$ for each $j \leq K$.
\end{itemize}
We may assume without loss of generality that $r(I_1) = 1$. To obtain this reduction we make the substitutions $\Omega \leftarrow \tau_{r(I_1)} \Omega$ and $I_j \leftarrow r(I_1)^{-1} I_j$.

Now fix $\xi > R_2$ and set
\[
    \widehat{\Omega} = \Omega \cap \xi \B.
\]
Note that $\widehat{\Omega} \in \cK(X)$ is compact. Furthermore, $\tau_\delta \widehat{\Omega} = \tau_\delta \Omega \cap \xi \tau_\delta \B$. By the form of $\tau_\delta$, we have $\tau_\delta \B \supseteq \B$ for $\delta \leq 1$. Thus,
\[
\tau_\delta \widehat{\Omega} \cap \xi \B = (\tau_\delta \Omega \cap \xi \tau_\delta \B) \cap \xi \B = \tau_\delta \Omega \cap \xi \B \qquad (\delta \leq 1).
\]
So, by Lemma \ref{lem:stability}, and by the second bullet point above, since $R_1 \leq R_2 < \xi$, we have that $ \widehat{\Omega}$ is $R_1$-transverse to $V_j$ for $j \leq K$. Further, if $\tau_{l(I_j)} \widehat{\Omega}$ were $R_2$-transverse to $V_j$, we would have that $\tau_{l(I_j)} \Omega$ is $R_2$-transverse to $V_j$, contradicting our choice of $V_j$ in the third bullet point above. Thus, $\tau_{l(I_j)} \widehat{\Omega}$ is not $R_2$-transverse to $V_j$ for each $j \leq K$. We deduce that $\C(\widehat{\Omega}, R_1, R_2) \geq K > 4 m d^2$.

We reduced the proof of Proposition \ref{prop:comp_bound} to the claim that $\C(\Omega, R_1, R_2) \leq 4 m d^2$ for all compact $\Omega \in \K(X)$. Fix a compact set $\Omega \in \K(X)$. By John's theorem (Proposition \ref{prop:john}), there exists an ellipsoid $\E$ such that $\E$ and $\Omega$ are $\sqrt{d}$-equivalent. By Remark \ref{rem:1}, $\C(\Omega, R_1,R_2) \leq \C(\E, \sqrt{d}R_1,R_2/ \sqrt{d})$.

We set $R = \sqrt{d} R_1$ and $R^* = R_2/ \sqrt{d}$. According to \eqref{eqn:const_assump2} we have $R \geq 16$ and $R^* \geq \max\{R^{4m},R^{3d}\}$. By Proposition \ref{prop:main_goal}, we have  $\C(\E,R,R^*) \leq 4m d^2$. This completes the proof of Proposition \ref{prop:comp_bound}.

\subsection{Proof of Proposition \ref{prop:tech2}}

Fix $x \in \R^n$. Consider the Hilbert space $\Po_x$ given by the vector space $\Po$ equipped with the inner product $\langle \cdot, \cdot \rangle_x$. Define the dilation operators, $\tau_{x,\delta}: \Po_x \rightarrow \Po_x$, given by $\tau_{x,\delta}(P)(z) = \delta^{-m} P(\delta(z-x)+x)$ for $\delta > 0$. Consider the Hilbert dilation system $\cX_x = (\Po_x, \tau_{x,\delta})_{\delta > 0}$, which satisfies the hypotheses of Section \ref{sec:rescale_grass}, with $d=\dim( \Po_x) = D$, and for the choice of subspaces 
\[
X_\nu := \spn \{ m_\alpha(z) = (z-x)^{\alpha} : |\alpha| = m-\nu \} \subseteq \Po_x \;\; \mbox{for } \nu=1,\dots,m,
\]
so that $\tau_{x,\delta}|_{X_\nu} = \delta^{-\nu} \mathrm{id}|_{X_\nu}$.

Pointwise complexity given in  Definition \ref{jet_complexity_def} satisfies 
\[
\C_x(\Omega, R , R^*,\delta) \leq \C_x(\Omega, R , R^*,\infty) = \C_x(\Omega, R , R^*).
\]
Thus, it is sufficient to prove that $\C_x(\Omega, R , R^*) \leq 4 m D^2$. Note that $\C_x(\Omega, R , R^*)$ is identical to the complexity $\C_{\cX_x}(\Omega, R , R^*)$ of $\Omega$ with parameters $(R,R^*)$ with respect to the Hilbert dilation system $\cX_x$; see Definition \ref{defn:comp2}.

According to Proposition \ref{prop:comp_bound}, if $R \geq 16$ and 
\begin{equation}\label{eqn:constsA}
R^* \geq \max\{ (\sqrt{D})^{4m+1} R^{4m}, (\sqrt{D})^{3D+1} R^{3D}\}
\end{equation}
then $\cC_x(\Omega, R, R^*) \leq 4 m D^2$. Note $D \geq m$. So the inequality \eqref{eqn:constsA} is implied by $R^* \geq D^{2D+1/2} R^{4D}$, as assumed in the statement of Proposition \ref{prop:tech2}.

This completes the proof of Proposition \ref{prop:tech2}.

\end{proof}

\section{Whitney convexity and ideals in the ring of jets}\label{sec:WhConv}

We study the relationship between ideals and Whitney convex sets in the ring of jets. Our goal is to give a proof of Proposition \ref{prop:tech1}. By translation, it suffices to prove this result for the jet space at $x=0$. 

We first set the notation to be used in the rest of this section.

Throughout this section we write $\Po$ to denote the vector space of polynomials on $\R^n$ of degree at most $m-1$. We write $\odot$ to denote the ``jet product'' on $\Po$ defined by $P \odot Q = J_0(P \cdot Q)$. We set $\cR = (\Po, \odot)$. We refer to $\cR$ as the  ``the ring of $(m-1)$-jets at $x=0$''.

We will work with subspaces of $\cR$ spanned by monomials. Let $\cM$ be the set of multiindices of length $n$ and order at most $m-1$. For $\cA \subseteq \cM$, let $V_\cA := \spn \{ x^\alpha : \alpha \in \cA\}$. 

Let $D = \dim \cR = \# \cM$.

For $\delta > 0$, let $\tau_\delta : \cR \rightarrow \cR$ be the dilation operator $\tau_{0,\delta}$ defined in Section \ref{sec:prelim},  characterized by its action on monomials: $\tau_\delta(x^\alpha) = \delta^{|\alpha| - m} x^\alpha$ ($\alpha \in \cM$).

Write $|\cdot|$ and $\langle \cdot,\cdot \rangle$ to denote the standard norm and inner product on $\cR$, for which the monomials $\{ x^\alpha : \alpha \in \cM \}$ are an orthonormal basis for $\cR$. Thus,
\begin{equation}\label{eqn:ip_norm}
\begin{aligned}
&\langle P, Q \rangle = \sum_{|\alpha| \leq m-1} \partial^\alpha P(0) \cdot \partial^\alpha Q(0) / (\alpha!)^2, \\
&|P| = \sqrt{\langle P,P \rangle} \qquad\qquad\qquad\qquad\qquad\qquad\qquad (P,Q \in \cR).
\end{aligned}
\end{equation}
We obtain an orthogonal decomposition $\cR = \bigoplus_{i=0}^{m-1} \cR_i$ by setting $\cR_i := \spn \{ x^\alpha : |\alpha| = i\}$ (the space of homogeneous polynomials of degree $i$).

Recall the Bombieri-type inequality (see Lemma \ref{lem:bombieri}): For any $P,Q \in \cR$,
\begin{equation}\label{eqn:bomb1}
    |P \oz Q| \leq C_b |P| |Q|, \quad C_b = (m+1)!.
\end{equation}

\subsection{Renormalization lemma}

Let $\zeta = (\zeta_1, \zeta_2,\dots,\zeta_n) \in [1,\infty)^{n}$. Define a mapping $T_{\zeta} : \cR \rightarrow \cR$  by
\begin{equation}\label{eqn:T_zeta}
T_{\zeta}(P)(x) = P( \zeta_1 x_1, \zeta_2 x_2, \dots,  \zeta_n x_n) \quad (P \in \cR, \; x=(x_1, x_2, \dots,x_n) \in \R^n).
\end{equation}
Observe that $T_{\zeta} : \cR \rightarrow \cR$ is a ring isomorphism, i.e., $T_{\zeta}(P \odot Q) = T_{\zeta}(P) \odot T_{\zeta}(Q)$ for $P,Q \in \cR$. Also,
\begin{equation}\label{eqn:tau_lambda}
     |P| \leq | T_{\zeta} (P) | \leq \Lambda^{m-1} \cdot | P | \qquad (P \in \cR, \; \zeta \in [1,\Lambda]^{n}).
\end{equation}
We first verify \eqref{eqn:tau_lambda} for a monomial $P = m_\alpha$, $m_\alpha(x) = x^\alpha$ ($|\alpha| \leq m-1$). Note that $T_{\zeta}(m_\alpha) = {\zeta}^\alpha m_\alpha$, where we use  multiindex notation: if $\zeta = (\zeta_1,\dots,\zeta_n)$ and $\alpha = (\alpha_1,\dots, \alpha_n)$ then $\zeta^\alpha = \prod_{i=1}^n \zeta_i^{\alpha_i}$. So $m_\alpha$ is an eigenvector of $T_\zeta$ with eigenvalue $\zeta^\alpha$. Observe that $|\zeta^\alpha| \in [1, \Lambda^{m-1}]$ if $\zeta \in [1,\Lambda]^n$ and $|\alpha| \leq m-1$, proving \eqref{eqn:tau_lambda} for $P = m_\alpha$. The full inequality \eqref{eqn:tau_lambda} then follows by orthogonality of the monomial basis $\{m_\alpha \}$ in $\cR$.

\begin{lemma}[Renormalization lemma] \label{lem:newlem}
Let $\epsilon \in \left(0,\frac{1}{2} \right)$, and $D = \dim \cR$. Set $\Lambda(\epsilon) :=(2^D/\epsilon)^{3D^4}$. Given a subspace $H \subseteq \cR$, there exist a multiindex set $\cA \subseteq \cM$ and $\zeta  \in [1,\Lambda(\epsilon)]^{n}$ with
\begin{equation}\label{eqn:newlem0}\cos(\theta_{\max}(T_{\zeta}(H),V_\cA)) > 1-\epsilon.
\end{equation}
\end{lemma}
\begin{proof}

The Euclidean inner product of $p,q \in \R^n$ is denoted by $\langle p, q \rangle = \sum_i p_i q_i$. An $n$-tuple $p = (p_1, p_2, \dots, p_n) \in \N^n$ is said to be \emph{admissible} if 
\begin{equation}
    \label{eqn:newlem1}
\langle p, \alpha \rangle \neq \langle p, \alpha' \rangle \mbox{ for all distinct } \alpha, \alpha' \in \cM.
\end{equation}

An application of the pigeonhole principle shows that there exists an admissible $p \in \N^n$ with 
\begin{equation}\label{eqn:newlem2}
\| p \|_\infty = \max_i p_i \leq \binom{D}{2}+1.
\end{equation}
Indeed, let $K := \binom{D}{2} + 1$. We want to show that there exists an admissible $p \in \{1,2,\dots,K\}^n$. For each pair of distinct multiindices $\alpha, \alpha' \in \cM$, the number of $p \in \{1,2,\dots,K\}^n$ such that $\langle p, \alpha - \alpha' \rangle = 0$ is at most $K^{n-1}$. There are $\binom{D}{2}$ many pairs of distinct multiindices $(\alpha,\alpha') \in \cM \times \cM$ (recall: $D =\# \cM$). Since $K^n > \binom{D}{2} K^{n-1}$, there exists an admissible $p \in \{1,2,\dots,K\}^n$.

Fix an admissible $p = (p_1,p_2,\dots,p_n) \in \N^n$ satisfying \eqref{eqn:newlem2}. 

Let $\psi_\alpha := 1 + \langle p, \alpha \rangle \in \N$ for $\alpha \in \cM$, and let $M:= m D^2$. Thanks to \eqref{eqn:newlem2},
\begin{equation}\label{eqn:newlem2.5}
\begin{aligned}
1 \leq \psi_\alpha &\leq 1 + | \alpha | \cdot \| p \|_\infty \\
&\leq 1 + (m-1) \cdot \left( \binom{D}{2}+1 \right) \leq M \qquad (\alpha \in \cM).
\end{aligned}
\end{equation}

Let $\Po^M$ be the vector space of univariate polynomials $p(t)$ of degree at most $M$. We define an injective linear map $\Phi : \cR \rightarrow \Po^M$,
given by
\[
\Phi(P) = h, \quad\text{where}\;\; h(t) = t \cdot P(t^{p_1}, t^{p_2},\dots, t^{p_n}).
\]
Observe that $\Phi$ sends the monomial $m_\alpha(x) = x^\alpha$ in $\cR$ ($\alpha \in \cM$) to the monomial $k_\alpha(t) := t^{\psi_\alpha} = t^{1 + \langle p, \alpha \rangle}$. Note that $k_\alpha$ is in $\Po^M$, and thus $\Phi: \cR \mapsto \Po^M$ is well-defined, thanks to \eqref{eqn:newlem2.5}. To see that $\Phi : \cR \mapsto \Po^M$ is injective, recall that $p$ is admissible, thus, $\psi_\alpha \neq \psi_{\alpha'}$ for distinct $\alpha,\alpha' \in \cM$.

Let $Y = \Phi(\cR) = \spn \{ k_\alpha : \alpha \in \cM \} \subseteq \Po^M$.  We equip $Y$ with an inner product so that $\{k_\alpha : \alpha \in \cM\}$ is an orthonormal basis for $Y$.

Therefore, $\Phi : \cR \rightarrow Y$ is an isometry, because $\Phi$ maps the orthonormal basis $\{ m_\alpha : \alpha \in \cM \}$ for $\cR$ to an orthonormal basis for $Y$.

Define a linear map $\tau_\delta^Y : Y \rightarrow Y$ by $\tau_\delta^Y(f)(t) = f(t/ \delta)$ for $f \in Y$ ($\delta > 0$). The basis $\{k_\alpha  : \alpha \in \cM \}$ diagonalizes the map $\tau_\delta^Y$; in fact, $\tau_\delta^Y( k_\alpha) = \delta^{-\psi_\alpha} k_\alpha$. We have $Y = \bigoplus_{\alpha \in \cM} \spn \{k_\alpha \}$.  These remarks and \eqref{eqn:newlem2.5} imply that $\cY = (Y, \tau_\delta^Y)_{\delta > 0}$ is a Hilbert dilation system satisfying the hypotheses of Section \ref{sec:rescale_grass} for $m = M$ and $d=\dim Y = D$. Further, the Hilbert dilation system $\cY$ is simple (see Definition \ref{defi:hilbert_system}) because $\psi_\alpha \neq \psi_{\alpha'}$ for $\alpha \neq \alpha'$.

Let $H$ be a $k$-dimensional subspace of $\cR$, and let $\epsilon \in \left(0,\frac{1}{2}\right)$. Set $\delta_0 := \left(\epsilon/2^D \right)^{Dk+2}$. We apply Proposition \ref{prop:subspace_dil} to the Hilbert dilation system $\cY$, subspace $\Phi(H) \subseteq Y$, and interval $I = [\delta_0,1]$. We obtain a subspace $\hat{Y} \subseteq Y$ and a number $\hat{\delta}$ such that
\begin{align}
\label{eqn:newlem3}
&0 <  \delta_0 \leq \hat{\delta} \leq 1,\\
 \label{eqn:newlem4}
&\hat{Y} \mbox{ is invariant under } \tau_\delta^Y \mbox{ for all } \delta > 0,\\
\label{eqn:newlem5}
&\cos (\theta_{\max}( \hat{Y}, \tau^Y_{\hat{\delta}} \Phi(H))) > 1 - \epsilon.
\end{align}

If $\delta > 0$ and $\zeta =  (\delta^{-p_1},\delta^{-p_2},  \dots,\delta^{-p_n})$ then $\tau_\delta^Y \circ \Phi = \delta^{-1} \Phi \circ T_{\zeta}$. In particular, $\tau_\delta^Y (\Phi(V)) = \Phi (T_{\zeta} (V))$ for any subspace $V \subseteq \cR$. Thus, \eqref{eqn:newlem5} implies that
\begin{equation}\label{eqn:newlem6}
    \cos (\theta_{\max}( \hat{Y},  \Phi( T_{\hat{\zeta}}H))) > 1 - \epsilon, \quad \mbox{where } \hat{\zeta} := (\hat{\delta}^{-p_1},\hat{\delta}^{-p_2}, \dots,\hat{\delta}^{-p_n}).
\end{equation}

From \eqref{eqn:newlem4} and the definition of $\tau_\delta^Y$, we see that $\hat{Y}$ is span of univariate monomials. Because $\Phi$ is injective and $\Phi$ maps the monomials $m_\alpha$ to monomials $k_\alpha$, we deduce that $\Phi^{-1}(\hat{Y})$ is the span of monomials; that is, $\Phi^{-1}(\hat{Y}) = V_\cA$ for some $\cA \subseteq \cM$. Because $\Phi$ is an isometry, we learn from \eqref{eqn:newlem6} that
\[
\cos ( \theta_{\max}( V_\cA, T_{\hat{\zeta}}H)) > 1 - \epsilon.
\]
Thus we have proven condition \eqref{eqn:newlem0} for $\zeta = \hat{\zeta}$ and the $\cA$ determined above.

Using \eqref{eqn:newlem2}, \eqref{eqn:newlem3}, and the definition of $\delta_0$, we see that $\hat{\zeta} = (\hat{\zeta}_1,\dots,\hat{\zeta}_n) = (\hat{\delta}^{-p_1},\dots, \hat{\delta}^{-p_n})$ satisfies $\hat{\zeta}_i \geq 1$ and
\[
\hat{\zeta}_i \leq \delta_0^{-D^2} = \left(\frac{2^D}{\epsilon} \right)^{(D k +2) \cdot D^2}  \leq \left(\frac{2^{D}}{\epsilon} \right)^{3D^4} = \Lambda(\epsilon) \qquad (i=1,2,\dots,n).
\]
Therefore, $\hat{\zeta} \in [1,\Lambda(\epsilon)]^n$, and the lemma is proven.

\end{proof}

\subsection{Whitney convexity and quasiideals}

 We recall the definition of Whitney convexity. We take $x=0$ in Definition \ref{wc_def}. We write $\Omega$ is $A$-Whitney convex to mean that $\Omega$ is $A$-Whitney convex at $x=0$. Define $X \odot Y := \{ P \odot Q : P \in X, \; Q \in Y\}$ for subsets $X,Y \subseteq \cR$. Let $\cB_{\delta} \subseteq \cR$ be the unit ball with respect to the $| \cdot |_{0,\delta}$-norm on $\cR$, and let $\cB = \cB_1$ be the unit ball with respect to the standard norm $|\cdot| = |\cdot|_{0,1}$ on $\cR$. A closed symmetric convex set $\Omega \subseteq \cR$ is $A$-Whitney convex provided that  $(\Omega \cap \cB_{\delta}) \odot \cB_{\delta} \subseteq A \delta^m \Omega$ for all $\delta > 0$. By specializing this condition to $\delta = 1$, we obtain: If $\Omega \subseteq \cR$ is $A$-Whitney convex then
\begin{equation}\label{eqn:wh_conv1}
P \in  \Omega \cap \cB \mbox{ and } Q \in \cB \implies P \odot Q \in A \Omega.
\end{equation}
We note that these conditions are a quantitative relaxation of the notion of an ideal in $\cR$. Indeed, any ideal is an $A$-Whitney convex set for any $A > 0$. 

Our next lemma gives the most basic properties of Whitney convexity. Given $\Omega,\Omega' \subseteq \cR$, we write $\Omega \sim_\lambda \Omega'$ ($\Omega$ and $\Omega'$ are $\lambda$-equivalent) for $\lambda \geq 1$ to mean that $\lambda^{-1} \Omega \subseteq \Omega' \subseteq \lambda \Omega$.
\begin{lemma}\label{lem:wc_props}
Let $A \geq 1$. The following properties hold:
\begin{enumerate}
    \item The unit ball $\cB \subseteq \cR$ is $C_b$-Whitney convex, for $C_b = (m+1)!$.
    \item If $\Omega_1 \sim_\lambda \Omega_2$ and $\Omega_1$ is $A$-Whitney convex then $\Omega_2$ is $\lambda^2 A$-Whitney convex.
    \item If $\Omega_1$ and $\Omega_2$ are $A$-Whitney convex then $\Omega_1 \cap \Omega_2$ is $A$-Whitney convex.
    \item If $\Omega$ is $A$-Whitney convex then $\tau_\delta \Omega$ is $A$-Whitney convex for any $\delta > 0$.
    \item If $\Omega$ is $A$-Whitney convex and $\xi \geq 1$ then $\xi \Omega$ is $A$-Whitney convex.
\end{enumerate}
\end{lemma}
\begin{proof}
Recall $\tau_\delta : \cR \rightarrow \cR$ is the dilation operator $\tau_{0,\delta}$ defined in Section \ref{sec:prelim}. Recall our notation that $\cB_\delta = \cB_{0,\delta}$ and $\cB = \cB_{0,1} = \cB_1$. Then identity \eqref{eqn:ball_scale_id} states that $\tau_{\rho} \B_{\delta} = \B_{\delta/\rho}$ for $\rho,\delta > 0$. In particular, for $\rho = \delta$, we have $\cB_\delta = \tau_{\delta^{-1}} \cB$.

We make use of additional set inclusions in the proof. Note that $\tau_\delta$ satisfies the identity $\tau_\delta(P \odot Q) = \delta^{m} \tau_\delta(P) \odot \tau_\delta(Q)$ for $P,Q \in \cR$. Thus, $\tau_\delta(X \odot Y) = \delta^{m} \tau_\delta(X) \odot \tau_\delta(Y)$ for $X,Y \subseteq \cR$. We also make use of the inclusion $(X \cap Y) \odot Z \subseteq (X \odot Z) \cap (Y \odot Z)$ for $X,Y,Z \subseteq \cR$. 

\textbf{Proof of property 1:} If $\delta \geq 1$ then $\B \subseteq \B_\delta \subseteq \delta^m \B$ (see \eqref{eqn:ball_scale}), so
\[
(\B \cap \B_\delta) \odot (\B_\delta) = \B \odot \B_\delta \subseteq \delta^m ( \B \odot \B) \subseteq C_b \delta^m \B,
\]
where the last inclusion is a consequence of \eqref{eqn:bomb1}. 

 If $\delta < 1$ then $\B_\delta \subseteq \B$ (see \eqref{eqn:ball_inc}), and so
\[
\begin{aligned}
(\B \cap \B_\delta) \odot \B_\delta &= \B_\delta \odot \B_\delta = \tau_{\delta^{-1}} \B \odot \tau_{\delta^{-1}} \B \\
& = \delta^{m} \tau_{\delta^{-1}} (\B \odot \B) \subseteq \delta^{m} \tau_{\delta^{-1}} (C_b \B) = C_b \delta^{m} \B_\delta \subseteq C_b \delta^m \B.
\end{aligned}
\]
Thus, $(\B \cap \B_\delta) \odot \B_\delta \subseteq C_b \delta^m \B$ in both cases $\delta \geq 1$ and $\delta <1$. Therefore, $\B$ is $C_b$-Whitney convex.

\textbf{Proof of property 2:} Suppose $\Omega_1$ is $A$-Whitney convex. Then for any $\delta > 0$, $(\Omega_1 \cap \B_\delta) \odot \B_\delta \subseteq A \delta^m \Omega_1$. If $\Omega_1 \sim_\lambda \Omega_2$, we have $\lambda^{-1} (\Omega_2 \cap \B_\delta) \odot \B_\delta \subseteq A \delta^m \lambda \Omega_2$, thus, $\Omega_2$ is $A \lambda^2$-Whitney convex.

\textbf{Proof of property 3:} Suppose that $\Omega_1$ and $\Omega_2$ are $A$-Whitney convex. Then, for any $\delta > 0$
\[
\begin{aligned}
((\Omega_1 \cap \Omega_2) \cap \B_\delta ) \odot \B_\delta &\subseteq (( \Omega_1 \cap \B_\delta) \odot \B_\delta ) \cap ( (\Omega_2 \cap \B_\delta) \odot \B_\delta) \\
&\subseteq A \delta^m \Omega_1 \cap A \delta^m \Omega_2 = A \delta^m (\Omega_1 \cap \Omega_2).
\end{aligned}
\]
So, $\Omega_1 \cap \Omega_2$ is $A$-Whitney convex.

\textbf{Proof of property 4:} Suppose $\Omega$ is $A$-Whitney convex, i.e., $(\Omega \cap \B_\rho) \odot \B_\rho \subseteq A \rho^m \Omega$ for any $\rho > 0$. Note, for any $\delta > 0$,
\[
\tau_\delta( (\Omega \cap \B_\rho) \odot \B_\rho) = \delta^m (\tau_\delta\Omega \cap \tau_\delta \B_\rho) \odot \tau_\delta \B_\rho.
\]
Thus, applying $\tau_\delta$ to both sides of the $A$-Whitney convexity condition, we learn that
\[
\delta^m (\tau_\delta\Omega \cap \tau_\delta \B_\rho) \odot \tau_\delta \B_\rho \subseteq A \rho^m \tau_\delta\Omega \qquad (\rho, \delta > 0).
\]
But $\tau_\delta \B_\rho = \B_{\rho/\delta}$.  By making the substitution $\rho \leftarrow \rho/\delta$, we learn that
\[
(\tau_\delta\Omega \cap \B_{\rho}) \odot  \B_{\rho} \subseteq A \rho^m \tau_\delta\Omega \qquad (\rho, \delta > 0).
\]
Thus, $\tau_\delta \Omega$ is $A$-Whitney convex for any $\delta > 0$.

\textbf{Proof of property 5:} Suppose $\Omega$ is $A$-Whitney convex. Then for any $\delta > 0$, $(\Omega \cap \B_\delta) \odot \B_\delta \subseteq A \delta^m \Omega$. Thus, $(\xi \Omega \cap \xi \B_\delta) \odot \B_\delta \subseteq A \delta^m \xi \Omega$. As $\xi \geq 1$, we have $\B_\delta \subseteq \xi \B_\delta$, thus,
\[
(\xi \Omega \cap \B_\delta) \odot \B_\delta \subseteq A \delta^m \xi \Omega.
\]
So, $\xi \Omega$ is $A$-Whitney convex.
\end{proof}

Next we introduce a concept relating the ring structure of $\cR = (\Po,\odot)$ and the geometric structure of $\cR$. 

\begin{defi}
Let $\epsilon > 0$, and let $H$ be a subspace of $\cR$. Say that $H$ is an $\epsilon$-quasiideal if for all $P \in H, Q \in \cR$ there exists $\widehat{P} \in H$ such that 
\[
|\widehat{P} - P \odot Q | \leq \epsilon |P| \cdot |Q|.
\]
Equivalently, $H$ is an $\epsilon$-quasiideal if
\[
(H \cap \B) \odot \B \subseteq H + \epsilon \B.
\]
\end{defi}
Much like Whitney convexity, the notion of a quasiideal is a quantitative relaxation of the notion of an ideal in $\cR$. Indeed, one easily checks that a subspace $H$ of $\cR$ is an ideal if and only if $H$ is an $\epsilon$-quasiideal for all $\epsilon > 0$. By \eqref{eqn:bomb1}, any subspace of $\cR$ is an $\epsilon$-quasiideal for $\epsilon = C_b = (m+1)!$.

\begin{lemma}\label{lem:newwlem}
Let $A > 0$ and $\epsilon \in (0,1)$, let $H$ be a subspace of $\cR$, and let $\Omega$ be a closed symmetric convex subset of $\cR$. Suppose that $\Omega$ is $A$-Whitney convex. Suppose the following conditions are met.
\begin{enumerate}
    \item[(i)] $\Omega \supseteq H \cap  \cB$.
    \item[(ii)] $\Omega \subseteq H + \epsilon \cB$.
\end{enumerate}
Then $H$ is an $A \cdot \epsilon$-quasiideal.

\end{lemma}
\begin{proof}
We have to demonstrate that $(H \cap \B) \odot \B \subseteq H + \epsilon A \B$. Let $P \in H \cap \cB$ and $Q \in \cB$.

Condition (i) implies that $H \cap \cB \subseteq \Omega \cap \cB$. Thus, $P \in \Omega \cap \cB$ and $Q \in \cB$. Applying condition \eqref{eqn:wh_conv1}, we have $P \odot Q \in  A \Omega$.

Thus, by condition (ii), $P\odot Q \in A (H+\epsilon \cB) = H + \epsilon A \cB$. Since $P \in H \cap \cB$ and $Q \in \cB$ are arbitrary, this completes the proof.

\end{proof}

A continuity argument shows that every $\epsilon$-quasiideal is within distance of $C(\epsilon)$ of an ideal, with $\lim_{\epsilon \rightarrow 0} C(\epsilon) = 0$ (here distance refers to the distance between subspaces; see Section \ref{sec: angles}). In the next lemma we establish a weaker statement, with explicit constants, which is sufficient for our purposes: If an $\epsilon$-quasiideal $I$ is close enough to a subspace of the form $V_\cA = \spn \{ x^\alpha : \alpha \in \cA\}$, then the multiindex set $\cA \subseteq \cM$ is monotonic. (For the definition of monotonic sets, see Definition \ref{def:mon}.) Further, if $\cA$ is monotonic then $V_\cA$ is an ideal (see Lemma \ref{lem:mon}). Consequently, if an $\epsilon$-quasiideal is close enough to a subspace spanned by monomials then it is also close to an ideal. 

We view the next lemma as a robust version of the property that $\cA$ is monotonic if $V_\cA$ is an ideal (see Lemma \ref{lem:mon}).

\begin{lemma}\label{lem:newwwlem}
Let $C_b = (m+1)!$. Let $\eta \leq \frac{1}{32 C_b^2}$ and $\epsilon \leq \frac{1}{8}$. Let $I$ be an $\epsilon$-quasiideal in $\cR$, and let $\cA \subseteq \cM$ satisfy
\begin{equation}\label{eqnnn0}
\cos (\theta_{\max}( I, V_\cA)) > 1 - \eta. 
\end{equation}
Then $\cA$ is monotonic.
\end{lemma}
\begin{proof}
Recall that the monomials $ m_\alpha(x):= x^\alpha$ ($ \alpha \in \cM$) form an orthonormal basis for $\cR$, and recall that $V_\cA = \spn \{m_\alpha : \alpha \in \cA\}$.

By definition of the maximum principal angle, condition \eqref{eqnnn0} ensures that 
\begin{equation}\label{eqn:per1}
|\Pi_{V_{\cA}}(q)| \geq (1-\eta) |q|  \mbox{ for all } q \in I.
\end{equation}
On the other hand, by symmetry we have $\cos (\theta_{\max}( V_\cA, I)) > 1 - \eta$, which implies
\begin{equation}\label{eqn:per2}
|\Pi_I(y)| \geq (1-\eta) |y| \mbox{ for all } y \in V_{\cA}.
\end{equation}

Fix $\alpha \in \cA$ (arbitrary) and consider the monomial $m_\alpha \in V_\cA$. Set $y_\alpha = \Pi_{I} m_\alpha$. By \eqref{eqn:per2},
\[
| y_\alpha| \geq (1-\eta) |m_\alpha| = 1 - \eta.
\]
Thus, by orthogonality of $y_\alpha$ and $y_\alpha - m_\alpha$, and the Pythagorean theorem,
\[
|y_\alpha - m_\alpha| = \sqrt{|m_\alpha|^2 - | y_\alpha|^2} \leq \sqrt{1 - (1-\eta)^2} \leq \sqrt{2 \eta}.
\]
Of course, also
\[
| y_\alpha| \leq |m_\alpha| = 1.
\]
Now fix $\beta \in \cM$ with $\beta + \alpha \in \cM$ (arbitrary). Then $m_\beta \odot m_\alpha = m_{\alpha + \beta}$. By the Bombieri-type inequality \eqref{eqn:bomb1},
\begin{equation}\label{eqnnn1}
\begin{aligned}
|y_\alpha \odot m_\beta - m_{\alpha + \beta} | = | (y_\alpha - m_\alpha) \odot m_\beta| &\leq C_b |y_\alpha - m_\alpha | \cdot |m_\beta| \\
&\leq \sqrt{2\eta} C_b.
\end{aligned}
\end{equation}

Because $I$ is an $\epsilon$-quasiideal, and $y_\alpha \in I$, there exists $q_{\alpha \beta} \in I$ such that
\begin{equation}\label{eqnnn2}
|q_{\alpha \beta} - y_\alpha \odot m_\beta| \leq \epsilon \cdot |y_\alpha| \cdot |m_\beta| \leq \epsilon.
\end{equation}
By the inequalities \eqref{eqnnn1}, \eqref{eqnnn2}, $\eta \leq \frac{1}{32 C_b^2} $, and $\epsilon \leq \frac{1}{8}$, we have
\begin{equation}\label{eqnnn3}
|q_{\alpha \beta} - m_{\alpha+\beta} | \leq \sqrt{2 \eta} C_b + \epsilon \leq (4 C_b)^{-1} C_b + \epsilon < 1/2.
\end{equation}
In particular,
\[
|q_{\alpha \beta} | \leq  |m_{\alpha+\beta}| + 1/2 \leq 2.
\]

Set $\widehat{q}_{\alpha \beta} = \Pi_{V_{\cA}} q_{\alpha \beta} \in V_\cA$. By \eqref{eqn:per1}, and given that $q_{\alpha \beta} \in I$, 
\[
| \widehat{q}_{\alpha \beta}| \geq (1 - \eta) | q_{\alpha \beta}|.
\]
Thus, by orthogonality of $\widehat{q}_{\alpha \beta}$ and $q_{\alpha \beta} - \widehat{q}_{\alpha \beta}$, and the Pythagorean theorem,
\[
\begin{aligned}
| q_{\alpha \beta} - \widehat{q}_{\alpha \beta}| &= \sqrt{|q_{\alpha \beta}|^2 - | \widehat{q}_{\alpha \beta}|^2} \leq \sqrt{1-(1-\eta)^2} |q_{\alpha \beta}| \\
& \leq \sqrt{2 \eta}  |q_{\alpha \beta}| \leq \sqrt{2 \eta}  \cdot 2 \leq 1/2,
\end{aligned}
\]
where the last inequality uses that $\eta \leq \frac{1}{32 C_b^2} \leq \frac{1}{32}$. Therefore, from \eqref{eqnnn3},
\begin{equation}\label{eqnnn4}
|\widehat{q}_{\alpha \beta} - m_{\alpha + \beta}| < 1.
\end{equation}
Because the monomials $\{m_\gamma : \gamma \in \cM\}$ form an orthonormal basis for $\cR$, and because $\widehat{q}_{\alpha \beta} \in V_\cA = \spn \{m_\gamma : \gamma \in \cA\}$, we see that \eqref{eqnnn4} implies that $\alpha + \beta \in \cA$.

We have shown that $\alpha + \beta \in \cA$ for arbitrary multiindices $\alpha \in \cA$, $\beta \in \cM$ such that $\beta + \alpha \in \cM$. Thus, $\cA$ is monotonic.

\end{proof}

\begin{lemma}\label{lem:mainlem}
There exist controlled constants $\epsilon_0 \in (0,1/8)$ and $R_0 \geq 1$ such that the following holds. 

Let $H \subseteq \cR$ be an $\epsilon$-quasiideal for $0 < \epsilon \leq \epsilon_0$.

Then $H$ is $R_0$-transverse to $V_\cA^\perp$ for some monotonic set $\cA \subseteq \cM$.

\end{lemma}
\begin{proof}
Let $\eta := \frac{1}{32 C_b^2} = \frac{1}{32 ((m+1)!)^2} < \frac{1}{2}$. Then $\eta$ is a controlled constant. We apply the Renormalization lemma (Lemma \ref{lem:newlem}) to the subspace $H \subseteq \cR$ with $\epsilon$ in the statement of this lemma taken equal to $\eta$. Set $\Lambda = (2^D/\eta)^{3 D^4}$, which is a controlled constant. Also set  $\epsilon_0 := \frac{1}{8}\Lambda^{1-m}$ and $R_0 := 2 \Lambda^{m-1}$, which are controlled constants.

By the Renormalization lemma there exist a multiindex set $\cA \subseteq \cM$ and a vector $\zeta = (\zeta_1, \dots,\zeta_n) \in [1,\Lambda]^n$ satisfying
\begin{equation}
\label{eqn:mainlem1}
    \cos ( \theta_{\max}(T_{\zeta} H, V_\cA)) > 1 - \eta.
\end{equation}
(See \eqref{eqn:T_zeta} for the definition of the mapping $T_\zeta : \cR \rightarrow \cR$.)

Using $\zeta \in [1,\Lambda]^n$ and \eqref{eqn:tau_lambda}, we have 
\begin{equation}\label{eqn:mainlem1a}
\B \subseteq T_{\zeta}(\B) \subseteq \Lambda^{m-1} \B.
\end{equation}

By assumption, $H$ is an $\epsilon$-quasiideal in the ring $\cR$ for $\epsilon \leq \epsilon_0$. Thus,
\begin{equation}\label{eqn:mainlem2}
(H \cap \B) \odot \B \subseteq H + \epsilon \B.
\end{equation}
Since $T_{\zeta} : \cR \rightarrow \cR$ is a ring isomorphism, we have
\[
T_{\zeta}( (H \cap \B) \odot \B) = (T_{\zeta}( H) \cap T_{\zeta}(\B) ) \odot T_{\zeta}(\B).
\]
Thus, applying $T_{\zeta}$ to both sides of \eqref{eqn:mainlem2}, and using \eqref{eqn:mainlem1a}, we obtain
\[
(T_{\zeta}( H) \cap \B ) \odot \B \subseteq T_{\zeta}(H) + \epsilon \Lambda^{m-1} \B.
\]
Therefore, $T_{\zeta} H$ is an $\epsilon'$-quasiideal in $\cR$, with $\epsilon' = \Lambda^{m-1} \epsilon \leq \Lambda^{m-1} \epsilon_0 = \frac{1}{8}$. Combining this with \eqref{eqn:mainlem1}, we apply Lemma \ref{lem:newwwlem} to deduce that $\cA$ is monotonic.

Now, \eqref{eqn:mainlem1} holds with $\eta < \frac{1}{2}$. So, $\cos ( \theta_{\max}(T_{\zeta} H, V_\cA)) > 1/2$. By Lemma \ref{lem:trans1}, we deduce that $T_{\zeta} H$ is $2$-transverse to $V_\cA^\perp$. By \eqref{eqn:tau_lambda}, we have $\Lambda^{1-m}|P| \leq |T_{\zeta}^{-1}(P)| \leq |P|$ for $P \in \cR$. Thus, by Lemma \ref{lem:basic_trans_result2}, we learn that $H$ is $ 2\Lambda^{m-1}$-transverse to $T_{\zeta}^{-1} V_\cA^\perp$. Finally note that $V_\cA^\perp$ is spanned by monomials, and each monomial is an eigenvector of $T_{\zeta}^{-1}$, thus $T_{\zeta}^{-1} V_{\cA}^\perp = V_{\cA}^\perp$. Therefore, $H$ is $ 2\Lambda^{m-1}$-transverse to $V_\cA^\perp$. This concludes the proof of the lemma.

\end{proof}

\subsection{Proof of Proposition \ref{prop:tech1}}
By translation invariance it suffices to prove Proposition \ref{prop:tech1} for the case $x=0$. Thus, we work in the ring $\cR = (\Po, \odot)$ of $(m-1)$-jets at $x=0$.

Let $A \geq 1$. We first prove Proposition \ref{prop:tech1} under the assumption that $\Omega = \E \subseteq \cR$ is an ellipsoid that is $A$-Whitney convex (at $x=0$). We then extend the result to an arbitrary convex set $\Omega \subseteq \cR$ that is $A$-Whitney convex (at $x=0$).

Let $\epsilon_0 \in (0,1/8)$ and $R_0 \geq 1$ be the controlled constants in Lemma \ref{lem:mainlem}. Set $\epsilon = \epsilon_0/A \in (0,1)$.

Let $\E \subseteq \cR$ be an ellipsoid that is $A$-Whitney convex. We claim there exists $\delta \in [\delta_0,1]$, for $\delta_0 := \frac{1}{2} \epsilon^{2D}$, such that $\tau_\delta \E$ is $\epsilon$-degenerate. To see this, let $J_1,\dots,J_D$ be intervals as in Lemma \ref{lem:help3}. Given that $r(J_p)/l(I_p) \leq \epsilon^{-2}$  for all $p$, there exists $\delta \in [\delta_0,1] \setminus \bigcup_p J_p$. This $\delta$ is as required, by Lemma \ref{lem:help3}.  Note that 
\begin{equation}\label{eqn:delta_0}
\delta_0 = O( \exp(\poly(D))) A^{-2D}.
\end{equation}

By Lemma \ref{lem:help2} (applied for $I=\{\delta\}$), there is a subspace $H \subseteq \cR$ with
\begin{align}
\label{eqn:finup2}
&\tau_\delta \E \supseteq H \cap (2\epsilon)^{-1} \cB,\\
\label{eqn:finup1}
&\tau_\delta \E \subseteq H + \epsilon \cB.
\end{align}
In particular, from \eqref{eqn:finup2}, 
\begin{equation}\label{eqn:finup2a}
\tau_\delta \E \supseteq H \cap \cB.
\end{equation}
By property 4 of Lemma \ref{lem:wc_props}, and because $\E$ is $A$-Whitney convex, we have that \begin{equation}\label{eqn:finup3}
\tau_\delta \E \mbox{ is }A\mbox{-Whitney convex}.
\end{equation}

Using \eqref{eqn:finup1}--\eqref{eqn:finup3} and the fact $\epsilon_0 = \epsilon A$, we apply Lemma \ref{lem:newwlem} to deduce that $H$ is an $ \epsilon_0$-quasiideal. Thus, by Lemma \ref{lem:mainlem}, there exists a monotonic set $\cA \subseteq \cM$ such that, for $W = V_{\cA}^\perp$,
\begin{equation}
    \label{eqn:finup4}
     H \mbox{ is }R_0\mbox{-transverse to } W.
\end{equation}
Note that $W =V_\cA^\perp = V_{\cM \setminus \cA}$ is a DTI subspace because $\cA$ is monotonic -- see Lemma \ref{lem:mon}.

From \eqref{eqn:finup1},  \eqref{eqn:finup4}, and Lemma \ref{lem:basic_trans_result}, we have
\[
\tau_\delta \E \cap W \subseteq (H + \epsilon \B) \cap W \subseteq \epsilon R_0 \B \subseteq R_0 \B,
\]
and from \eqref{eqn:finup2a}, \eqref{eqn:finup4}, we have
\[
R_0^{-1} \B \cap W^\perp \subseteq \Pi_{W^\perp} (H \cap \B) \subseteq \Pi_{W^\perp} (\tau_\delta \E \cap \B).
\]
Therefore, $\tau_\delta \E$ is $R_0$-transverse to $W$.

Recall $\delta \in [\delta_0,1]$, and so $\delta_0^m |P| \leq |\tau_\delta^{-1}(P)| \leq |P|$ for $P \in \cR$. Also, $\tau_\delta^{-1} W = W$, since $W = V_{\cA^\perp}$ is spanned by monomials. By Lemma \ref{lem:basic_trans_result2}, $\E$ is $Z$-transverse to $W$, for $Z = Z(A) := R_0 \delta_0^{-m} \geq 1$. Note that $Z = O(\exp(\mbox{poly}(D))) A^{2mD}$, since $R_0$ is a controlled constant and by the form of $\delta_0$ in \eqref{eqn:delta_0}.

Thus, if $\E$ is an $A$-Whitney convex ellipsoid, we have produced $Z = Z(A) \geq 1$ and a DTI subspace $W$ such that $\E$ is $Z$-transverse to $W$. This establishes Proposition \ref{prop:tech1} for ellipsoids.

Now suppose $\Omega \subseteq \cR$ is $A$-Whitney convex. Set $\widehat{\Omega} = \Omega \cap \xi \B$, for $\xi \geq 1$ to be determined below. By John's theorem (Proposition \ref{prop:john}) there is an ellipsoid $\E$ that is $\sqrt{D}$-equivalent to $\widehat{\Omega}$.

From properties 1, 3, and 5 in Lemma \ref{lem:wc_props},  $\widehat{\Omega}$ is $A_*$-Whitney convex for $A_* = \max\{A,C_b\}$. From property 2 in Lemma \ref{lem:wc_props}, $\E$ is $DA_*$-Whitney convex.

By the established case of Proposition \ref{prop:tech1} for ellipsoids, there exists $Z \geq 1$ and a DTI subspace $W \subseteq \cR$ such that $\E$ is $Z$-transverse to $W$, where
\[
Z = O(\exp(\mbox{poly}(D))) (DA_*)^{2mD} = O( \exp(\mbox{poly}(D))) A^{2mD}.
\]
Because $\widehat{\Omega} \sim_{\sqrt{D}} \E$, we have that $\widehat{\Omega}$ is $ZD$-transverse to $W$ -- see Remark \ref{rem:1}.

Recall that $\widehat{\Omega} = \Omega \cap \xi \cB$. We now fix $\xi > ZD$. Then $\Omega$ is $ZD$-transverse to $W$, by Lemma \ref{lem:stability}. We note that $ZD = O( \exp(\mbox{poly}(D))) A^{2mD}$.

We take $R_0$ in Proposition \ref{prop:tech1} of the form $R_0 = \exp(\poly(D) \log(A))$ satisfying $R_0 \geq ZD$. This completes the proof of Proposition \ref{prop:tech1}.

\section{Main Extension Theorem for finite sets}\label{sec:statement_mt}
In the previous sections we proved the main technical results, Propositions \ref{prop:tech1} and \ref{prop:tech2}.

We return to the task of proving the main theorems from the introduction. We first state Theorem \ref{thm:sharpfinitenessprinciple}, our extension theorem for finite $E \subseteq \R^n$. We develop additional analytical tools in the next few sections. We prove Theorem \ref{thm:sharpfinitenessprinciple} in Section \ref{subsec:mainproofs1}. and we prove Theorems \ref{thm: new c sharp} and \ref{thm: lin op} from the introduction in Section \ref{subsec:mainproofs2}.

Given a set $E\subseteq\R^n$, function $f:E\rightarrow \R$, integer $k^\# \geq 1$, and $M  > 0$, we consider the following hypothesis on $f$: 
\begin{equation}\label{eqn:FH}
\cF\cH(k^\#, M) \begin{cases}
\text{For all }  S\subseteq E \text{ with } \#(S) \le k^\#\\
\quad \text{there exists } F^S \in C^{m-1,1}(\R^n)\\
\qquad \text{with } F^S = f \text{ on } S \text{ and } \|F^S\|_{C^{m-1,1}(\R^n)}\le M.
\end{cases}
\end{equation}
We refer to $\cF\cH(k^\#,M)$ as a \emph{finiteness hypothesis} on $f$ with \emph{finiteness constant} $k^\#$ and \emph{finiteness norm} $M$. 

For $E$ finite, let $C(E)$ denote the space of all real-valued functions on $E$.

\begin{thm}\label{thm:sharpfinitenessprinciple}
For $m,n\ge 1$, there exist constants $C^\#\ge 1$ and $k^\#\in\N$ with $C^\# = O(\exp(\poly(D)))$ and $k^\# = O(\exp(\poly(D)))$ such that the following holds. Let $E\subseteq \R^n$ be finite. 

\noindent (A) If $f \in C(E) $ satisfies $\cF\cH(k^\#,M)$ then $\| f \|_{C^{m-1,1}(E)} \leq C^\# M$.

\noindent (B) There exists a linear map $T:C(E)\rightarrow C^{m-1,1}(\R^n)$ satisfying that $Tf = f$ on $E$ and $\|Tf\|_{C^{m-1,1}(\R^n)}\le C^\# \| f \|_{C^{m-1,1}(E)}$ for all $f \in C(E)$.
\end{thm}

\section{The Basic Convex Sets}\label{sec:basicconv}

In this section we introduce indexed families of convex subsets of $\Po$ that lie at the heart of the proof of Theorem \ref{thm:sharpfinitenessprinciple}.

Below, the seminorm of $\varphi \in C^{m-1,1}(\R^n)$ is denoted by $\| \varphi \| := \| \varphi \|_{C^{m-1,1}(\R^n)}$.

Fix a  finite set $E \subseteq \R^n$ and function $f : E \rightarrow \R$.

Given $S \subseteq E$, $x \in \R^n$, and $M > 0$, let
\begin{equation}\label{eqn:gamma_S}
\begin{aligned}
&\sigma_S(x) := \{ J_x \varphi : \varphi \in C^{m-1,1}(\R^n), \; \| \varphi \| \leq 1, \; \varphi = 0 \mbox{ on } S \},\\
&\Gamma_S(x,f,M) := \{ J_x F : F \in C^{m-1,1}(\R^n),\; \| F \| \leq M,  \; F = f \mbox{ on } S \}.
\end{aligned}
\end{equation}
Note that $\sigma_S(x)$ is a symmetric convex set in $\Po$, while $\Gamma_S(x,f,M)$ is merely convex. By a compactness argument using the Arzela-Ascoli theorem, we see that $\sigma_S(x)$, $\Gamma_S(x,f,M)$ are closed. When $S = E$, we abbreviate the notation by setting $\sigma(x) := \sigma_E(x)$ and $\Gamma(x,f,M) := \Gamma_E(x,f,M)$. 

We define variants of the above convex sets indexed by an integer parameter $\ell$ rather than a subset $S \subseteq E$. Given $x \in \R^n$ and $\ell \geq 0$, let
\[
\sigma_\ell(x) := \bigcap_{\substack{S\subseteq E \\ \#(S) \leq (D+1)^\ell}} \sigma_S(x).
\]
Given also $M >0 $, let
\[
\Gamma_\ell(x,f,M) := \bigcap_{\substack{ S \subseteq E \\ \#(S) \leq (D+1)^\ell}} \Gamma_S(x,f,M).
\]
A more explicit description of $\Gamma_\ell(x,f,M)$ is given by:
\begin{equation}\label{eqn:gamma_ell}
\begin{aligned}
\Gamma_\ell(x,f,M) = \{ P \in \Po : \; & \forall S \subseteq E, \; \#(S) \leq (D+1)^\ell, \; \exists F^S \in C^{m-1,1}(\R^n) \\
&  \mbox{s.t.} \; F^S = f \mbox{ on } S, \; J_{x} F^S = P, \; \| F^S \| \leq M \}.
\end{aligned}
\end{equation}
Evidently, $\sigma_\ell(x)$ is a closed, symmetric, convex set, whereas $\Gamma_\ell(x,f,M)$ is closed and convex. 

The $\sigma$-sets arise from the $\Gamma$-sets by taking $f\equiv 0|_E$ and $M=1$; that is,
\[
\begin{aligned}
&\sigma_S(x) = \Gamma_S(x, 0|_E, 1), \\
&\sigma_\ell(x) = \Gamma_\ell(x,0|_E,1).
\end{aligned}
\]

Next we state the important properties of these sets that will be used in the ensuing proof of Theorem \ref{thm:sharpfinitenessprinciple}. Many of these results are borrowed from \cite{coordinateFree}. In many cases we point the reader to \cite{coordinateFree} for proofs.

The following standard result on convex sets is a key ingredient in our proofs. See Lemma \ref{lem: sharp helly} for a related version.

\begin{lemma}[Helly’s Theorem (see, e.g., \cite{Web94})] Let $\mathcal{J}$ be a finite family of convex subsets of $\R^d$, any $d+1$ of which have non–empty intersection. Then the whole family $\mathcal{J}$ has
non–empty intersection.
\end{lemma}

\begin{lemma}\label{lem:sigma_gamma_rel}
For any $\ell\geq 0$ and $M_1,M_2>0$,
\begin{align*}
    &\Gamma_\ell(x,f,M_1)+ M_2 \cdot \sigma_\ell(x)\subseteq \Gamma_{\ell}(x,f,M_1+M_2), \quad \text{and}\\
    &\Gamma_\ell(x,f,M_1) - \Gamma_{\ell}(x,f,M_2)\subseteq (M_1+M_2) \sigma_\ell(x).
\end{align*}
Similarly, for any $S \subseteq E$ and $M_1,M_2>0$,
\begin{align*}
    &\Gamma_S(x,f,M_1)+ M_2 \cdot \sigma_S(x)\subseteq \Gamma_{S}(x,f,M_1+M_2), \quad \text{and}\\
    &\Gamma_S(x,f,M_1) - \Gamma_{S}(x,f,M_2)\subseteq (M_1+M_2) \sigma_S(x).
\end{align*}
\end{lemma}
\begin{proof}
The proof is immediate from the definitions and the triangle inequality in $C^{m-1,1}(\R^n)$.
\end{proof}

\begin{rmk}\label{rmk: 2.5 cf}
Lemma \ref{lem:sigma_gamma_rel} implies the following property: If $\Gamma_\ell(x,f,M/2) \neq \emptyset$ then $P_x + \frac{M}{2} \cdot \sigma_\ell(x)\subseteq \Gamma_\ell(x,f,M)\subseteq P_x+2M\cdot \sigma_\ell(x)$ for any $P_x \in \Gamma_\ell(x,f,M/2)$. Thus, the convex set $\Gamma_\ell(x,f,M)$ is essentially a translate of a scalar multiple of the symmetric convex set $\sigma_\ell(x)$.

Similarly, the convex set $\Gamma_S(x,f,M)$ is essentially a translate of a scalar multiple of the symmetric convex set $\sigma_S(x)$.
\end{rmk}

\begin{proposition}[cf. Lemma 2.11 of \cite{coordinateFree}]\label{prop:WC}
There exists a controlled constant $A_0 \geq 1$ such that for any $S \subseteq E$ and $z \in \R^n$, the set $\sigma_S(z) \subseteq \Po$ is $A_0$-Whitney convex at $z$.
\end{proposition}
\begin{proof}
We follow the proof of Lemma 2.11 in \cite{coordinateFree}, which gives the desired result for a constant $A_0$ determined by $m$, $n$. The proof uses the existence of a cutoff function $\theta \in C^{m-1,1}(\R^n)$, with $\supp(\theta) \subseteq B(z, \delta/2)$, $\theta \equiv 1$ on a neighborhood of $z$, and $\| \theta \| \leq C_\theta \delta^{-m}$. Following the proof in \cite{coordinateFree}, we learn that $A_0$ is bounded by the product of a finite number (independent of $m,n$) of the constants $C_\theta$, $C$ in Lemma 2.2 of \cite{coordinateFree}, and $C_T$ in Taylor's theorem. By Proposition \ref{lem:TT} and Lemmas \ref{lem:jetprodbd}, \ref{lem:theta_bound} of the present paper, these constants may be taken to be controlled constants. Thus, $A_0$ is a controlled constant.
\end{proof}

Our next result relates the finiteness hypothesis $\cF\cH(k^\#,M)$ on $f$ (see \eqref{eqn:FH}) to the convex sets $\Gamma_\ell(x,f,M)$, and establishes a ``quasicontinuity property'' of the indexed families $\Gamma_\ell$ and $\sigma_\ell$.

\begin{lemma}[cf. Lemma 2.6 in \cite{coordinateFree}, and Lemmas 10.1, 10.2 in \cite{F1}]\label{lem:gamma_trans}
If $x\in\R^n$, $(D+1)^{\ell+1}\le k^\#$, and $M > 0$, then
\[
f \mbox{ satisfies } \cF\cH(k^\#,M) \Longrightarrow \Gamma_{\ell}(x,f,M)\ne \emptyset.
\]
Furthermore, if $x,y\in\R^n$, $\ell \ge 1$, $\delta\ge|x-y|$, and $M>0$, then
\begin{equation}\label{eqn:quasicont}
\begin{aligned}
    \Gamma_\ell(x,f,M) &\subseteq \Gamma_{\ell-1}(y,f,M) + C_T M \cdot \cB_{x,\delta}\\
    \sigma_\ell(x)&\subseteq \sigma_{\ell-1}(y)+C_T\cdot\cB_{x,\delta},
    \end{aligned}
\end{equation}
where $\cB_{x,\delta}$ is the closed unit ball in the $|\cdot |_{x,\delta}$-norm on $\Po$.
\end{lemma}
\begin{proof}
Note that $\Gamma_{\ell}(x,f,M)\ne \emptyset$ $\iff$ $\Gamma_{\ell}(x,f/M,1)\ne \emptyset$. Further, $f$ satisfies $\cF\cH(k^\#,M)$ $\iff$ $f/M$ satisfies $\cF\cH(k^\#) := \cF \cH(k^\#,1)$. Thus, for the first part of the lemma, we reduce matters to the case $M=1$. This result is stated in Lemma 2.6 of \cite{coordinateFree}. The proof is a straightforward application of Helly's theorem.

The second part of the lemma is stated in Lemma 2.6 of \cite{coordinateFree}. We refer the reader there for the proof, also using Helly's theorem.
\end{proof}

We define a notion of transversality in $\Po$ with respect to the $\langle \cdot, \cdot \rangle_{x,\delta}$ inner product.

\begin{defi}\label{def:trans2}
Given a closed, symmetric, convex set $\Omega \subseteq \Po$, a subspace $V \subseteq \Po$, $R\ge 1$, $x\in \R^n$, and $\delta > 0$, we say that $\Omega$ is $(x,\delta, R)$-transverse to $V$ if  (1) $\cB_{x,\delta} / V \subseteq R \cdot (\Omega \cap \cB_{x, \delta})/V$, and (2) $\Omega \cap V \subseteq R \cdot \cB_{x,\delta}$.
\end{defi}

\begin{rmk}\label{rmk:trans}
We note that $\Omega$ is $(x,\delta, R)$-transverse to $V$ if $\Omega$ is $R$-transverse to $V$ with respect to the Hilbert space structure $(\Po, \langle\cdot,\cdot\rangle_{x,\delta})$. To see this, we use the  formulation of transversality in a Hilbert space given in Corollary \ref{cor:eq_trans}.

We note that $\Omega$ is $R$-transverse to $V$ at $x$ (in the notation of Definition \ref{defn:trans1}) if and only if $\Omega$ is $(x,1,R)$-transverse to $V$. Again, see Corollary \ref{cor:eq_trans}.
\end{rmk}

\begin{lemma}[cf. Lemma 3.7 in \cite{coordinateFree}] \label{lem:tau trans}
If $\Omega$ is $(x,\delta,R)$-transverse to $V$, then the following holds.
\begin{itemize}
\item $\tau_{x,r}(\Omega)$ is $(x, \delta/r,R)$-transverse to $\tau_{x,r}(V)$.
\item If $\delta'  \in [\kappa^{-1} \delta, \kappa \delta]$ for some $\kappa \geq 1$, then $\Omega$ is $(x,\delta',\kappa^m R)$-transverse to $V$.
\end{itemize}
\end{lemma}
\begin{proof}
For the first bullet point: Apply $\tau_{x,r}$ to both sides of (1) and (2) in Definition \ref{def:trans2} and use the scaling relation \eqref{eqn:ball_scale_id} which states that $\tau_{x,r} \cB_{x,\delta} = \cB_{x,\delta/r}$.

For the second bullet point: In conditions (1) and (2) in Definition \ref{def:trans2}, use the inclusions $\cB_{x,\delta} \subseteq  \max \{ 1, (\delta/\delta')^m \} \cB_{x,\delta'}$ and $\cB_{x,\delta'} \subseteq  \max \{ 1, (\delta'/\delta)^m \} \cB_{x,\delta}$ from \eqref{eqn:ball_scale}, and the property that $A \cap rB \subseteq r (A \cap B)$ if $A,B$ are symmetric convex sets and $r \geq 1$.
\end{proof}

\begin{lemma}[cf. Lemma 3.8 in \cite{coordinateFree}]\label{lem:sigmaStab}
There exists a controlled constant $0<c_1<1$ such that the following holds. Let $V\subseteq\Po$ be a subspace, $x,y\in\R^n$, $\delta>0$, and $R\ge 1$. If $\sigma_E(x)$ is $(x,\delta,R)$-transverse to $V$ and $|x-y|\le c_1 \frac{\delta}{R}$, then $\sigma_E(y)$ is $(y,\delta,8R)$-transverse to $V$.
\end{lemma}
\begin{proof}
The proof of Lemma 3.8 in \cite{coordinateFree} gives the desired result for a constant $c_1$ determined by $m$, $n$. This proof uses two conditions on $c_1$: First, that $c_1 < \frac{1}{4 C_T}$, with $C_T$ the controlled constant in Taylor's theorem. Second, the following claim is used: If  $|x - y | \leq c_1 \delta$ and $c_1$ is sufficiently small then $\frac{9}{10}\cB_{x,\delta} \subseteq \cB_{y,\delta} \subseteq \frac{10}{9} \cB_{x,\delta}$. To verify this claim, we apply Lemma \ref{lem:poly1}. We learn that if $c_1 < \frac{1}{9 C_{\ref{lem:poly1}}}$, with $C_{\ref{lem:poly1}}$ the controlled constant $C$ in Lemma \ref{lem:poly1}, then $|P|_{x,\delta}$ and $|P|_{y,\delta}$ differ by a factor of at most $\frac{10}{9}$ for $|x-y| \leq c_1 \delta$. This implies the desired inclusions for the unit balls $\cB_{x,\delta}$ and $\cB_{y,\delta}$. We choose the controlled constant $c_1 < \min \{ \frac{1}{4 C_T}, \frac{1}{9 C_{\ref{lem:poly1}}}\}$ so as to satisfy the conditions for this proof.
\end{proof}

\begin{lemma}[cf. Lemma 2.9 of \cite{coordinateFree}]\label{lem:sigma_inc}
There exists a controlled constant $C^0 \ge 1$ so that, for any ball $B\subseteq \R^n$ and $z \in \frac{1}{2} B$, we have
\[
\sigma_{E \cap B}(z)\cap \B_{z,\diam(B)} \subseteq C^0 \cdot \sigma_E(z).
\]
\end{lemma}
\begin{proof}
The proof of Lemma 2.9 in \cite{coordinateFree} gives the desired inclusion for a constant $C^0$ determined by $m$, $n$. This proof uses the existence of a cutoff function $\varphi \in C^{m-1,1}(\R^n)$, with $\supp(\varphi) \subseteq B$, $\varphi \equiv 1$ on a neighborhood of $z$, and $\| \varphi \| \leq C_\varphi \delta^{-m}$ (for $\delta = \diam(B)$). Following this proof, we learn that $C^0$ is bounded by the product of a finite number (independent of $m,n$) of the constants $C_\varphi$, $C$ in Lemma 2.2 of \cite{coordinateFree}, and $C_T$ in Taylor's theorem. By Proposition \ref{lem:TT} and Lemmas \ref{lem:jetprodbd}, \ref{lem:theta_bound} of this paper, these constants may be taken to be controlled constants. Thus, $C^0$ is a controlled constant.
\end{proof}

\begin{lemma}\label{lem:sigma_int}
Let $S \subseteq E$, for $E \subseteq \R^n$ finite. 

For $z \in \R^n$, let  $I_z := \{ P \in \Po : P(z) = 0\}$ be the codimension 1 subspace of $\Po$ consisting of polynomials vanishing at $z$.

If $z \in \R^n \setminus S$ then $\sigma_S(z)$ has non-empty interior in $\Po$.

If $z \in S$ then $\sigma_S(z) \subseteq I_z$ and $\sigma_S(z)$ has non-empty (relative) interior in $I_z$.
\end{lemma}
\begin{proof}
By translation invariance, it suffices to assume $z = 0$.

Suppose $z=0 \notin S$. Consider the basis $\{m_\alpha(x) = x^\alpha \}_{\alpha \in \cM}$ for $\Po$. We shall demonstrate there exists $\epsilon > 0$ so that $\pm \epsilon m_\alpha \in \sigma_S(0)$ for all $\alpha \in \cM$. Given that $0 \in \R^n \setminus S$, there exists $\delta > 0$ so that $B(0,\delta)$ is disjoint from $S$. Let $\theta : \R^n \rightarrow \R$ be a $C^\infty$ cutoff function satisfying $\theta \equiv 1$ in a neighborhood of $0$, and $\supp(\theta) \subseteq B(0,\delta)$. For $\alpha \in \cM$ and $\epsilon > 0$, let $\varphi_\alpha^\pm(x) := \pm \epsilon m_\alpha(x) \theta(x)$. If $\epsilon > 0$ is picked small enough then $\| \varphi_\alpha^\pm \|_{C^{m-1,1}(\R^n)} \leq 1$. Note that $\varphi_\alpha^\pm$ vanishes on $S$, because $\theta$ vanishes on $S$. Finally, we have $J_0(\varphi_\alpha^\pm) = \pm \epsilon m_\alpha$. Thus, $\pm \epsilon m_\alpha \in \sigma_S(0)$ for all $\alpha \in \cM$. Therefore, $0 \in \Po$ is an interior point of $\sigma_S(0)$.

Suppose $z=0 \in S$. Let $I_0 = \{ P \in \Po : P(0) = 0\}$. Any function $\varphi \in C^{m-1,1}(\R^n)$ of seminorm $\leq 1$ that vanishes on $S$ must satisfy $\varphi(0)= 0$, hence, $J_z(\varphi) \in I_0$. We deduce that $\sigma_S(0) \subseteq I_0$. Consider the basis $\{m_\alpha(z) = x^\alpha \}_{\alpha \in \cM^+}$ for $I_0$, where $\cM^+ := \cM \setminus \{0 \}$ is the set of all nonzero multiindices of order at most $m-1$. Fix $\delta > 0$ so that $B(0,\delta) \cap S = \{0\}$. Let $\theta : \R^n \rightarrow \R$ be a $C^\infty$ cutoff function satisfying $\theta \equiv 1$ in a neighborhood of $0$, and $\supp(\theta) \subseteq B(0,\delta)$. Evidently, $\theta$ vanishes on $S \setminus \{0\}$. For $\alpha \in \cM^+$ and $\epsilon > 0$, let $\varphi_\alpha^\pm(x) := \pm \epsilon m_\alpha(x) \theta(x)$. If $\epsilon > 0$ is picked small enough then $\| \varphi_\alpha^\pm \|_{C^{m-1,1}(\R^n)} \leq 1$. We check that $\varphi_\alpha^\pm = 0$ on $S$. Indeed, $\varphi_\alpha^\pm(0) =0$ because $m_\alpha(0) = 0$ for $\alpha \in \cM^+$; meanwhile, $\varphi_\alpha^\pm$ vanishes on $S \setminus \{0\}$ because $\theta$ vanishes on $S \setminus \{0\}$.  Finally, we have $J_0(\varphi_\alpha^\pm) = \pm \epsilon m_\alpha$. Therefore, $\pm \epsilon m_\alpha \in \sigma_S(0)$ for all $\alpha \in \cM^+$. We deduce that $0 \in I_0$ is an interior point of $\sigma_S(0)$ in $I_0$.

\end{proof}

We finish the section by proving a version of Lemma 8.3 in \cite{coordinateFree} with controlled constants.

\begin{lemma}[cf. Lemma 8.3 of \cite{coordinateFree}] \label{fip_lem}
Let $C_0 \geq 1$ and $\ell_0 \in \N$. Let $\cW$ be a Whitney cover (see Definition \ref{defn:whit_cover}) of a ball $\widehat{B} \subseteq \R^n$, and let $N := \# \cW < \infty$. Suppose the following condition is valid for every $B \in \cW$:
\begin{equation}\label{eqn:fin_hyp}
\Gamma_{\ell_0}(x,f,M) \subseteq \Gamma_{E \cap \frac{6}{5}B}(x,f, C_0 M), \;\; \mbox{ for all } x \in (6/5)B, \; M  > 0.
\end{equation}
Then a corresponding condition is valid on $\widehat{B}$:
\begin{equation}\label{eqn:fin_con}
\Gamma_{\ell_1}(x_0,f,M) \subseteq \Gamma_{E \cap\widehat{B}}(x_0,f, C_1 M), \;\; \mbox{ for all } x_0 \in \widehat{B}, \; M  > 0.
\end{equation}
The constants $C_1,\ell_1$ in \eqref{eqn:fin_con} are given by $C_1 := C'  C_0$ and $\ell_1 :=  \ell_0 + \lceil \frac{\log(D \cdot N + 1)}{\log(D+1)} \rceil$, for a controlled constant $C'$. In particular, $C_1$ is independent of the cardinality $N$ of the cover $\cW$.
\end{lemma}

\begin{proof}

Let $f : E \rightarrow \R$ and $M>0$. Fix a point $x_0 \in \widehat{B}$. Our goal is to prove \eqref{eqn:fin_con} for  $C_1 \geq 1$ to be determined below.

For each $B \in \cW$, we fix $x_B \in (6/5)B$ satisfying
\begin{equation}\label{pts11}
x_B = x_{0} \iff x_{0} \in (6/5) B.
\end{equation}
(If $x_0 \notin (6/5) B$ then we take $x_B$ to be an arbitrary element of $(6/5)B$.)

Fix an arbitrary $P \in \Gamma_{\ell_1} (x_0, f,M)$. We will prove that $P \in \Gamma_{E \cap \widehat{B}}(x_0, f, C_1 M)$. To do so, we  define a family of auxiliary convex sets to which we apply Helly's theorem and obtain the conclusion. These convex sets will belong to the vector space $\Po^\cW$ consisting of tuples of $(m-1)$-st order Taylor polynomials indexed by elements of the cover $\cW$.  The vector space $\Po^\cW$ has dimension $J := \dim(\Po^\cW) = N \cdot D$. For each $S \subseteq E$, the convex set $\cK_{(f,P)}(S, M) \subseteq \Po^\cW$ is defined by
\[
\begin{aligned}
\cK_{(f,P)}(S, M) := \{ (J_{x_B} F)_{B \in \cW} : F \in C^{m-1,1}(\R^n), \; & \| F \| \leq M, \\
& F = f \mbox{ on } S, J_{x_0} F  = P \}.
\end{aligned}
\]
If $\#(S) \leq (D+1)^{\ell_1}$ then $P \in \Gamma_{\ell_1}(x_0,f,M) \subseteq \Gamma_S(x_0,f,M)$. Therefore, there exists $F \in C^{m-1,1}(\R^n)$ with $\| F \| \leq M$, $F = f$ on $S$, and $J_{x_0} F = P$. Hence, $(J_{x_B} F)_{B \in \cW} \in \cK_{(f,P)}(S,M)$. Thus, $\cK_{(f,P)}(S,M) \neq \emptyset$ if $\#(S) \leq (D+1)^{\ell_1}$.

If $S_1,\cdots, S_{J+1} \subseteq E$, then
\[
\bigcap_{j=1}^{J+1} \cK_{(f,P)}(S_j,M) \supseteq \cK_{(f,P)}( S, M), \mbox{ for } S = S_1 \cup \cdots \cup S_J.
\]
If also $\#(S_j) \leq (D+1)^{\ell_0}$ for all $j$, then $\#(S) \leq J (D+1)^{\ell_0} \leq (D+1)^{ \ell_1}$, by definition of $\ell_1$. Consequently, by the previous remark, $\cK_{(f,P)}(S,M) \neq \emptyset$. Thus, given subsets $S_1,\cdots,S_{J+1} \subseteq E$, with $\#(S_j) \leq (D+1)^{\ell_0}$ for all $j$, we have
\[
\bigcap_{j=1}^{J+1} \cK_{(f,P)}(S_j,M) \neq \emptyset.
\]
Therefore, since $\dim(\Po^\cW) = J$, by Helly's theorem,
\[
\cK :=  \bigcap_{\substack{S \subseteq E \\ \#(S) \leq (D+1)^{\ell_0}}} \cK_{(f,P)}(S,M) \neq \emptyset.
\]
Fix  $(P_B)_{B \in \cW}$ in $\cK$. By definition of the sets $\cK_{(f,P)}(S,M)$, the following condition holds: 
\begin{equation}\tag{$*$}\label{eq: cond *}
\begin{rcases}
    \text{For any }S \subseteq E \text{ with } \#(S) \leq (D+1)^{\ell_0}, \text{ there exists a function }\\ 
    F^S \in C^{m-1,1}(\R^n) \text{ with } \| F^S \| \leq M, F^S = f \text{ on } S, J_{x_0} F^S = P,\\
    \text{ and }J_{x_B} F^S = P_B \text{ for all } B \in \cW.
    \end{rcases}
\end{equation}

Using Condition \eqref{eq: cond *} we establish the following properties: For all $B,B' \in \cW$,
\begin{enumerate}
\item[(a)] $P_B = P$ if $x_0 \in \frac{6}{5} B$.
\item[(b)] $| P_B - P_{B'} |_{x_B, \diam(B)} \leq C^1 M$ if $\frac{6}{5} B \cap \frac{6}{5} B' \neq \emptyset$, for the controlled constant $C^1 := 11^m C_T$.
\item[(c)] There exists $F_B \in C^{m-1,1}(\R^n)$ such that $\|F_B \| \leq C_0 M$, $F_B = f$ on $E \cap \frac{6}{5}B$, and $J_{x_B} F_B = P_B$.
\end{enumerate}

For the proofs of (a) and (b), consider the function $F^\emptyset$ arising in \eqref{eq: cond *} for $S = \emptyset$. For the proof of (a), fix $B \in \cW$ with $x_0 \in \frac{6}{5} B$. Then $x_B = x_0$ by \eqref{pts11}, and $P_B = J_{x_B} F^\emptyset = J_{x_0} F^\emptyset = P$ by \eqref{eq: cond *}, which yields (a). For the proof of (b), suppose $\frac{6}{5} B \cap \frac{6}{5} B' \neq \emptyset$ for $B,B' \in \cW$. Note that $x_B \in \frac{6}{5} B$, $x_{B'} \in \frac{6}{5} B'$, and by the definition of a Whitney cover, $\diam(B)$ and $\diam(B')$ differ by a factor of at most $8$. Therefore, $|x_B - x_{B'} | \leq \frac{6}{5} \diam(B) + \frac{6}{5} \diam(B') \leq 11 \diam(B)$. Thus, by \eqref{eqn:norm_scale},  Taylor's theorem (rendered in the form \eqref{eqn:Taylor}), and \eqref{eq: cond *}, 
\[
\begin{aligned}
| P_B - P_{B'} |_{x_B, \diam(B) } &\leq 11^m | P_B - P_{B'} |_{x_B, 11 \diam(B) } \\
&= 11^m | J_{x_B} F^\emptyset - J_{x_{B'}} F^\emptyset |_{x_B, 11 \diam(B)} \\
&\leq 11^m C_T \| F^\emptyset \| \leq C^1 M.
\end{aligned}
\]

For the proof of (c), note that \eqref{eq: cond *} implies $P_B \in \Gamma_{\ell_0}(x_B, f,M)$ for all $B \in \cW$. Thus, by assumption \eqref{eqn:fin_hyp}, $P_B \in \Gamma_{E \cap \frac{6}{5}B} (x_B, f, C_0 M)$ for each $B \in \cW$. Then, by definition of the set $\Gamma_S$ in \eqref{eqn:gamma_S}, we complete the proof of (c). 

Let $\{\theta_B\}$ be a partition of unity adapted to the Whitney cover $\cW$, as in Lemma \ref{lem: p of u}, and set $F := \sum_{B \in \cW} \theta_B F_B$. We refer the reader to Lemma \ref{lem: p of u} for the conditions on $\{\theta_B\}$ used below. By properties (b), (c), and Lemma \ref{lem:glue}, we have (A) $F = f$ on $E \cap \widehat{B}$ and (B) $\| F \|_{C^{m-1,1}(\widehat{B})} \leq C C^1 C_0 M \leq C' C_0 M$ for controlled constants $C$, $C'$.  Since $\supp \theta_B \subseteq \frac{6}{5} B$, $J_{x_0} \theta_B = 0$ if $x_0 \notin \frac{6}{5} B$; on the other hand,  $J_{x_0} F_B = J_{x_B} F_B = P_B = P$ if $x_0 \in \frac{6}{5} B$ by \eqref{pts11} and properties (a), (c). Therefore, by a term-by-term comparison of sums we obtain the identity
\[
J_{x_0} F  =  \sum_{B \in \cW} J_{x_0} \theta_B \odot_{x_0} J_{x_0} F_B =  \sum_{B \in \cW} J_{x_0} \theta_B \odot_{x_0} P.
\]
Recall that $\sum_{B \in \cW} \theta_B = 1$ on $\widehat{B}$ and $x_{0} \in \widehat{B}$. Thus, $\sum_{B \in \cW} J_{x_0} \theta_B =  J_{x_0} (1) = 1$. Therefore, (C) $J_{x_0} F = P$. 

By an outcome of the classical Whitney extension theorem (see Lemma \ref{lem:dom_ext}), we extend $F \in C^{m-1,1}(\widehat{B})$ to  $F_0 \in C^{m-1,1}(\R^n)$ satisfying $F_0 = F$ on $\widehat{B}$ and
\[
\| F_0 \|_{C^{m-1,1}(\R^n)} \leq C \| F \|_{C^{m-1,1}(\widehat{B})}\leq C C' C_0 M.
\]
for a controlled constant $C \geq 1$. Then $\| F_0 \|_{C^{m-1,1}(\R^n)} \leq  C'' C_0 M $ for $C'' := C C'$  a controlled constant. Because $F_0 = F$ on $\widehat{B}$, properties (A) and (C) of $F$ imply that $F_0 = f$ on $E \cap \widehat{B}$ and $J_{x_0} F_0 = P$. Since $\| F_0 \|_{C^{m-1,1}(\R^n)} \leq  C'' C_0 M $, we deduce that $P \in \Gamma_{E \cap \widehat{B}}(x_0, f, C'' C_0 M)$. This proves \eqref{eqn:fin_con} with $C_1 = C'' C_0$.
\end{proof}

\section{Making Linear Selections}\label{sec:MLS}

Fix a finite set $E \subseteq \R^n$. This section contains additional properties of the sets $\Gamma_\ell(x,f,M)$ and $\sigma_\ell(x)$, defined in Section \ref{sec:basicconv}, that will be used in the construction of the linear extension operator $T$ in Theorem \ref{thm:sharpfinitenessprinciple}.

Below, the seminorm of $\varphi \in C^{m-1,1}(\R^n)$ is denoted by $\| \varphi \| := \| \varphi \|_{C^{m-1,1}(\R^n)}$.

\begin{lemma}[Theorem 1.3 of \cite{Braz}]\label{lem: sharp helly}
Let $\cF$ be a finite collection of symmetric convex sets in $\R^d$. Suppose $0$ is an interior point of each $\cK\in\cF$. Then there exist $\cK_1,\dots, \cK_{2d}\in\cF$, with
\[
\cK_1\cap\dots\cap\cK_{2d}\subseteq 2 \sqrt{d} \left(\bigcap_{\cK\in\cF}\cK\right).
\]
\end{lemma}

\begin{lemma}\label{lem:sigma_assist}
Fix $\ell \in \N$. For each $y \in \R^n$ there exists a set $S^y \subseteq E$ such that $\#(S^y) \leq 2D (D+1)^\ell$ and $\sigma_{S^y}(y) \subseteq 2 \sqrt{D} \sigma_\ell(y)$.
\end{lemma}
\begin{proof}
Recall that 
\begin{equation}
    \label{eqn:sigma_inter}
    \sigma_\ell(y) = \bigcap \left\{ \sigma_S(y): S \subseteq E, \; \#(S) \leq (D+1)^\ell \right\}. 
\end{equation}

Suppose first that $y \notin E$. Then $y \notin S$ for all $S \subseteq E$. By Lemma \ref{lem:sigma_int} the sets $\sigma_S(y)$ have nonempty interior in the $D$-dimensional vector space $\Po$. Thus we can apply Lemma \ref{lem: sharp helly} to the collection of sets $\sigma_S(y) \subseteq \Po$ for $S\subseteq E$ with $\#(S) \le (D+1)^\ell$ to get $S_1,\dots,S_{2D} \subseteq E$ such that $\#(S_i) \le (D+1)^\ell$ for each $i$ and the following inclusion holds:
\[
\bigcap_{i=1}^{2D}\sigma_{S_i}(y) \subseteq 2 \sqrt{D} \cdot \sigma_{\ell}(y).
\]
Let $S^y = S_1\cup\dots\cup S_{2D}$. Then $\sigma_{S^y}(y)\subseteq\sigma_{S_i}(y)$ for each $i$ and so 
\[
\sigma_{S^y}(y) \subseteq 2 \sqrt{D}\cdot \sigma_{\ell}(y).
\]
Furthermore, $\#(S^y) \leq 2 D (D+1)^\ell$, as claimed.

Suppose instead that $y \in E$. Then $y \in S_0$ for some $S_0 \subseteq E$ with $\#(S_0) \leq (D+1)^\ell$. By Lemma \ref{lem:sigma_int}, the set $\sigma_{S_0}(y)$ is contained in the $(D-1)$-dimensional subspace $I_y = \{ P \in \Po : P(y) = 0 \}$ of $\Po$. But $\sigma_\ell(y) \subseteq \sigma_{S_0}(y)$, so $\sigma_\ell(y)$ is contained in $I_y$. Set $\overline{\sigma}_S(y) = \sigma_S(y) \cap I_y$ for $S \subseteq E$. Intersecting both sides of \eqref{eqn:sigma_inter} with $I_y$, we have 
\[
\sigma_\ell(y) = \bigcap \left\{ \overline{\sigma}_S(y) : S \subseteq E, \; \#(S) \leq (D+1)^\ell \right\}.
\]
By Lemma \ref{lem:sigma_int}, for each $S \subseteq E$ either $\sigma_S(y)$ has nonempty interior in $\Po$ (if $y \notin S$) or $\sigma_S(y)$ has nonempty interior in $I_y$ (if $y \in S$). Therefore, $\overline{\sigma}_S(y)$ has nonempty interior in $I_y$ for all $S \subseteq E$. Thus we can apply Lemma \ref{lem: sharp helly} to the collection of sets $\overline{\sigma}_S(y) \subseteq I_y$ for $S\subseteq E$ with $\#(S) \le (D+1)^\ell$ to get $S_1,\dots,S_{2(D-1)} \subseteq E$ such that $\#(S_i) \le (D+1)^\ell$ for each $i$ and the following inclusion holds:
\begin{equation}\label{eqn:sigma_ass1}
\bigcap_{i=1}^{2(D-1)} \overline{\sigma}_{S_i}(y) \subseteq 2 \sqrt{D} \cdot \sigma_{\ell}(y).
\end{equation}
Since $\sigma_{S_0}(y) \subseteq I_y$, we have
\begin{equation}\label{eqn:sigma_ass2}
\bigcap_{i=1}^{2(D-1)} \overline{\sigma}_{S_i}(y) = I_y \cap \left( \bigcap_{i=1}^{2(D-1)} \sigma_{S_i}(y)  \right)  \supseteq \bigcap_{i=0}^{2(D-1)} \sigma_{S_i}(y).
\end{equation}
Let $S^y = S_0 \cup S_1\cup\dots\cup S_{2(D-1)}$. Then $\sigma_{S^y}(y)\subseteq\sigma_{S_i}(y)$ for each $i = 0, 1, \dots, 2(D-1)$ and so, combining \eqref{eqn:sigma_ass1} and \eqref{eqn:sigma_ass2},
\[
\sigma_{S^y}(y) \subseteq 2 \sqrt{D}\cdot \sigma_{\ell}(y).
\]
Furthermore, $\#(S^y) \leq (2 (D-1) + 1) (D+1)^\ell \leq 2D (D+1)^\ell$, as claimed.

\end{proof}

\begin{lemma}\label{lem: linear select}
Fix $y\in\R^n$ and $\ell \in \N$. There exists a linear map $P_{\ell}^y:C(E)\rightarrow\Po$ such that if $f \in C(E)$ satisfies $\cF\cH(k^\#, M)$ for some $k^\# \ge (D+1)^{\ell+3}$ and $M>0$, then $P_{\ell}^y (f) \in \Gamma_{\ell}(y,C_{\ell} M)$. Here, $C_{\ell} = C' (D+1)^{\ell}$ for a controlled constant $C'$.
\end{lemma}
\begin{proof}
By Lemma \ref{lem:sigma_assist}, there exists $S^y \subseteq E$ with $\#(S^y) \leq 2D (D+1)^\ell$ such that 
\begin{equation}\label{eqn:sigma_assist}
\sigma_{S^y}(y) \subseteq 2 \sqrt{D} \cdot \sigma_{\ell}(y).
\end{equation}
Let $S^y\cup\{y\} = \{x_1,\dots,x_N\}$, with $x_N = y$. Then 
\begin{equation}
\label{eqn:Nbd}
N = \#(S^y \cup \{y\}) \leq 2D(D+1)^\ell + 1 \leq (D+1)^{\ell+2}.
\end{equation}
Introduce the vector space $\Po^N$ of all
\[
\vec{P} = (P_\mu)_{1\le \mu \le N} \quad \text{with } P_\mu \in \Po \mbox{ for all } \mu.
\]
We define a quadratic function $\cQ$ on $\Po^N$ by
\begin{equation}\label{eq: linear selection eq0}
\cQ(\vec{P}) :=  \sum_{\mu \ne \nu}\sum_{|\beta|\le m-1}\frac{|\partial^\beta(P_\mu - P_\nu)(x_\mu)|^2}{(\beta !)^2 |x_\mu - x_\nu|^{2(m-|\beta|)}} = \sum_{\mu \ne \nu} |P_\mu-P_\nu|_{x_\mu,|x_\mu-x_\nu|}^2.
\end{equation}
Given a function $f\in C(E)$, we define $W_f$ to be the subspace of $\Po^N$ consisting of $\vec{P}\in\Po^N$ satisfying $P_\mu(x_\mu) = f(x_\mu)$ for all $1\le \mu \le N-1$ and $P_N(x_N) = f(x_N)$ if $x_N = y \in E$. Note that $\cQ$ achieves a minimum on $W_f$ at some point $\vec{P}(f,y)\in W_f$ that depends linearly on $f$ for fixed $y$. Letting $P_\mu(f,y)\in\Po$ denote the $\mu$-th component of $\vec{P}(f,y)$, we define
\[
P_{\ell}^y(f):=P_N(f,y).
\]
We've constructed a linear map $P_{\ell}^y:C(E)\rightarrow \Po$; it remains to show that $P_{\ell}^y (f) \in \Gamma_{\ell}(y,f, C_\ell M)$, with $C_\ell$ as in the statement of the lemma, whenever $f$ satisfies $\cF\cH(k^\#, M)$ for some $k^\# \ge (D+1)^{\ell+3}$ and $M>0$.

To this end, suppose $f$ satisfies $\cF\cH(k^\#, M)$ for $k^\# \ge (D+1)^{\ell+3}$ and $M>0$. We will demonstrate that there exists a function $\widetilde{F} \in C^{m-1,1}(\R^n)$ satisfying
\begin{align}
    &\|\widetilde{F}\|\le C'\cdot(D+1)^\ell M,\label{eqn:lin_sel0}\\
    &\widetilde{F} = f \text{ on } S^y, \text{ and} \label{eqn:lin_sel1}\\
    &J_y(\widetilde{F}) = P_{\ell}^y(f)\label{eqn:lin_sel2}
\end{align}
 for a controlled constant $C'$.
 
 First, we claim that $\cQ(\vec{P}(f,y)) \le  C_T^2 (D+1)^{2\ell+4} M^2$. By \eqref{eqn:Nbd}, $\#(S^y \cup \{y\}) = N\leq k^\#$. By assumption, $f$ satisfies $\cF\cH(k^\#, M)$, so there exists a function $\widehat{F}$ satisfying
\begin{align}
    &\|\widehat{F}\|\le M,\label{eqn:lin_sel3}\\
    &\widehat{F} = f  \text{ on } S^y, \text{ and} \label{eqn:lin_sel3.1}\\
    &\widehat{F}(y) = f(y) \quad \text{if } y \in E. \label{eqn:lin_sel3.2}
\end{align}
Define $\vec{R} := (R_\mu)_{1\le \mu\le N}$ where $R_\mu := J_{x_\mu}(\widehat{F})$ and $ \{x_\mu\}_{1\leq \mu \leq N}= S^y \cup \{y\}$. Then $\vec{R} \in W_f$, due to \eqref{eqn:lin_sel3.1} and \eqref{eqn:lin_sel3.2}. By Taylor's theorem (see \eqref{eqn:Taylor}), $\vec{R}$ satisfies
\begin{equation}\label{eq: linear selection eq4}
    |R_\mu-R_\nu|_{x_\mu,|x_\mu-x_\nu|} \le C_T \| F \| \leq C_T M \quad \text{for all } \mu \ne \nu.
\end{equation}
We use  \eqref{eq: linear selection eq0} and \eqref{eq: linear selection eq4}, and then \eqref{eqn:Nbd}, to get
\[
\cQ(\vec{R}) \le N^2 \cdot (C_T M)^2 \leq C_T^2 (D+1)^{2\ell+4} M^2.
\]
Since $\vec{P}(f,y)$ was chosen to minimize $\cQ$ on $W_f$, we have
\begin{equation}\label{eq: linear selection eq4a}
\cQ(\vec{P}(f,y))\le C_T^2 (D+1)^{2\ell+4} M^2,
\end{equation}
as claimed. 

From \eqref{eq: linear selection eq4a} we have
\begin{equation}\label{eq: linear selection eq5}
\begin{aligned}
    |\partial^\beta(P_\mu(f,y) - P_\nu(f,y))(x_\mu)| \le  \;& C (D+1)^{\ell} M  |x_\mu - x_\nu|^{m-|\beta|}\\
    &\text{for } \mu \ne \nu, \; |\beta|\le m-1,
    \end{aligned}
\end{equation}
for a controlled constant $C$. Since \eqref{eq: linear selection eq5} holds, the classical Whitney extension theorem (see Proposition \ref{prop:cwet}) guarantees the existence of a function $\widetilde{F}\in C^{m-1,1}(\R^n)$ satisfying  $J_{x_\mu} \widetilde{F} = P_\mu(f,y)$ for $\mu=1,2,\dots,N$, and $\| \widetilde{F} \|_{C^{m-1,1}(\R^n)} \leq C_{Wh} C (D+1)^{\ell} M$. Here, $C_{Wh}$ is a controlled constant. Thus, the function $\widetilde{F}$ satisfies \eqref{eqn:lin_sel0}. Furthermore, \eqref{eqn:lin_sel1}  follows because $J_{x_\mu} \widetilde{F} = P_\mu(f,y)$ for all $\mu$, and $\vec{P}(f,y) = (P_\mu(f,y))_{1 \leq \mu \leq N} \in W_f$. Finally, \eqref{eqn:lin_sel2} follows because $J_{y} (\widetilde{F}) = J_{x_N} (\widetilde{F}) = P_N(f,y) = P^y_\ell(f)$. This completes the proof of \eqref{eqn:lin_sel0}-\eqref{eqn:lin_sel2}.

Given that $f$ satisfies $\cF\cH(k^\#,M)$ for $k^\# \geq (D+1)^{\ell+3}$, we apply Lemma \ref{lem:gamma_trans} to deduce that $\Gamma_{\ell+2}(y,f,M) \neq \emptyset$. 

Fix $P_0^y \in \Gamma_{\ell+2}(y,f,M)$. Given that $\#(S^y) \leq (D+1)^{\ell+2}$ (see \eqref{eqn:Nbd}), we have $P_0^y \in \Gamma_{S^y}(y,f,M)$.

From \eqref{eqn:lin_sel0}-\eqref{eqn:lin_sel2}, we have that $P^y_\ell(f) \in \Gamma_{S^y}(y,f,C' (D+1)^\ell M)$. By Lemma \ref{lem:sigma_gamma_rel} we deduce that $P_0^y - P^y_\ell(f) \in  C'' (D+1)^{\ell} M \sigma_{S^y}(y)$ for a controlled constant $C'' = C' + 1$.

By  \eqref{eqn:sigma_assist}, $P_0^y - P_{\ell}^y(f) \in \widehat{C}  (D+1)^{\ell} M \sigma_{\ell}(y)$ for a controlled constant $\widehat{C}$.

But $P_0^y \in \Gamma_{\ell+2}(y,f,M) \subseteq \Gamma_\ell(y,f,M)$. By  Lemma \ref{lem:sigma_gamma_rel}, we deduce that
\[
\begin{aligned}
P_\ell^y(f) = P_0^y + (P_\ell^y(f) - P_0^y) &\in  \Gamma_\ell(y,f,M) + \widehat{C}  (D+1)^{\ell} M \sigma_{\ell}(y) \\
&\subseteq \Gamma_\ell(y,f, \overline{C} (D+1)^\ell M),
\end{aligned}
\]
for a controlled constant $\overline{C} = \widehat{C} + 1$. This proves the lemma with $C_\ell = \overline{C}  (D+1)^{\ell}$.

\end{proof}

\begin{lemma}\label{lem:perturb}
Suppose $X$ is a $d$-dimensional Hilbert space with norm $| \cdot |$. Let $\B$ denote the unit ball of $X$. Let $V$ be a subspace of $X$ and let $\Omega \in \cK(X)$ be a symmetric convex set in $X$. Suppose that $\B / V \subseteq R (\Omega \cap \B)/V$. Then there exists a linear mapping $T : X \rightarrow X$ such that $\| T \|_{op} \leq dR$, $x - Tx \in V$ and $Tx \in d R |x| ( \Omega \cap \B)$ for all $x \in X$.
\end{lemma}
\begin{proof}
Let $\{e_j : 1 \le j \le d\}$ be an orthonormal basis for $X$. Given that $\B / V \subseteq R (\Omega \cap \B)/V$, for each $e_j \in \B$ we can find $\omega_j \in R(\Omega \cap \B)$ such that $e_j - \omega_j \in V$. In particular, $|\omega_j | \leq R$ for all $j$.

Given $x \in X$, write $x = \sum_j c_j e_j$ for $c_j = \langle x, e_j \rangle$ and define $Tx := \sum_j c_j \omega_j$. Note $\max_j |c_j|\leq (\sum_j c_j^2)^{1/2} = |x|$.

We have $x - Tx = \sum_j c_j (e_j - \omega_j) \in V$.  Also, by the triangle inequality,
\[
|Tx| \leq \max_j |c_j| \cdot \sum_{j=1}^d |\omega_j| \leq  Rd |x|.
\]
Thus, $\| T \|_{op} \leq Rd$, as desired. Using that $\omega_j \in R (\Omega \cap \B)$ and $|c_j| \leq |x|$ for all $j$, and by symmetry and convexity of $\Omega \cap \B$, 
\[
Tx = \sum_{j=1}^d c_j \omega_j \in \sum_{j=1}^d |c_j| \cdot R ( \Omega \cap \B) \subseteq d R |x| (\Omega \cap \B).
\]
This completes the proof.
\end{proof}

\begin{lemma} \label{lem: linear select2}
Fix $x,y\in\R^d$, $\ell\in\N$, $R \geq 1$, $C_1 \geq 1$, $\delta \geq |x-y|$, and a DTI subspace $V\subseteq \Po$ such that $\sigma(x)$ is $(x,C_1\delta,R)$-transverse to $V$. Suppose that $f$ satisfies $\cF\cH(k^\#,M)$ for $k^\# \ge (D+1)^{\ell+2}$ and $M> 0$. Let $P_0 \in \Gamma_{\ell}(y,f,M)$. Then there exists a constant $\widehat{C}_\ell \geq 1$ and $P' \in \Gamma_{\ell-1}(x,f,\widehat{C}_\ell M)$ such that
\begin{enumerate}
    \item $P'-P_0 \in V$,
    \item $P'-P_0 \in \widehat{C}_\ell M \cB_{x,\delta}$,
    \item $P'$ depends linearly on $f$ and $P_0$,
    \item $\widehat{C}_\ell = (RD+2) \cdot C_1^m\sqrt{C_T^2 + 4DC_{\ell-1}^2}$, where $C_{\ell-1} = C' (D+1)^{\ell-1}$ is the constant arising in Lemma \ref{lem: linear select}.
\end{enumerate}
\end{lemma}

\begin{proof}
We apply Lemma \ref{lem: linear select} to find a linear map $P^y_{\ell-1} : C(E) \rightarrow \Po$. Given  that $f$ satisfies $\cF\cH(k^\#, M)$ for $k^\#\ge (D+1)^{\ell+2}$, we have $P_{\ell-1}^{x}(f) \in \Gamma_{\ell-1}(x,f,C_{\ell-1} M)$. 

By Lemma \ref{lem:gamma_trans}, and $\delta \geq |x-y|$, $\Gamma_{\ell}(y,f,M)\subseteq \Gamma_{\ell-1}(x,f,M) + C_TM\cB_{x,\delta}$. Thus, given that $P_0 \in \Gamma_\ell(y,f,M)$, there exists $Q \in \Gamma_{\ell-1}(x,f,M)$ with 
\begin{equation}\label{eqn:error0}
|P_0 - Q|_{x,\delta} \leq C_T M.
\end{equation}
By Lemma \ref{lem:sigma_gamma_rel}, 
\begin{equation}\label{eqn:error1}
Q - P_{\ell-1}^x(f)\in (C_{\ell-1} + 1) M \sigma_{\ell-1}(x) \subseteq 2 C_{\ell-1}M\sigma_{\ell-1}(x).
\end{equation}
Since $\sigma_{\ell-1}(x)$ is a closed symmetric convex set, there exists a vector subspace $V_{\ell-1}^x\subseteq \Po$ and a quadratic form $q_{\ell-1}^x$ on $V_{\ell-1}^x$ such that $\E := \{ x \in V_{\ell-1}^x : q_{\ell-1}^x \le 1 \}$ satisfies $\E \subseteq  \sigma_{\ell-1}(x) \subseteq \sqrt{D}\cdot \E$. This is a consequence of the John ellipsoid theorem (see Proposition \ref{prop:john}). Here, $V_{\ell-1}^x$ is the linear span of $\sigma_{\ell-1}(x)$, and $\E$ is the John ellipsoid of $\sigma_{\ell-1}(x)$ in $V_{\ell-1}^x$. By \eqref{eqn:error1}, 
\begin{equation}\label{eqn:error2}
\begin{aligned}
&Q - P_{\ell-1}^x(f) \in V_{\ell-1}^x, \\
&q_{\ell-1}^x(Q-P_{\ell-1}^x(f))\le 4 D C_{\ell-1}^2 M^2.
\end{aligned}
\end{equation}

We let $Q^* \in \Po$ be the minimizer of the quadratic function
\[
q_0(R):= q_{\ell-1}^x(R - P_{\ell-1}^x(f)) + |P_0-R|^2_{x,\delta},
\]
for $R \in \Po$ ranging in the affine subspace $P_{\ell-1}^x(f) + V_{\ell-1}^x$. Then $Q^*$ depends linearly on $P_0$ and $f$, and $Q^* \in P_{\ell-1}^x(f) + V_{\ell-1}^x$. Due to \eqref{eqn:error0} and \eqref{eqn:error2},
\[
q_0(Q) \le 4D C_{\ell-1}^2 M^2 + C_T^2 M^2 = \bar{C}_\ell^2 M^2,
\]
with $\bar{C}_\ell = \sqrt{4D C_{\ell-1}^2 + C_T^2}$, and $Q \in P^x_{\ell-1}(f) + V^x_{\ell-1}$. Thus, by definition of $Q^*$ as the minimizer of $q_0$ on $P^x_{\ell-1}(f) + V^x_{\ell-1}$, $q_0(Q^*) \le q_0(Q) \leq \bar{C}_\ell^2 M^2$, and thus 
\[
q_{\ell-1}^x(Q^* - P_{\ell-1}^x(f)) \le \bar{C}_\ell^2 M^2 \mbox{ and } |P_0 - Q^*|_{x,\delta}\le \bar{C}_\ell M.
\]
These inequalities imply $Q^* - P_{\ell-1}^x(f) \in \bar{C}_\ell M \sigma_{\ell-1}(x)$ and $P_0 - Q^*\in \bar{C}_\ell M \cB_{x,\delta}$. By Lemma \ref{lem:sigma_gamma_rel},
\begin{equation}\label{eqn:error2.5}
\begin{aligned}
Q^* = P_{\ell-1}^x(f) + (Q^* - P_{\ell-1}^x(f)) &\in \Gamma_{\ell-1}(x,f,C_{\ell-1} M) + \bar{C}_\ell M \sigma_{\ell-1}(x)\\
&\subseteq \Gamma_{\ell-1}(x,f,2\bar{C}_\ell M)
\end{aligned}
\end{equation}
(we've used that $\bar{C}_\ell = \sqrt{4 D C_{\ell-1}^2 + C_T^2} > C_{\ell-1}$). We've succeeded in producing $Q^*\in \Gamma_{\ell-1}(x,f,2\bar{C}_\ell M)$ satisfying $P_0 - Q^*\in \bar{C}_\ell M \cB_{x,\delta}$ and $Q^*$ depends linearly on $(P_0,f)$. It remains to  modify $Q^*$ to obtain a polynomial $P'$ such that $P'$ satisfies the same properties (potentially for larger constants) and $P' - P_0 \in V$.

Since $\sigma(x)$ is $(x,C_1\delta,R)$-transverse to $V$ and $\sigma(x) \subseteq \sigma_{\ell-1}(x)$,
\[
\cB_{x,C_1\delta}/V \subseteq R(\sigma(x) \cap \cB_{x,C_1 \delta})/V \subseteq R (\sigma_{\ell-1}(x)\cap\cB_{x,C_1\delta})/V.
\]
We equip the vector space $\Po$ with the inner product $\langle \cdot, \cdot \rangle_{x,C_1 \delta}$. Then $\cB_{x,C_1\delta}$ is the corresponding unit ball of $X$. By the above inclusion and Lemma \ref{lem:perturb} there exists a linear map $T : \Po \rightarrow \Po$ satisfying 
\begin{align}
    \label{eqn: v cond 1}
    &|T \widetilde{P} |_{x, C_1 \delta} \leq R D |\widetilde{P}|_{x,C_1 \delta}, \\
    \label{eqn: v cond 2}
    & T \widetilde{P} \in RD | \widetilde{P}|_{x,C_1 \delta} ( \sigma_{\ell-1}(x) \cap \cB_{x,C_1\delta}),\\
    \label{eqn: v cond 3}
    & T \widetilde{P} - \widetilde{P} \in V \quad \mbox{for all } \widetilde{P} \in \Po.
\end{align}
Given that $P_0 - Q^* \in \bar{C}_\ell M \cB_{x,\delta}$, we find that
\begin{equation}\label{eqn:error3}
|P_0 - Q^*|_{x,C_1 \delta}  \leq | P_0 - Q^*|_{x,\delta} \leq \bar{C}_\ell M.
\end{equation}

We set $P' = Q^* + T(P_0 - Q^*)$. Then $P'$ depends linearly on $(P_0,f)$. By \eqref{eqn: v cond 1} and \eqref{eqn:error3}, we have $|T(P_0 - Q^*)|_{x, C_1 \delta} \leq RD \bar{C}_\ell M $. Thus,
\[
|P' - P_0|_{x, C_1 \delta} \leq |Q^* - P_0|_{x, C_1 \delta} + |T(P_0 - Q^*)|_{x, C_1 \delta} \leq \bar{C}_\ell M + RD\bar{C}_\ell M.
\]
Therefore,
\[
P' - P_0 \in (RD+1) \bar{C}_\ell M \cB_{x,C_1\delta} \subseteq \widehat{C}_\ell M \cB_{x,\delta}
\]
with $\widehat{C}_\ell := (2 + RD) \bar{C}_\ell C_1^m $. Here, the last set inclusion uses \eqref{eqn:ball_scale}.

By \eqref{eqn: v cond 3}, we have
\[
P' - P_0 = (Q^* - P_0) - T(Q^* - P_0) \in V.
\]

Finally, by \eqref{eqn:error2.5}, \eqref{eqn: v cond 2}, and \eqref{eqn:error3}, we have
\[
\begin{aligned}
P' = Q^* + T(P_0 - Q^*) &\in \Gamma_{\ell-1}(x,f,2 \bar{C}_\ell M ) + RD |P_0 - Q^*|_{x,C_1 \delta} \sigma_{\ell-1}(x) \\
& \subseteq \Gamma_{\ell-1}(x,f,2 \bar{C}_\ell M ) + RD \bar{C}_\ell M  \sigma_{\ell-1}(x) \\
& \subseteq \Gamma_{\ell-1}(x,f, (2 \bar{C}_\ell + RD \bar{C}_\ell)M ) \subseteq  \Gamma_{\ell-1}(x,f, \widehat{C}_\ell M ),
\end{aligned}
\]
where the second to last inclusion uses Lemma \ref{lem:sigma_gamma_rel}.

This completes the proof of the lemma.

\end{proof}

\section{The Local Main Lemma}\label{sec:localMain}

Let $E \subseteq \R^n$ be a finite set. By Proposition \ref{prop:WC}, $\sigma(z) := \sigma_E(z)$ is $A_0$-Whitney convex at $z$ for all $z \in \R^n$. Here, $A_0 \geq 1$ is a controlled constant. By Proposition \ref{prop:tech1} with $A=A_0$ we find a constant $R_0 = O(\exp( \poly(D) \log(A_0)))$ such that 
\begin{equation}\label{eqn:wc_trans}
\begin{aligned}
&\mbox{if } \Omega \subseteq \cR_x \mbox{ is } A_0\mbox{-Whitney convex at } x \in \R^n \\
&\mbox{then there exists a DTI subspace } V \subseteq \cR_x \\
&\mbox{such that } \Omega \mbox{ is } R_0 \mbox{-transverse to } V \mbox{ at } x.
\end{aligned}
\end{equation} 
The constant $A_0$ is controlled, so $\log(A_0) = O(\poly(D))$, thus $R_0 = O(\exp(\poly(D))$, so $R_0$ is also controlled. Let $c_1$ be the controlled constant from Lemma \ref{lem:sigmaStab}. Define new controlled constants $R_4 \geq R_3 \geq R_2 \geq R_1 \geq R_0$ and $\bar{C}$ as follows. 
\begin{equation}
\label{fix_c}
\begin{aligned}
&R_1 := 8 R_0, \;\;  R_2 := D^{2D+1/2} R_1^{4D}, \;\; R_3 := 10^{m} R_2, \;\; R_4 := 8^{m+1} R_3 \\
&\bar{C} = 100 c_1^{-1} R_3
\end{aligned}
\end{equation}

\begin{lemma}\label{lem: label}
Let $B$ be a closed ball in $\mathbb R^n$. There exists a DTI subspace $V\subseteq\Po$ such that $\sigma(z)$ is $(z, \bar{C} \diam(B), R_1)$-transverse to $V$ for all $z \in 100 B$.
\end{lemma}

\begin{proof}
    Let $x_0$ be the center of $B$. We shall use the following property: If $\Omega \subseteq \Po$ is $A$-Whitney convex at $x_0$, then $\tau_{x_0,\delta}(\Omega)$ is $A$-Whitney convex at $x_0$. (See Lemma \ref{lem:wc_props} for the corresponding property when $x_0=0$.) By Proposition \ref{prop:WC}, $\sigma(x_0)$ is $A_0$-Whitney convex at $x_0$, thus, $\tau_{x_0, (\bar{C}\diam(B))^{-1}}(\sigma(x_0))$ is $A_0$-Whitney convex at $x_0$. Thanks to \eqref{eqn:wc_trans}, there is a DTI subspace $V$ such that $\tau_{x_0,(\bar{C}\diam(B))^{-1}}(\sigma(x_0))$ is $R_0$-transverse to $V$ at $x_0$. Thus, $\tau_{x_0,(\bar{C}\diam(B))^{-1}}(\sigma(x_0))$ is $(x_0, 1, R_0)$-transverse to $V$. Therefore, by Lemma \ref{lem:tau trans}, $\sigma(x_0)$ is $(x_0, \bar{C} \diam(B), R_0)$-transverse to $\tau_{x_0, \bar{C}\diam(B)}(V) = V$, where the set equality holds because $V$ is DTI (in particular, $V$ is dilation invariant at $x_0$). Given $z \in 100 B$ (arbitrary), we have $|z-x_0| \leq 100 \diam(B) \leq c_1 \frac{\bar{C} \diam(B)}{R_0}$ (observe that $100 = c_1 \frac{\bar{C}}{R_3} \leq c_1 \frac{\bar{C}}{R_0}$). By Lemma \ref{lem:sigmaStab}, we conclude that $\sigma(z)$ is $(z, \bar{C} \diam(B), 8 R_0)$-transverse to $V$. This completes the proof of the lemma.
\end{proof}

\begin{defi}
     Given a ball $B \subseteq \mathbb R^n$ and finite set $E \subseteq \mathbb R^n$, the local complexity of $E$ on $B$ is the integer quantity
    \[
        \cC(E|B) = \sup_{x \in B} \cC_{x}(\sigma(x), R_1, R_2, \bar{C} \diam(B)).
    \]
\end{defi}

See Definition \ref{jet_complexity_def} for the definition of the pointwise complexity $\cC_x(\Omega,R,R^*,\delta)$ of a  symmetric convex set $\Omega \subseteq \cR_x$ at $x$ at scale below $\delta$. Evidently, pointwise complexity is monotone in $\delta$ in the sense that $\cC_x(\Omega,R,R^*,\delta) \leq \cC_x(\Omega,R,R^*,\delta')$ for $\delta \leq \delta'$. This implies the following monotonicity property of local complexity.
\begin{cor}
If $B_1 \subseteq B_2$, then $\cC(E|B_1) \leq \cC(E|B_2)$.
\end{cor} 

Due to the relation $R_2 = D^{2D+1/2} R_1^{4D}$ and inequality $R_1 \geq 16$ (see \eqref{fix_c}), we can apply Proposition \ref{prop:tech2} to deduce the following result:

\begin{cor}\label{cor:ballCompBound}
    For any ball $B \subseteq \mathbb R^n$ and finite set $E \subseteq \mathbb R^n$, $\cC(E|B) \leq 4 m D^2$.
\end{cor}

We provide an equivalent formulation of complexity in the next result. 

\begin{lemma}\label{lem:comp_restate}
Let $E \subseteq \R^n$ (finite), a ball $B \subseteq \R^n$, and an integer $J \geq 1$ be given. Then $\cC(E| B) \geq J$ if and only if there exists $x \in B$, and there exist subspaces $V_j \subseteq \Po$ and intervals $I_j \subseteq (0, \diam(B)]$ ($j =1,2,\dots,J$), such that the following conditions hold.
\begin{itemize}
    \item $I_1 > I_2 > \dots > I_J > 0$.
    \item $\tau_{x,r(I_j)}\sigma(x)$ is $(x,\bar{C},R_1)$-transverse to $V_j$.
    \item $\tau_{x,l(I_j)}\sigma(x)$ is not $(x,\bar{C},R_2)$-transverse to $V_j$.
    \item $V_j$ is dilation invariant at $x$.
\end{itemize}
\end{lemma}
\begin{proof}
Evidently, $\cC(E|B) \geq J$ if and only if $\cC_x(\sigma(x), R_1,R_2, \bar{C} \diam(B)) \geq J$ for some $x \in B$. By Definition \ref{jet_complexity_def}, the second inequality is equivalent to the  assertion: There exist subspaces $V_1,\dots,V_J \subseteq \Po$ and intervals $\widetilde{I}_1 >  \dots > \widetilde{I}_J > 0$ satisfying that, for all $j$,
\begin{itemize}
\item $\tau_{x,r(\widetilde{I}_j)}\sigma(x)$ is $(x,1,R_1)$-transverse to $V_j$.
\item $\tau_{x,l(\widetilde{I}_j)}\sigma(x)$ is not $(x,1,R_2)$-transverse to $V_j$.
\item $\widetilde{I}_j \subseteq (0, \bar{C} \diam(B)]$.
\item $V_j$ is dilation invariant at $x$.
\end{itemize}
Here, in the application of Definition \ref{jet_complexity_def}, we use that a convex set $\Omega$ is $(x,1,R)$-transverse to $V$ if and only if $\Omega$ is $R$-transverse to $V$ at $x$ (see Remark \ref{rmk:trans}).

We apply the first conclusion of Lemma \ref{lem:tau trans} (for $r= \bar{C}^{-1}$) to the first two bullet points above. We learn that these conditions are respectively equivalent to the following: \begin{itemize}
\item $\tau_{x,r(\widetilde{I}_j)/\bar{C}}\sigma(x)$ is $(x,\bar{C},R_1)$-transverse to $\tau_{x, \bar{C}^{-1}}V_j$.
\item $\tau_{x,l(\widetilde{I}_j)/\bar{C}}\sigma(x)$ is not $(x,\bar{C},R_2)$-transverse to $\tau_{x, \bar{C}^{-1}}V_j$.
\end{itemize}
Because $V_j$ is dilation invariant at $x$, we have $\tau_{x, \bar{C}^{-1}}V_j = V_j$. Let $I_j := \{ \delta/\bar{C} : \delta \in \widetilde{I}_j\}$, so that $l(I_j) = l(\widetilde{I}_j)/\bar{C}$ and $r(I_j) = r(\widetilde{I}_j)/\bar{C}$. Then $I_1 > \dots > I_J > 0$, and $I_j \subseteq (0,\diam(B)]$ for all $j$. The previous two bullet points are equivalent to the assertion that $\tau_{x,r(I_j)}\sigma(x)$ is $(x,\bar{C},R_1)$-transverse to $V_j$, and $\tau_{x,l(I_j)}\sigma(x)$ is not $(x,\bar{C},R_2)$-transverse to $V_j$. This completes the proof of the lemma.
\end{proof}

We will see that Theorem \ref{thm:sharpfinitenessprinciple} is a consequence of the following:

\begin{lemma}[Local Main Lemma for $K$] \label{lml}

Let $K \in \Z$ with $K \geq -1$. There exist constants $C^\# = C^\#(K) \geq 1$ and $\ell^\# = \ell^\#(K) \in \Z_{\geq 0}$, depending only on $K,m,n$, with the following properties. 

Fix a finite set $E \subseteq \R^n$, a closed ball $B_0 \subseteq \R^n$, and a point $x_0 \in B_0$.

Suppose $\cC(E|5B_0) \leq K$. Then there exists a linear map $T:C(E)\times \Po\rightarrow C^{m-1,1}(\R^n)$ such that the following holds:

Suppose $(f,P_0) \in C(E) \times \Po$ and $M > 0$ satisfy that $P_0 \in \Gamma_{\ell^\#}(x_0,f,M)$, or equivalently, by \eqref{eqn:gamma_ell}, the following condition holds: For all $S \subseteq E$ with $\#(S) \leq (D+1)^{\ell^\#}$ there exists $F^S \in C^{m-1,1}(\R^n)$ with $F^S = f$ on $S$, $J_{x_0} F^S = P_0$, and $\|F^S \|_{C^{m-1,1}(\R^n)} \leq M$. 

Then $T(f,P_0) = f$ on $E \cap B_0$, $J_{x_0} (T(f,P_0)) = P_0$, and $\|T(f,P_0) \|_{C^{m-1,1}(\R^n)} \leq C^\# M$.

Here, $C^\#(K) = \Lambda^{(K+1)^2+1}$ and $\ell^\#(K) = \overline{\chi} \cdot (K+1)$ for all $K \geq -1$, where $\Lambda \geq 1$ is a controlled constant ($O(\exp(\poly(D)))$) and $\overline{\chi} \in \N$ is $O(\poly(D))$.

\end{lemma}

\begin{rmk} \label{reform_lfip}
The conclusion of the Local Main Lemma for $K$ implies that $P_0 \in \Gamma_{E \cap B_0}(x_0,f,C^\#M)$ as long as $\cC(E|5 B_0) \leq K$ and $P_0 \in \Gamma_{\ell^\#}(x_0,f,M)$. To see this, take $F = T(f,P_0)$ in the definition of $\Gamma_{E \cap B_0}(\cdots)$. Thus, we derive the following as a consequence of the Local Main Lemma for $K$: If $\cC(E|5 B_0) \leq K$ then for any $f \in C(E)$ and $M > 0$,
\[
\Gamma_{\ell^\#}(x_0,f,M) \subseteq  \Gamma_{E \cap B_0} (x_0, f, C^\# M) \quad \mbox{ for any } x_0 \in B_0.
\]
In particular, by taking $f = 0$ and $M=1$,
\[
\sigma_{\ell^\#}(x_0) \subseteq C^\# \cdot \sigma_{E \cap B_0}(x_0)\quad \mbox{ for any } x_0 \in B_0.
\]
Here, $C^\# = C^\#(K)$ and $\ell^\# = \ell^\#(K)$ are as in the Main Lemma for $K$.
\end{rmk}

The layout of the rest of the paper is as follows.

In Section \ref{sec_ind} we give the proof of Lemma \ref{lml} by induction on $K$. Then, in Section \ref{sec:mainproofs}, we apply Lemma \ref{lml} to prove the main extension theorems: Theorem \ref{thm:sharpfinitenessprinciple} (for finite $E$) and Theorems \ref{thm: new c sharp} and \ref{thm: lin op} (for arbitrary $E$).

\section{The Main Induction Argument}\label{sec_ind}
We prove Lemma \ref{lml} by induction on $K \in \{-1,0,\cdots, K_0\}$. Here $K_0 = 4mD^2$ is a universal upper bound on the local complexity $\cC(E|B)$; see Corollary \ref{cor:ballCompBound}. In this section, we write the seminorm of $\varphi \in C^{m-1,1}(\R^n)$ as $\| \varphi \| := \| \varphi \|_{C^{m-1,1}(\R^n)}$.

\subsection{Setup}\label{sec_ind:setup}

Because $\cC(E|B) \geq 0$ for any $E$ and $B$, the Local Main Lemma for $K=-1$ is true vacuously; we take $C^\#(-1) = \Lambda$ and $\ell^\#(-1) = 0$ when $K=-1$. This establishes the base case of the induction.

For the induction step, fix $K \in \{0,1,\cdots, K_0 \}$. Let $E \subseteq \R^n$ be finite. We assume the inductive hypothesis that the Local Main Lemma for $K-1$ is true. Let $\ell_\old := \ell^\#(K-1)$ and $C_\old := C^\#(K-1)$ be the \emph{finiteness constants} arising in the Local Main Lemma for $K-1$. Given any ball $B$ in $\R^n$, we apply the Local Main Lemma for $K-1$ to the ball $(6/5)B$ to obtain:
\begin{equation}\label{ind_hyp}
\begin{aligned}
&\mbox{If }  x \in (6/5) B \mbox{ and } \cC(E|6B) \leq K-1 \mbox{ then}\\
&\mbox{there exists a linear map } T_B : C(E) \times \Po \rightarrow C^{m-1,1}(\R^n) \\
&\mbox{such that if } P \in \Gamma_{\ell_{\old}}(x,f,M), \mbox{ then } T_B(f,P) = f \mbox{ on } E \cap (6/5)B, \\
&J_{x} T_B(f,P) = P, \mbox{ and }\| T_{B}(f,P) \| \leq C_{\old} M.
\end{aligned}
\end{equation}
We refer to conclusion \eqref{ind_hyp} as the \textbf{induction hypothesis}.

To prove the Main Lemma for $K$, we fix a ball $B_0 \subseteq \R^n$ with $\cC(E|5B_0) \leq K$ and a point  $x_0 \in B_0$. Our task is to construct a linear map $T : C(E) \times \Po \rightarrow C^{m-1,1}(\R^n)$ such that, for the finiteness constants $C^\#=C^\#(K)$ and $\ell^\# = \ell^\#(K)$ defined in the Local Main Lemma for $K$, the following holds:
\begin{equation}\label{main-task}
P_0 \in \Gamma_{\ell^\#}(x_0,f,M) \implies \left\{
\begin{aligned}
    &T(f,P_0) = f \mbox{ on } E \cap B_0\\
    &J_{x_0} T(f,P_0) = P_0\\
    & \| T(f,P_0) \| \leq C^\# M.
\end{aligned}
\right.
\end{equation}
From the Local Main Lemma for $K-1$ and $K$, the constants $\ell_{\old}$, $C_{\old}$, $\ell^\#$, and $C^\#$ will have the following form:
\begin{equation}\label{eqn:ind_consts}
\begin{aligned}
    &\ell_\old = \overline{\chi} \cdot K, \qquad C_{\old} = \Lambda^{K^2+1}\\
    &\ell^\#= \overline{\chi} \cdot (K+1), \; C^\# = \Lambda^{(K+1)^2+1},
\end{aligned}
\end{equation}
where $\overline{\chi} = O(\poly(D))$ and $\Lambda = O(\exp(\poly(D)))$ are suitably chosen constants, depending only on $m$ and $n$, determined in the proof of \eqref{main-task}. In particular, $\overline{\chi}$ and $\Lambda$ will be chosen independently of the induction parameter $K$. We assume that $\overline{\chi} \ge 5$, so that $\ell^\# \ge 5$. Later we will consider the sets $\Gamma_{\ell^\# - j}$ for $0\le j \le 4$; this assumption ensures that these sets are well-defined.

\begin{proposition}\label{prop:onept_fp}
Given a ball $B \subseteq \R^n$ with $\#(B \cap E) \leq 1$, and given $x \in \frac{6}{5} B$, there exists a linear map $T : C(E) \times \Po \rightarrow C^{m-1,1}(\R^n)$ satisfying the following: If $P \in \Gamma_0(x,f,M)$ then
\begin{enumerate}
\item $T(f,P) = f$ on $B \cap E$.
\item $J_{x} T(f,P) = P$.
\item $\| T(f,P) \| \leq C M$.
\end{enumerate}
Here, $C$ is a controlled constant.
\end{proposition}
\begin{proof}
If $B \cap E = \emptyset$ or if $B \cap E = \{ x \}$, we define $T(f,P) = P$. Conditions 2 and 3 are obviously true. If $B \cap E = \{x\}$ then $P \in \Gamma_0(x,f,M)$ implies that $P(x) = f(x)$, hence condition 1 of $T$ is implied by condition 2 of $T$ in this case. Else if $B \cap E = \emptyset$, then condition 1 is vacuously true.

On the other hand, suppose $B \cap E = \{z\}$ and $x \neq z$. Let $P \in \Gamma_0(x,f,M)$ Then let $\hat{B} = B(z,\frac{1}{2}|z-x|)$. We apply Lemma \ref{lem:theta_bound} to find a $C^m$ cutoff function $\theta$ with $\theta \equiv 1$ on $(1/2)\hat{B}$,  $\theta \equiv 0$ on $\R^n \setminus \hat{B}$, and $\| \partial^\alpha \theta \|_{L^\infty(\R^n)} \leq C |z-x|^{-|\alpha|}$ for $|\alpha| \leq m$, for a controlled constant $C$. 

Define $P_z \in \Po$ by the conditions $P_z(z) = f(z)$ and $\partial^\alpha P_z(z) = \partial^\alpha P_x(z)$ for all $|\alpha| \geq 1$. Then set
\[
T(f,P) = \theta P_z + (1-\theta) P = P + \theta (P_z-P).
\]
Note that $J_x \theta = 0$ because $x \notin \hat{B}$ and $\theta$ is supported on $\hat{B}$. Thus, $J_x T(f,P) = P$. Also, $\theta \equiv 1$ in a neighborhood of $z$, so $T(f,P) = f$ at the unique point $z \in E \cap B$. 

We now seek to control  
\[
\| T(f,P) \|_{\dot{C}^m(\R^n)} = \sup_{y \in \R^n} \max_{|\beta| = m} | \partial^\beta T(f,P)(y)|.
\]
Note that $T(f,P)$ agrees with the $(m-1)$'st degree polynomial $P$ on $\R^n \setminus \hat{B}$. Thus, $\partial^\beta T(f,P)(y) = 0$ for $|\beta| = m$ and $y \notin \hat{B}$. For $y \in \hat{B}$ and $|\beta| = m$, $\partial^\beta T(f,P)(y) = \partial^\beta (\theta (P_z-P))(y)$. By applying the product rule, and the derivative bounds for $\theta$, we learn that
\[
\begin{aligned}
\| T(f,P) \|_{\dot{C}^m(\R^n)} & =  \sup_{y \in \hat{B}} \max_{|\beta| = m} | \partial^\beta T(f,P)(y)| \\
&\leq C \sup_{y \in \hat{B}} \sum_{|\alpha| \leq m-1} |\partial^\alpha (P_z-P)(y)| \cdot |x-z|^{|\alpha|-m} \\
& \leq C' \sup_{y \in \hat{B}} |P_z-P|_{y,|x-z|}.
\end{aligned}
\]
By Lemma \ref{lem:poly1}, and because $|y-z| \leq |x-z|$ for $y \in \hat{B}$ (by definition of $\hat{B}$), we have $|P_z-P|_{y,|x-z|} \leq C |P_z-P|_{z,|x-z|}$ for $y \in \hat{B}$. Thus,
\[
\begin{aligned}
\| T(f,\; &P) \|_{\dot{C}^m(\R^n)} \leq C |P_z-P|_{z,|x-z|} \\
&= C  \left( \sum_{|\alpha| \leq m-1} (\alpha!)^{-2} |\partial^\alpha P_z(z) - \partial^\alpha P(z)|^2 \cdot |x-z|^{2(|\alpha|-m)} \right)^{1/2} \\
&=  C |f(z) - P(z)|,
\end{aligned}
\]
where we have used that $\partial^\alpha P_z(z) = \partial^\alpha P(z)$ for $|\alpha| \geq 1$ and $P_z(z) = f(z)$. Thus, using \eqref{eqn:Cm_norm_bd}, for a controlled constant $C'$ we have
\begin{equation}\label{eqn:onept_fp}
\| T(f,P) \|_{C^{m-1,1}(\R^n)} \leq C' |f(z) - P(z) |.
\end{equation}

Recall that $P \in \Gamma_0(x,f,M)$. Thus, by definition, for any $S \subseteq E$ with $\#(S) \leq (D+1)^0 = 1$ there exists $F^S$ with $F^S = f$ on $S$, $J_x F^S = P$, and $\| F^S \| \leq M$. Apply this condition with $S = \{z\}$. Then, there exists $F$ with $F(z) = f(z)$, $J_x F = P$ and $\| F \| \leq M$. By Taylor's theorem (see \eqref{eqn:Taylor}),
\[
|J_z F - P|_{z,|x-z|} = | J_z F - J_x F|_{z,|x-z|} \leq C_T M.\]
In particular, $|f(z) - P(z)| = |(J_z F - P)(z)| \leq |J_z F - P|_{z,|x-z|} \leq C_T M$. Using this inequality in \eqref{eqn:onept_fp}, we deduce that $\| T(f,P) \|_{C^{m-1,1}(\R^n)} \leq C M$ for a controlled constant $C$. This completes the proof of Proposition \ref{prop:onept_fp}.
\end{proof}

We assume the parameter $\Lambda$ in Lemma \ref{lml} is chosen to satisfy
\begin{equation}\label{eqn:Lam_ass1}
\Lambda \geq C, \mbox{ for the controlled constant } C \mbox{ in Proposition \ref{prop:onept_fp}}.
\end{equation}
Then $C^\# = \Lambda^{(K+1)^2 + 1} \geq C$. If $\#(B_0 \cap E) \leq 1$, we apply Proposition \ref{prop:onept_fp} to the ball $B=B_0$ and point $x_0 \in B_0$, to obtain a linear map $T : C(E) \times \Po \rightarrow C^{m-1,1}(\R^n)$. If $P_0 \in \Gamma_{\ell^\#}(x_0,f,M)$ then $P_0 \in \Gamma_0(x_0,f,M)$, so the map $T$ satisfies conditions 1,2,3 in Proposition \ref{prop:onept_fp}, implying \eqref{main-task}, for $C^\# \geq C$.

Having given the construction of $T$ and proof of \eqref{main-task} in the case $\#(B_0 \cap E) \leq 1$, we now assume that
\begin{equation}
\label{twopoints}
\#(B_0 \cap E) \geq 2.
\end{equation}
Under the assumption \eqref{twopoints}, in the remainder of Section \ref{sec_ind} we will explain how to construct a linear map  $T : C(E) \times \Po \rightarrow C^{m-1,1}(\R^n)$ and prove it satisfies \eqref{main-task}.

\subsection{The Main Decomposition Lemma}


Recall the constant $\bar{C}$, defined in \eqref{fix_c}, arises in Lemma \ref{lem: label} and in the definition of local complexity $\cC(E|B)$. Write $R_1 \leq R_2 \leq R_3 \leq R_4$ for the controlled constants defined in \eqref{fix_c}. We continue in the setting of Section \ref{sec_ind:setup}, and fix data $(B_0,x_0,E,K,f,\ell^\#,M,P_0)$. Suppose $P_0 \in \Gamma_{\ell^\#}(x_0,f,M)$ as in \eqref{main-task}.

In the next lemma we introduce a cover of the ball $2B_0$ that will be used to decompose the local extension problem on $B_0$ into a family of easier subproblems associated to the elements of the cover.

\begin{lemma}[Main Decomposition Lemma]
\label{mdl}
Given  $(B_0,x_0,E,K,f,\ell^\#,M,P_0)$ satisfying $\#(B_0 \cap E) \geq 2$,  $\cC(E|5B_0) \leq K$, $x_0 \in B_0$, and $P_0 \in \Gamma_{\ell^\#}(x_0,f,M)$, there exist a DTI subspace $V \subseteq \Po$, a Whitney cover $\cW$ of $2B_0$, and collections of polynomials $\{P_B\}_{B \in \cW} \subseteq \Po$ and points $\{z_B\}_{B \in \cW}$ such that
\begin{enumerate}
\item $\sigma(x)$ is $(x, \bar{C} \diam(B_0), R_1)$-transverse to $V$ for all $x \in 100 B_0$.
\item $B \subseteq 100 B_0$ and $\diam(B) \leq \frac{1}{2} \diam(B_0)$ for all $B \in \cW$.
\item $\sigma(x)$ is $(x, \bar{C} \delta, R_4)$-transverse to $V$ for all $x \in 8B$, $\delta \in [\diam(B),\diam(B_0)]$, $B \in \cW$.
\item  Either $\#(6B \cap E) \leq 1$ or $\cC(E|6B) < K$ for all $B \in \cW$.
\item $z_B \in \frac{6}{5}B \cap 2B_0$ for all $B \in \cW$;   if $x_0 \in \frac{6}{5}B$ then $z_B = x_0$.
\item $P_B\in \Gamma_{\ell^\#-3}(z_B,f,\bar{C}_{\ell^\#} M)$ and $P_0 - P_B \in \bar{C}_{\ell^\#} M \cB_{z_B,\diam(B_0)}$ for all $B \in \cW$; if $x_0 \in \frac{6}{5}B$ then $P_B = P_0$. Here, $\bar{C}_{\ell^\#} = C (D+1)^{ {\ell^\#}}$ for a controlled constant $C \geq 1$.
\item $P_0 - P_B \in V$ for all $B \in \cW$.
\item $P_B$ depends linearly on $(f,P_0)$ for every $B \in \cW$.
\end{enumerate}
Furthermore, the Whitney cover $\cW$, the subspace $V$, and the point set $\{z_B\}_{B \in \cW}$ depend only on the data $(B_0,x_0,E,K,\ell^\#)$ and the parameters $m,n$ -- in particular, these objects are independent of $(f,P_0)$ and $M>0$.  \end{lemma}

Using the inductive hypothesis and Proposition \ref{prop:onept_fp}, we obtain a local extension theorem on the elements of the cover $\cW$.
\begin{lemma}
\label{old_fp_lem}
For any $B \in \cW$ and $x \in \frac{6}{5}B$ there exists a linear map $T_B : C(E) \times \Po \rightarrow C^{m-1,1}(\R^n)$ satisfying the following conditions: If $P \in \Gamma_{\ell_\old}(x,f,M)$ for $M>0$ then
\begin{enumerate}
    \item $T_B(f,P) = f $ on $E \cap (6/5)B$.
    \item $J_x T_B(f,P) = P$.
    \item $\| T_B(f,P) \| \leq C_{\old} M$.
\end{enumerate}
In particular,
\begin{equation}\label{eqn:old_fp_lem}
\Gamma_{\ell_\old}(x,f,M) \subseteq \Gamma_{E \cap \frac{6}{5}B}(x,f,C_\old M).
\end{equation}
\end{lemma}
\begin{proof}
Condition 4 of Lemma \ref{mdl} states that either $\cC(E|6B) < K$ or $\#(E \cap 6B) \leq 1$. If $\cC(E|6B) < K$, the result follows from \eqref{ind_hyp}. Else if $\#(E \cap 6B) \leq 1$, the result follows from Proposition \ref{prop:onept_fp}. Here, we take $\Lambda \geq C$ so that $C_{\old} = \Lambda^{K^2+1} \geq C$ for the controlled constant $C$ in Proposition \ref{prop:onept_fp}. (See \eqref{eqn:Lam_ass1}.)  \end{proof}

\subsection{Proof of the Main Decomposition Lemma}\label{proof_mdl}
By Lemma \ref{lem: label}, there exists a DTI subspace $V \subseteq \Po$ such that 
\begin{equation}\label{eqn:cond1}
\sigma(x) \mbox{ is } (x, \bar{C} \diam(B_0), R_1)\mbox{-transverse to } V \mbox{ for all }x \in 100B_0.
\end{equation}
This proves condition 1 in the Main Decomposition Lemma.

The construction of $\cW$ is based on the following definition:

\begin{defi}\label{defn:OK}
A ball $B \subseteq 100 B_0$ is OK if $\#(B \cap E) \geq 2$ and if there exists $z \in B$ such that $\sigma(z)$ is $(z, \bar{C} \delta, R_3)$-transverse to $V$ for all $\delta \in [\diam(B), \diam(B_0)]$.
\end{defi}

The OK property is \emph{inclusion monotone} in the sense that if $B \subseteq B' \subseteq 100 B_0$ and $B$ is OK then $B'$ is OK.

For each $x \in 2B_0$, we define
\[
r(x) := \inf \{r > 0:  B(x,r)\subseteq 100 B_0, \; B(x,r) \mbox{ is OK}\}
\]
Also set 
\[
\Delta := \min \{ |x-y| : x,y \in E, x \neq y \}.
\]
Since $E$ is finite, $\Delta > 0$. 

\begin{lemma}\label{lem:rx_bd}
For all $x \in 2 B_0$, we have $0 < \Delta/2 \leq r(x) \leq \frac{3}{2} \diam(B_0)$.
\end{lemma}
\begin{proof}
Let $x \in 2 B_0$, and set $r_0 = \frac{3}{2} \diam(B_0)$. Then $B_0 \subseteq B(x,r_0) \subseteq 100 B_0$. Since $\#(B_0 \cap E) \geq 2$, we obtain $\#(B(x,r_0) \cap E) \geq 2$. Further, $\diam(B(x,r_0)) = 2r_0 > \diam(B_0)$, so the transversality condition in Definition \ref{defn:OK} holds vacuously for $B = B(x,r_0)$. Consequently, $B(x,r_0)$ is OK, and the infimum in the definition of $r(x)$ is over a set containing $r=r_0$. Thus, $r(x) \leq r_0$.

If $B(x,r)$ is OK then $\#(B(x,r) \cap E) \geq 2$, which implies $r \geq \Delta/2$ by definition of $\Delta$. Thus, $r(x) \geq \Delta/2 > 0$.
\end{proof}

Define the ball $B_x := B(x, \frac{1}{7}r(x))$ for $x \in 2B_0$. By Lemma \ref{lem:rx_bd}, we have
\begin{equation}\label{inc1}
70B_x = B(x, 10 r(x)) \subseteq 100 B_0, \quad \mbox{for } x \in 2B_0.
\end{equation}
Define the cover $\cW^* = \{B_x \}_{x \in 2B_0}$ of $2B_0$.
\begin{lemma}\label{ok_lem}
If $B \in \cW^*$ then $8B$ is OK, and $6B$ is not OK.
\end{lemma}
\begin{proof}
Write $B = B_x = B(x, \frac{1}{7} r(x))$ for $x \in 2B_0$. According to \eqref{inc1}, $6B \subseteq 8B \subseteq 100B_0$. By definition of $r(x)$ as an infimum and the inclusion monotonicity of the OK property, the result follows.
\end{proof}

We recall the Vitali covering lemma (see, for example, \cite{stein}).
\begin{lemma}[Vitali covering lemma]
Let $\widetilde{B}_1, \dots, \widetilde{B}_J$ be any finite collection of balls contained in $\R^n$. Then there exists a subcollection $\widetilde{B}_{j_1}, \widetilde{B}_{j_2}, \dots, \widetilde{B}_{j_k}$ of these balls which is pairwise disjoint and satisfies
\[
\bigcup_{j = 1}^J \widetilde{B}_j \subseteq \bigcup_{i=1}^{k} 3 \widetilde{B}_{j_i}.
\]
\end{lemma}

Because $\diam(B_x) = \frac{2}{7}r(x) \geq \Delta/7 > 0$ for all $x \in 2 B_0$ (see Lemma \ref{lem:rx_bd}), there exists a finite sequence of points $x_1,\cdots,x_J \in 2B_0$ such that $2 B_0 \subseteq \bigcup_{j=1}^J \frac{1}{3} B_{x_j}$. Applying the Vitali covering lemma to the collection $\{ \widetilde{B}_j = \frac{1}{3} B_{x_j} : j=1,\cdots,J\}$, we identify a finite subsequence $x_{j_1},\cdots, x_{j_k}$ such that $2 B_0 \subseteq \bigcup_{i=1}^k  B_{x_{j_i}}$ and $\{ \frac{1}{3} B_{x_{j_i}} : i=1,\cdots, k\}$ is pairwise disjoint. Thus we have found a finite subcover $\cW := \{ B_{x_{j_i}} :i=1,\cdots,k\} \subseteq \cW^*$ of $2 B_0$ such that the family of third-dilates $\{\frac{1}{3} B\}_{B \in \cW}$ is pairwise disjoint.

\begin{lemma}
$\cW$ is a Whitney cover of $2 B_0$.
\end{lemma}
\begin{proof}
We only have to verify  the third condition in Definition \ref{defn:whit_cover}. Suppose for sake of contradiction that there exist balls $B_j = B(x_j,r_j) \in \cW$ for $j=1,2$, with $\frac{6}{5} B_1 \cap \frac{6}{5} B_2 \neq \emptyset$ and $r_1 < \frac{1}{8} r_2$. Since $\frac{6}{5} B_1 \cap \frac{6}{5} B_2 \neq \emptyset$, we have $|x_1-x_2| \leq \frac{6}{5}r_1+ \frac{6}{5} r_2$. If  $z \in 8 B_1$ then $|z-x_1| \leq 8 r_1$, and therefore
\[
|z-x_2| \leq |z-x_1| + |x_1-x_2| \leq 8 r_1 + \frac{6}{5} r_1 + \frac{6}{5} r_2 <  r_2 + \frac{3}{20} r_2 + \frac{6}{5} r_2 \leq 6 r_2.
\] 
Hence, $8 B_1 \subseteq 6 B_2$. By Lemma \ref{ok_lem}, $8 B_1$ is OK. By inclusion monotonicity, $6B_2$ is OK. But this contradicts Lemma \ref{ok_lem}, finishing the proof of the lemma.
\end{proof}

We now establish conditions 2--8 in the Main Decomposition Lemma. 

Fix a ball $B \in \cW$. Because $6B$ is not OK, while $6B \subseteq 100 B_0$ (a consequence of \eqref{inc1}), by negation of the OK property we have:
\begin{equation}\label{eqn:pc}
\begin{aligned}
&\mbox{If } \#(6 B \cap E) \geq 2 \mbox{ then for all } x \in 6B \\
&\mbox{ there exists } \delta_x \in [ 6\diam(B), \diam(B_0)] \\
& \mbox{ so that } \sigma(x) \mbox{ is not } (x, \bar{C} \delta_{x}, R_3) \mbox{-transverse to }V.
\end{aligned}
\end{equation}

\noindent \textbf{Proof of condition 2:} Just above \eqref{eqn:pc} we noted that $B \subseteq 100B_0$. Write $B = B(x,\frac{1}{7} r(x))$ for $x \in 2 B_0$. By Lemma \ref{lem:rx_bd}, $\diam(B) = \frac{2}{7} r(x) \leq \frac{1}{2} \diam(B_0)$.

\noindent  \textbf{Proof of condition 3:} Let $x \in 8B$. Since $8B$ is OK, there exists  $z \in 8B$ such that $\sigma(z)$ is $(z, \bar{C} \delta, R_3)$-transverse to $V$ for all $\delta \in [ 8\diam(B),\diam(B_0)]$. By definition of $\bar{C}$ in \eqref{fix_c}, we have
\[
|x-z| \leq 8 \diam(B) \leq \delta \leq \frac{c_1}{R_3} \cdot (\bar{C} \delta) \qquad (\delta \in [8\diam(B),\diam(B_0)]).
\]
So, by Lemma \ref{lem:sigmaStab},
\begin{equation}\label{eqn:trans1}
\sigma(x) \mbox{ is } (x, \bar{C} \delta , 8 R_3)\mbox{-transverse to } V  \qquad (\delta \in [8\diam(B),\diam(B_0)]).
\end{equation}

First suppose $\diam(B) \leq \frac{1}{8} \diam(B_0)$. Then the interval $[8\diam(B),\diam(B_0)]$ is nonempty. Any number in  $[\diam(B),\diam(B_0)]$ differs from a number in $[8\diam(B),\diam(B_0)]$ by a factor of at most $8$. Hence, by \eqref{eqn:trans1} and the second bullet point of Lemma \ref{lem:tau trans} (for $\kappa = 8$), $\sigma(x)$ is $(x, \bar{C} \delta ,  8^{m+1} R_3)$-transverse to $V$ for all $\delta \in [\diam(B),\diam(B_0)]$. Since $R_4 = 8^{m+1} R_3$ (see \eqref{fix_c}), we obtain  condition 3 in this case.

Suppose instead that $\diam(B) > \frac{1}{8} \diam(B_0)$. We cannot use \eqref{eqn:trans1}, because  $[8\diam(B),\diam(B_0)]$ is empty. Instead we use \eqref{eqn:cond1}. Note $x \in 8 B \subseteq 100 B_0$. By \eqref{eqn:cond1}, $\sigma(x)$ is $(x,\bar{C} \diam(B_0), R_1)$-transverse to $V$. Any number in  $[\diam(B),\diam(B_0)]$ differs from  $\diam(B_0)$ by a factor of at most $8$. So, by Lemma \ref{lem:tau trans}, $\sigma(x)$ is $(x, \bar{C} \delta,  8^m R_1)$-transverse to $V$ for all $\delta \in [\diam(B),\diam(B_0)]$. Since $R_4  = 8^{m+1} R_3 \geq 8^m R_1$, this completes the proof of condition 3.

\noindent  \textbf{Proof of condition 4:}  Suppose that $\#(6B \cap E) \geq 2$ and set $J := \cC(E|6B)$. According to the definition of complexity (see the formulation given in Lemma \ref{lem:comp_restate}), there exists a point $z \in 6B$, and there exist intervals $I_1 > I_2 > \cdots > I_J > 0$ in $(0, 6  \diam(B)]$ and subspaces $V_1, V_2, \cdots, V_J \subseteq \Po$, such that, for all $j$, \\
(A) $\tau_{z,r(I_j)} ( \sigma(z))$ is $(z, \bar{C}, R_1)$-transverse to $V_j$,\\
(B) $\tau_{z,l(I_j)} ( \sigma(z))$ is not $(z, \bar{C} , R_2)$-transverse to $V_j$, and \\
(C) $V_j$ is invariant under the mappings $\tau_{z,\delta} : \Po \rightarrow \Po$ ($\delta > 0$).

Because the center of $B$ is contained in $2B_0$ and the radius of $B$ is at most half the radius of $B_0$ (see condition 2) it follows that $6B \subseteq 5 B_0$. Hence, $z \in 5B_0$. 

Condition \eqref{eqn:pc} implies the existence of $\delta_z \in [6 \diam(B), \diam(B_0)]$ so that
\begin{equation}
\label{eqn12}
\sigma(z) \mbox{ is not } (z, \bar{C}\delta_z,R_3)\mbox{-transverse to } V.
\end{equation}
Define an interval $I_0 := [ \delta_z,  \diam(B_0)]$, with endpoints $l(I_0) = \delta_z$ and $r(I_0) = \diam(B_0)$, and define a subspace $V_0 := V$. We will next demonstrate that (A) and (B) hold for $j=0$. Since $V$ is a DTI subspace, $\tau_{z,l(I_0)} V = \tau_{z,r(I_0)} V = V$. Therefore, by rescaling \eqref{eqn12},
\begin{equation}
\label{a1}
\tau_{z,l(I_0)}(\sigma(z)) \mbox { is not } (z,\bar{C},R_3)\mbox{-transverse to } V.
\end{equation}
(Here we use the first bullet point of Lemma \ref{lem:tau trans}.) Recall \eqref{eqn:cond1} states that $\sigma(z)$ is $(z, \bar{C} \diam(B_0), R_1)$-transverse to $V$. By rescaling,
\begin{equation}
\label{a2}
\tau_{z,r(I_0)}(\sigma(z)) \mbox { is } (z,\bar{C},R_1)\mbox{-transverse to } V.
\end{equation}
Conditions \eqref{a1} and \eqref{a2} imply (A) and (B) for $j=0$ (recall $R_3 \geq R_2$). Note that $V_0 = V$ is DTI, so $V_0$ is dilation invariant at $z$. Thus, (C) holds for $j=0$.

Observe that $r(I_1) \leq  6 \diam(B) \leq  \delta_z =  l(I_0)$, thus $I_1 < I_0$. Therefore,  $I_0 > I_1 > \cdots > I_J$ are subintervals of $(0, \diam(B_0)]$.

We produced intervals $I_0 > I_1 > \cdots > I_J$ in $(0, 5 \diam(B_0)]$ and subspaces $V_0,\cdots,V_J \subseteq \Po$, so that (A), (B), and (C) hold for $j = 0,1,\cdots,J$. Since $z \in 5 B_0$, by the definition of complexity (see Lemma \ref{lem:comp_restate}), we have $\cC(E|5B_0) \geq J+1$. Since $\cC(E|5B_0) \leq K$ and $J = \cC(E| 6B)$, this completes the proof of condition 4.

Next we define a collection of points $\{z_B\}_{B \in \cW} \subseteq \R^n$ and polynomials $\{P_B\}_{B \in \cW} \subseteq \Po$ and prove conditions 5--8.

To verify condition 5, fix any family  $\{z_B\}_{B \in \cW}$ satisfying $z_B \in \frac{6}{5} B \cap 2 B_0$ and $z_B = x_0$ if $x_0 \in \frac{6}{5} B$.

\noindent  \textbf{Proofs of conditions 6--8:} If $B \in \cW$ satisfies $x_0 \in \frac{6}{5}B$ then set $P_B= P_0$. Note $z_B = x_0$. Conditions 7 and 8 are trivially true. The first containment in condition 6 is true because $P_0 \in \Gamma_{\ell^\#}(x_0,f,M)$ by hypothesis, and $\Gamma_{\ell^\#}(x_0,f,M) \subseteq \Gamma_{\ell^\#-1}(x_0,f,M) \subseteq \Gamma_{\ell^\#-1}(x_0,f, \bar{C}_{\ell^\#} M) = \Gamma_{\ell^\#-1}(z_B,f, \bar{C}_{\ell^\#} M)$ for any choice of $\bar{C}_{\ell^\#} \geq 1$. The second containment in condition 6 is trivially satisfied.

Suppose now $B \in \cW$ and $x_0 \notin \frac{6}{5}B$. Note that $z_B \in \frac{6}{5}B \cap 2B_0$, and thus $|x_0-z_B| \leq \delta_0$ for $\delta_0 :=  2\diam(B_0)$.

We prepare to verify the hypotheses of Lemma \ref{lem: linear select2} for the choice of parameters $y = x_0$, $x = z_B$, $R=R_1$, $C_1 = \bar{C}/2$, $\delta = \delta_0$, and $\ell=\ell^\#-2$.

By \eqref{eqn:cond1}, $\sigma(z_B)$ is $(z_B,\frac{\bar{C}}{2} \delta_0, R_1)$-transverse to $V$. 

Given that $P_0 \in \Gamma_{\ell^\#}(x_0,f,M)$, we have the following condition (see \eqref{eqn:gamma_ell}): For every $S \subseteq E$ with $\#(S) \leq (D+1)^{\ell^\#}$ there exists $F^S \in C^{m-1,1}(\R^n)$ satisfying $F^S = f$ on $S$, $J_{x_0} F^S = P_0$, and $\| F^S \| \leq M$. In particular, $f$ satisfies $\cF\cH(k^\#,M)$ for $k^\# = (D+1)^{\ell^\#}$ (see \eqref{eqn:FH}).

Because $\Gamma_{\ell^\#}(x_0,f,M) \subseteq \Gamma_{\ell^\#-2}(x_0,f,M)$, we have $P_0 \in \Gamma_{\ell^\#-2}(x_0,f,M)$.

By Lemma \ref{lem: linear select2}, given that $P_0 \in \Gamma_{\ell^\#-2}(x_0,f,M)$, we produce a polynomial $P_B \in \Gamma_{\ell^\#-3}(z_B,f,\widehat{C}_{\ell^\#-2} M)$ such that $P_B - P_0 \in V$, $P_B-P_0 \in \widehat{C}_{\ell^\#-2} M \B_{z_B, \delta_0}$, and $P_B$ depends linearly on $(f,P_0)$, verifying conditions 7 and 8. Here,
\[
\widehat{C}_{\ell^\#-2} = (R_1 D+2) \cdot (\bar{C}/2)^m\sqrt{C_T^2 + 4DC_{\ell^\#-3}^2},
\]
with $C_{\ell^\#-3} = C' \cdot (D+1)^{\ell^\#-3}$ the constant arising in Lemma \ref{lem: linear select}, for a controlled constant $C'$. Recall that $R_1, \bar{C}, C_T$, and $D$ are controlled constants. Hence, $\widehat{C}_{\ell^\#-2} \leq C \cdot (D+1)^{\ell^\#}$ for a controlled constant $C$.

Recalling $\delta_0 = 2 \diam(B_0)$, we apply \eqref{eqn:ball_scale} to obtain
\[
P_B-P_0 \in \widehat{C}_{\ell^\#-2} M \B_{z_B, \delta_0} \subseteq \widehat{C}_{\ell^\#-2} 2^m M \B_{z_B, \diam(B_0)}.
\]
Note $\widehat{C}_{\ell^\#-2} 2^m \leq C'' \cdot (D+1)^{\ell^\#}$ for a controlled constant $C''$. We set $\bar{C}_{\ell^\#} = C'' \cdot (D+1)^{\ell^\#}$, so that $P_B - P_0 \in \bar{C}_{\ell^\#} M \B_{z_B, \diam(B_0)}$. Given that $\widehat{C}_{\ell^\#-2} \leq \bar{C}_{\ell^\#}$, we have $P_B \in \Gamma_{\ell^\#-3}(z_B,f,\widehat{C}_{\ell^\#-2} M) \subseteq \Gamma_{\ell^\#-3}(z_B,f,\bar{C}_{\ell^\#} M)$, completing the proof of condition 6. 

This finishes the proof of the Main Decomposition Lemma (Lemma \ref{mdl}).

\subsection{Upper bounds on the sets $\sigma_\ell(x)$}\label{sec:}

We continue in the setting of Section \ref{sec_ind:setup}.

We fix data $(B_0,x_0,E,K,f,\ell^\#,M,P_0)$ satisfying $\#(B_0 \cap E) \geq 2$,  $\cC(E|5B_0) \leq K$, $x_0 \in B_0$, and $P_0 \in \Gamma_{\ell^\#}(x_0,f,M)$.

We apply the Main Decomposition Lemma (Lemma \ref{mdl}) to this data 
and obtain a Whitney cover $\cW$ of $2B_0$, a DTI subspace $V \subseteq \Po$, and collections $\{P_B\}_{B \in \cW} \subseteq \Po$ and $\{z_B\}_{B \in \cW} \subseteq \R^n$, satisfying conditions 1--8 of Lemma \ref{mdl}.

Introduce a Whitney cover $\cW_0$ of $B_0$ by setting
\begin{equation}\label{defn:W0}
\cW_0 := \{ B \in \cW : B \cap B_0 \neq \emptyset\}  \subseteq \cW.
\end{equation}

Our next result provides geometric information on the sets $\sigma_\ell(x)$ for $\ell \gg \ell_{\old}$. Recall that $z_B \in \frac{6}{5} B$ for $B \in \cW$.

\begin{lemma}\label{main_lem}
There exist constants $\epsilon_0 \in (0,1)$, $\chi \geq 1$, and $C \geq 1$, determined by $m, n$, satisfying the following. Suppose there exists a ball $\widehat{B} \in \cW_0$ satisfying $\diam(\widehat{B}) \leq \epsilon_0  \cdot \diam(B_0)$. Then for any $B \in \cW_0$, $x \in 3B$, and $\ell \geq \ell_{\old} + \chi$,
\[
(\sigma_{\ell+1}(x) + \cB_{z_B, \diam(B)}) \cap V \subseteq C C_{\old} \cB_{z_B,\diam(B)}.
\]
Here, $\epsilon_0$ and $C$ are controlled constants, and $\chi = O(\poly(D))$.
\end{lemma}

Note that the constant $C_{\old} = C^\#(K-1)$ in Lemma \ref{main_lem} is not a controlled constant because it depends on $K$.


\subsubsection{Proof of Lemma \ref{main_lem}}\label{proof_main_lem}

We define constants $A \geq 10$ and $\epsilon_0 \in (0,1/300]$ as follows:
\begin{equation}\label{defn:A}
\begin{aligned}
A = 2 C^0 \cdot {\bar{C}}^m \cdot R_4, \qquad \epsilon_0 = 1/(30 A^2).
\end{aligned}
\end{equation}
Here, $C^0$ is the controlled constant in Lemma \ref{lem:sigma_inc}, and $\bar{C},R_4$ are controlled constants defined in \eqref{fix_c}. Clearly, both $A$ and $\epsilon_0$ are controlled constants.

We define 
\begin{equation}\label{defn:chi}
\chi = \lceil \log (D \cdot (180A)^n  + 1)/\log(D+1) \rceil.
\end{equation}
Since $A = O(\exp(\poly(D)))$ and $n \leq D$, we have that $\chi = O(\poly(D))$.



\begin{defi}
\label{key_defn}
A ball $B^\# \in \cW$ is keystone if $\diam(B) \geq \frac{1}{2} \diam(B^\#)$ for every $B \in \cW$ with $B \cap AB^\# \neq \emptyset$. Let $\cW^\# \subseteq \cW$ be the set of all keystone balls.
\end{defi}

Any ball $B \in \cW$ of minimal radius is a keystone ball. Because $\cW$ is finite, there exists a ball of minimal radius in $\cW$. So $\cW^\#$ is nonempty.

\begin{lemma} \label{keystone_lem}
For each ball $B \in \cW$ there exists a keystone ball $B^\# \in \cW^\#$ satisfying $B^\# \subseteq 3A B$, $\dist(B,B^\#) \leq 2A\diam(B)$, and $\diam(B^\#) \leq  \diam(B)$.
\end{lemma}
\begin{proof}
We produce a sequence of balls $B_1, B_2, \cdots, B_J \in \cW$, starting with $B_1 = B$, such that $B_{j} \cap A B_{j-1} \neq \emptyset$, $\diam(B_j) < \frac{1}{2} \diam(B_{j-1})$ for all $j \geq 2$, and $B_J$ is keystone. If $B$ is keystone, simply take a length-$1$ sequence with $B_1 = B$. Otherwise, let $B_1 = B$. Since $B_1$ is not keystone there exists $B_2 \in \cW$ with $B_2 \cap A B_1 \neq \emptyset$ and $\diam(B_2) < \frac{1}{2} \diam(B_1)$. If $B_2$ is keystone we conclude the process. Otherwise, if $B_2$ is not keystone there exists $B_3 \in \cW$ with $B_3 \cap A B_2 \neq \emptyset$ and $\diam(B_3) < \frac{1}{2} \diam(B_2)$. We continue this process until, at some step, we find a keystone ball. The process will terminate after finitely many steps because $\cW$ is finite, and $\diam(B_j)$ is decreasing in $j$.

As $B_{j} \cap A B_{j-1}\neq \emptyset$ we have $\dist(B_{j-1},B_j) \leq \frac{A}{2} \diam(B_{j-1})$. Now estimate
\[
\begin{aligned}
\dist(B_1,B_J) &\leq \sum_{j=2}^J  \dist(B_{j-1},B_j) +  \sum_{j=2}^{J-1} \diam(B_j)  \leq \left(A/2 + 1 \right) \sum_{j=1}^J \diam(B_j)  \\
& \leq  (A + 2 ) \diam(B_1) \leq 2A\diam(B_1).
\end{aligned}
\]
Since $\diam(B_J) \leq \diam(B_1)$, we deduce from the previous inequality that $B_J \subseteq (2A+6) B_1 \subseteq 3A B_1$. Set $B^\# = B_J$ to finish the proof.
\end{proof}

We prepare to define a mapping $\kappa : \cW_0 \rightarrow \cW^\#$. By hypothesis of Lemma \ref{main_lem}, there exists a ball $\widehat{B} \in \cW_0$ with $\diam(\widehat{B}) \leq \epsilon_0 \diam(B_0)$. By Lemma \ref{keystone_lem}, we can associate to $\widehat{B}$ a keystone ball $\widehat{B}^\#$ satisfying
\begin{equation}\label{eqn:keyball1}
\widehat{B}^\# \subseteq 3 A \widehat{B} \mbox{ and } \diam(\widehat{B}^\#) \leq \diam(\widehat{B}).
\end{equation}
To define $\kappa$, we proceed as follows: For each $B \in \cW_0$,
\begin{itemize}
\item If $\diam(B)  > \epsilon_0 \diam(B_0)$ ($B$ is \emph{medium-sized}), set $\kappa(B) := \widehat{B}^\#$.
\item If $\diam(B) \leq \epsilon_0 \diam(B_0)$ ($B$ is \emph{small-sized}), Lemma \ref{keystone_lem} yields a keystone ball $B^\#$ with $B^\# \subseteq 3A B$, $\dist(B,B^\#) \leq 2A\diam(B)$, and $\diam(B^\#) \leq  \diam(B)$; set $\kappa(B) := B^\#$.
\end{itemize}

We record a simple geometrical result that will be used in the analysis of $\kappa$.
\begin{lemma}\label{lem:ball_geom}
If $B \in \cW_0$ and $\diam(B) \leq \epsilon_0 \diam(B_0)$, then $3 A^2 B \subseteq 2 B_0$.
\end{lemma}
\begin{proof}
Since $B \in \cW_0$, we have $B \cap B_0 \neq \emptyset$. Thus, $3 A^2 B \cap B_0 \neq \emptyset$. Also, 
\[
\diam(3 A^2B ) \leq 3 A^2 \epsilon_0 \diam(B_0) = (1/10) \diam(B_0).
\]
Therefore, $3 A^2 B \subseteq 2 B_0$.
\end{proof}

\begin{lemma}\label{key_geom_lem}[Properties of $\kappa$]
The mapping $\kappa : \cW_0 \rightarrow \cW^\#$ satisfies the following:  For any $B \in \cW_0$, (a) $\dist(B,\kappa(B)) \leq C_4 \diam(B)$, (b) $\diam(\kappa(B)) \leq \diam(B)$, and (c) $A \cdot \kappa(B) \subseteq 2 B_0$. Here, $C_4$ is a controlled constant.
\end{lemma}
\begin{proof}
Set $C_4= 810 A^3$, which is a controlled constant. Recall that $\epsilon_0 = \frac{1}{30 A^2}$.

There exists a ball $\widehat{B} \in \cW_0$ with  $\diam(\widehat{B}) \leq \epsilon_0 \diam(B_0)$, by hypothesis of Lemma \ref{main_lem}. By Lemma \ref{lem:ball_geom},
\begin{equation}\label{eqn:ball_geom1}
3 A^2 \widehat{B} \subseteq 2 B_0.
\end{equation}

We split the proof into cases depending on whether $B \in \cW_0$ is medium-sized or small-sized.

Case 1: Suppose $B \in \cW_0$ is medium-sized, i.e., $\diam(B) > \epsilon_0 \diam(B_0)$ and $B \cap B_0 \neq \emptyset$. Then $9 (\epsilon_0)^{-1} B \supseteq 2 B_0 \supseteq \widehat{B}$; furthermore, by \eqref{eqn:keyball1},  $\widehat{B}^\# \subseteq 3A \widehat{B}$. Thus, 
\[
\widehat{B}^\# \subseteq 27 (\epsilon_0)^{-1} A B = 810 A^3 B = C_4 B.
\]
Therefore, the distance from the center of $\kappa(B) = \widehat{B}^\#$ to the center of $B$ is at most $C_4 \diam(B)$, which implies property (a). Also, from \eqref{eqn:keyball1}, 
\[
\diam(\widehat{B}^\#) \leq \diam(\widehat{B}) \leq \epsilon_0 \diam(B_0)  < \diam(B),
\]
which establishes property (b). By \eqref{eqn:ball_geom1}, \eqref{eqn:keyball1}, we have $A \widehat{B}^\# \subseteq 3 A^2 \widehat{B} \subseteq 2 B_0$, which gives (c).

Case 2: Suppose $B \in \cW_0$ is small-sized, i.e., $\diam(B) \leq \epsilon_0 \diam(B_0)$ and $B \cap B_0 \neq \emptyset$. By Lemma \ref{lem:ball_geom}, $3 A^2 B \subseteq 2 B_0$. In this case, $\kappa(B) = B^\#$, where $B^\#$ and $B$ are related via Lemma \ref{keystone_lem}. In particular,  
\[
\dist(B,B^\#) \leq  2A \diam(B) \leq C_4 \diam(B) \mbox{ and } \diam(B^\#) \leq \diam(B),
\]
yielding properties (a) and (b). Furthermore, $B^\# \subseteq 3A B$.  Thus, $A B^\# \subseteq 3 A^2 B \subseteq 2 B_0$. Thus, we have established property (c).
\end{proof}

This concludes our description of $\kappa: \cW_0 \rightarrow \cW^\#$. We will use the mapping $\kappa$ later, in the proof of Lemma \ref{main_lem}. Next we establish two lemmas describing the geometry of the sets $\sigma_\ell(x)$. The first lemma gives a stronger form of \eqref{eqn:old_fp_lem}.

\begin{lemma}\label{lemma1}
Let $B^\# \in \cW$ be a keystone ball. Suppose that $AB^\# \subseteq 2B_0$.  Let $\chi$ be defined as in \eqref{defn:chi}, and let $\ell \in \N$ with $\ell \geq \ell_{\old} + \chi$. Then
\[
\Gamma_{\ell}(x,f,M) \subseteq  \Gamma_{E \cap A B^\#}(x,f,CC_{\old} M)  \mbox{ for all } x \in A B^\#, M>0,
\]
for a controlled constant $C$. In particular, by taking $f \equiv 0|_E$ and $M=1$,
\begin{equation}\label{eqn:lemma1}
\sigma_{\ell}(x) \subseteq CC_{\old} \sigma_{E \cap A B^\#}(x) \; \; \mbox{for any } x \in A B^\#.
\end{equation}
\end{lemma}

\begin{proof}

Let $\cW(B^\#)$ be the set of all balls in $\cW$ that intersect $A B^\#$. Since $\cW$ is a Whitney cover of $2B_0$ and $AB^\# \subseteq 2 B_0$, we have that $\cW(B^\#)$ is a Whitney cover of $A B^\#$. From \eqref{eqn:old_fp_lem} we have the inclusion
\[
\Gamma_{\ell_{\old}}(x,f,M) \subseteq \Gamma_{E \cap \frac{6}{5} B}(x,f, C_{\old} M) \mbox{ for all } B \in \cW(B^\#), \; x \in (6/5)B.
\]
We apply Lemma \ref{fip_lem} to the Whitney cover $\cW(B^\#)$ of $A B^\#$, with $\ell_0 = \ell_{\old}$ and $C_0 = C_{\old}$. We deduce that
\[
\Gamma_{\ell_1}(x,f,M) \subseteq \Gamma_{E \cap A B^\#}(x,f, C_1 M)
\]
for the constants $C_1 = C \cdot C_{\old}$ and $\ell_1 = \ell_\old + \lceil \frac{\log (D \cdot N  + 1) }{\log(D+1)} \rceil$, where $N = \# \cW(B^\#)$; here, $C$ is a controlled constant.

We prepare to estimate $N = \# \cW(B^\#)$ using a volume comparison bound.

For any $B \in \cW(B^\#)$, we have $\diam(B) \geq \frac{1}{2} \diam(B^\#)$ by definition of keystone balls -- furthermore, we claim that $\diam(B) \leq 10A \diam(B^\#)$. We proceed by contradiction: Suppose  $\diam(B) > 10 A \diam(B^\#)$ for some $B \in \cW(B^\#)$. We have $B \cap AB^\# \neq \emptyset$ by definition of $\cW(B^\#)$. The previous conditions yield that $\frac{6}{5} B \cap B^\# \neq \emptyset$. Then $\diam(B) \leq 8 \diam(B^\#)$ by the properties of the Whitney cover $\cW$ (see Definition \ref{defn:whit_cover}). This completes the proof by contradiction. 

For any $B \in \cW(B^\#)$ we have $B \cap A B^\# \neq \emptyset$ and $\diam(B) \leq 10 A \diam(B^\#)$, and therefore $B \subseteq 30A B^\#$.

We estimate the volume of $\Omega := \bigcup_{B \in \cW(B^\#)}  \frac{1}{3} B$ in two ways. First, note that $\mbox{Vol}(\Omega) \leq \mbox{Vol}(30A B^\#) = (30A)^n \mbox{Vol}(B^\#)$. Since $\{\frac{1}{3} B\}_{B \in \cW}$ is pairwise disjoint (by properties of the Whitney cover $\cW$), $N= \# \cW(B^\#)$, and $\diam(B) \geq \frac{1}{2} \diam(B^\#)$ for $B \in \cW(B^\#)$, we have
\[
\mbox{Vol}(\Omega) = \sum_{B \in \cW(B^\#)} 3^{-n} \mbox{Vol}(B) \geq  N 6^{-n} \mbox{Vol}(B^\#).
\]
Thus, $N \leq (180A)^n$. By definition of $\chi$ in \eqref{defn:chi}, $\ell_1 = \ell_\old +\lceil \frac{\log (D \cdot N + 1) }{ \log (D+1)} \rceil   \leq \ell_\old + \chi \leq \ell$. Hence,
\[
\Gamma_{\ell}(x,f,M) \subseteq \Gamma_{\ell_1}(x,f,M) \subseteq \Gamma_{E \cap A B^\#}(x,f,C_1 M),
\]
as desired.
\end{proof}

\begin{lemma}\label{mainlem_1}
If $\ell \geq \ell_{\old} + \chi$, and if $B^\# \in \cW$ is a keystone ball satisfying $AB^\# \subseteq 2B_0$, then
\begin{equation}\label{eqn:mainlem_1}
\sigma_\ell(z_{B^\#})  \cap V \subseteq C  C_{\old} \cB_{z_{B^\#}, \diam(B^\#)}.
\end{equation}
Here, the constant $\chi \geq 1$ is defined in \eqref{defn:chi}, and $C \geq 1$ is a controlled constant. \end{lemma}
\begin{proof}

Let $C_0$ be the constant $C$ in Lemma \ref{lemma1}, and $C^0$ the constant in Lemma \ref{lem:sigma_inc}. Note that $z_{B^\#} \in \frac{6}{5} B^\# \subseteq \frac{1}{2} A B^\#$ (since $A \geq 10$). By condition \eqref{eqn:lemma1} in Lemma \ref{lemma1}, and Lemma \ref{lem:sigma_inc} (applied for $B = AB^\#$ and $x= z_{B^\#}$),
\begin{equation}\label{eq:1}
\begin{aligned}
\sigma_{\ell}(z_{B^\#}) &  \cap C_0 C_{\old}  \cB_{z_{B^\#}, A \diam(B^\#)}  \\
& \subseteq C_0 C_{\old} ( \sigma_{E \cap AB^\#}(z_{B^\#})\cap \cB_{z_{B^\#}, A \diam(B^\#)} ) \\
& \subseteq C^0 C_0 C_{\old} \cdot \sigma(z_{B^\#}) \qquad\qquad\qquad \mbox{for } \ell \geq \ell_{\old} + \chi.
\end{aligned}
\end{equation}

Apply condition 3 of Lemma \ref{mdl} to $B = B^\#$,  $x= z_{B^\#}$, and $\delta = \diam(B^\#)$, giving that $\sigma(z_{B^\#})$ is $(x,\bar{C} \diam(B^\#), R_4)$-transverse to $V$. By Lemma \ref{lem:tau trans}, $\sigma(z_{B^\#})$ is $(x,\diam(B^\#), \widehat{R})$-transverse to $V$ for $\widehat{R}=\bar{C}^m R_4$. Therefore, $\sigma(z_{B^\#}) \cap V \subseteq \widehat{R} \cB_{z_{B^\#},\diam(B^\#)}$. Applying this inclusion and taking the intersection with $V$ on each side of  \eqref{eq:1}, we obtain
\[
\sigma_\ell(z_{B^\#}) \cap V \cap (C_0 C_{\old} \cB_{z_{B^\#}, A \diam(B^\#)}) \subseteq C^0 C_0 C_{\old} \widehat{R}  \cB_{z_{B^\#}, \diam(B^\#)}.
\]
From \eqref{eqn:ball_scale}, $A  \cB_{z_{B^\#}, \diam(B^\#)} \subseteq \cB_{z_{B^\#},A \diam(B^\#)}$ (recall $A \geq 1$). Thus, 
\begin{equation}\label{eq:3}
\sigma_\ell(z_{B^\#}) \cap V \cap (C_0 C_{\old} A \cB_{z_{B^\#},\diam(B^\#)}) \subseteq C^0 C_0 C_{\old} \widehat{R}  \cB_{z_{B^\#}, \diam(B^\#)}.
\end{equation}
By definition of $A$ in \eqref{defn:A}, $A = 2 C^0 \bar{C}^m R_4 = 2 C^0 \widehat{R}$. Therefore, \eqref{eq:3} reads as
\[
(\sigma_\ell(z_{B^\#}) \cap V) \cap (2 C^0 C_0 C_{\old} \widehat{R}  \cB_{z_{B^\#},\diam(B^\#)}) \subseteq C^0 C_0 C_{\old} \widehat{R}  \cB_{z_{B^\#}, \diam(B^\#)}.
\]
Note that $\Omega \cap  2r \B \subseteq  r  \B \implies \Omega \subseteq r \B$, valid when $\Omega$ is a symmetric convex subset of a Hilbert space $X$ with unit ball $\B$, and $r > 0$. By this fact and the above inclusion, we have
\[
\sigma_\ell(z_{B^\#}) \cap V \subseteq C^0 C_0 C_{\old} \widehat{R} \cB_{z_{B^\#}, \diam(B^\#)}.
\]
This completes the proof of \eqref{eqn:mainlem_1} for the controlled constant $C = C^0 C_0 \widehat{R}$.

\end{proof}

We require one last lemma before the proof of our main result.

\begin{lemma}\label{lem:stabv} Let $R,Z \geq 1$ and $\lambda \geq 1$ be given. If $\Omega$ is a symmetric closed convex set in a Hilbert space $X$, $\cB$ is the closed unit ball of $X$, and $V \subseteq X$ is a subspace, satisfying (i) $ \cB/V \subseteq R \cdot (\Omega \cap \cB)/V$ and  (ii) $\Omega \cap V \subseteq Z \cB$, then
\begin{equation}\label{eqn:201}
(\Omega + \lambda \cB) \cap V \subseteq Z \cdot (3R \lambda + 1) \cB.
\end{equation}
\end{lemma}
\begin{proof}
Fix $P \in (\Omega + \lambda \cB) \cap V$. Write $P = P_0 + P_1$ with $P_0 \in  \Omega$ and $P_1 \in \lambda \cB$. Since $P_1 \in \lambda \cB$, there exists $P_2 \in R \lambda (\Omega \cap \cB)$ with $P_1 - P_2 \in V$ by condition (i). Define $\tilde{P} := P - (P_1 - P_2) \in  V$. As $\tilde{P} = P_0 + P_2$, with $P_0 \in \Omega$ and $P_2 \in R \lambda \cdot  \Omega$, we have $\tilde{P} \in (R \lambda + 1) \Omega$. Thus, by condition (ii),
\[
\tilde{P} \in (R\lambda + 1) \cdot ( \Omega \cap V) \subseteq (R \lambda + 1) \cdot  Z \cB.
\]
Therefore,
\[
P =  \tilde{P} + P_1 - P_2 \in (R \lambda + 1)Z  \cB + \lambda \cB + R \lambda \cB \subseteq (3R\lambda + 1)Z \cB.
\]
\end{proof}

We finish this section with the proof of Lemma \ref{main_lem}.

\begin{proof}[Proof of Lemma \ref{main_lem}]
\label{fin_up_sec}

Fix the constants $A$, $\epsilon_0$, and $\chi$ as in \eqref{defn:A}, \eqref{defn:chi}.

Let $B \in \cW_0$, $x \in 3B$, and $\ell \geq \ell_{\old} + \chi$. Set $B^\# = \kappa(B) \in \cW$, as defined in Lemma \ref{key_geom_lem}. Thus, $\diam(B^\#) \leq \diam(B)$,  $A B^\# \subseteq 2 B_0$, and $\dist(B^\#,B) \leq C_4 \diam(B)$ for a controlled constant $C_4$. By Lemma \ref{mainlem_1} and \eqref{eqn:ball_inc},
\begin{equation}
\label{eqn:501}
\sigma_\ell(z_{B^\#}) \cap V \subseteq C  C_{\old} \cB_{z_{B^\#}, \diam(B^\#)} \subseteq C  C_{\old} \cB_{z_{B^\#}, \diam(B)}.
\end{equation}

Note that $\diam(B) \leq \frac{1}{2} \diam(B_0)$ (see condition 2 of Lemma \ref{mdl}). We apply condition 3 of Lemma \ref{mdl}, with $B^\# \in \cW$, $x=z_{B^\#} \in \frac{6}{5} B^\#$, and $\delta = \diam(B) \in[ \diam(B^\#), \diam(B_0)]$. Thus, $\sigma(x)$ is $(z_{B^\#},\bar{C} \diam(B), R_4)$-transverse to $V$. By Lemma \ref{lem:tau trans}, $\sigma(x)$ is $(z_{B^\#}, \diam(B), \widehat{R})$-transverse to $V$, for $\widehat{R} = \bar{C}^m R_4$. Hence,
\[
\cB_{z_{B^\#},\diam(B)}/V  \subseteq \widehat{R} \cdot (\sigma(z_{B^\#}) \cap \cB_{z_{B^\#},\diam(B)})/V.
\]
By the inclusion $\sigma(z_{B^\#}) \subseteq \sigma_\ell(z_{B^\#})$, we obtain
\begin{equation}\label{eqn:502}
\cB_{z_{B^\#},\diam(B)}/V  \subseteq \widehat{R} \cdot (\sigma_\ell(z_{B^\#}) \cap \cB_{z_{B^\#},\diam(B)})/V.
\end{equation}

Since $z_{B^\#} \in \frac{6}{5} B^\#$ and $x \in 3B$, we have 
\begin{equation}\label{eqn:stuff1}
\begin{aligned}
|z_{B^\#} - x| &\leq \dist(B^\#, B) + 3 \diam(B) + (6/5) \diam(B^\#) \\
& \leq C_4 \diam(B) + 3 \diam(B) + (6/5) \diam(B) \\
&\leq C_5 \diam(B),
\end{aligned}
\end{equation}
for a controlled constant $C_5$. 

By Lemma \ref{lem:gamma_trans} and \eqref{eqn:stuff1}, $\sigma_{\ell+1}(x) \subseteq \sigma_\ell(z_{B^\#}) + C_T \cB_{z_{B^\#}, C_5 \diam(B)}$. Then by \eqref{eqn:ball_scale}, $\sigma_{\ell+1}(x) \subseteq \sigma_\ell(z_{B^\#}) + C_T C_5^m \cB_{z_{B^\#}, \diam(B)}$. Therefore,
\begin{equation}\label{eqn:503}
\sigma_{\ell+1}(x) + \cB_{z_{B^\#}, \diam(B)} \subseteq \sigma_{\ell}(z_{B^\#}) + \widetilde{C} \cB_{z_{B^\#}, \diam(B)},
\end{equation}
where $\widetilde{C} = C_T C_5^m + 1$ is a controlled constant.

We apply Lemma \ref{lem:stabv} to the convex set $\Omega = \sigma_\ell(z_{B^\#})$ in the Hilbert space $X= (\Po,\langle \cdot, \cdot \rangle_{z_{B^\#},\diam(B)})$. We take $\lambda = \widetilde{C}$ in Lemma \ref{lem:stabv}. Inclusions \eqref{eqn:501}, \eqref{eqn:502} imply hypotheses (i), (ii) of Lemma \ref{lem:stabv} with $R = \widehat{R}$, $Z = C C_{\old}$. So,
\begin{equation}\label{eqn:504}
\left(\sigma_{\ell}(z_{B^\#}) + \widetilde{C} \cB_{z_{B^\#},\diam(B)} \right) \cap V \subseteq C C_{\old} \cdot(3 \widehat{R} \widetilde{C} + 1) \cB_{z_{B^\#},\diam(B)}.
\end{equation}
From \eqref{eqn:503} and \eqref{eqn:504},
\begin{equation}\label{eqn:505}
(\sigma_{\ell+1}(x) +  \cB_{z_{B^\#},\diam(B)}) \cap V \subseteq C' C_{\old} \cdot \cB_{z_{B^\#},\diam(B)},
\end{equation}
for a controlled constant $C'$.

Finally, note that  $\widehat{C}^{-1} \cdot \cB_{z_{B},\diam(B)} \subseteq \cB_{z_{B^\#},\diam(B)} \subseteq \widehat{C} \cdot \cB_{z_{B},\diam(B)}$ for a controlled constant $\widehat{C}$; these inclusions follow from Lemma \ref{lem:poly2} and the estimate $|z_{B} - z_{B^\#}| \leq C \diam(B)$ (let $x=z_B$ in \eqref{eqn:stuff1}). Therefore, \eqref{eqn:505} implies that
\[
(\sigma_{\ell+1}(x) + \cB_{z_{B},\diam(B)}) \cap V \subseteq C C_{\old} \cdot \cB_{z_{B},\diam(B)},
\]
for a controlled constant $C$, as desired. This finishes the proof of Lemma  \ref{main_lem}.
\end{proof}

\subsection{Compatibility of the jets $(P_B)_{B \in \cW_0}$}

Our next result states that the polynomials $( P_B)_{B \in \cW_0}$ are pairwise compatible.

\begin{lemma}\label{FMCJ_lem}
There exist constants $\overline{\chi} \geq 5$ and $\widetilde{C} \geq 1$, determined by $m$ and $n$, such that the following holds. Let $(P_B)_{B \in \cW}$, $\ell^\#$, and $\bar{C}_{\ell^\#}$ be as in the statement of Lemma \ref{mdl}, and suppose $\ell^\# \geq \ell_{\old} + \overline{\chi}$. Then $P_B - P_{B'} \in \widetilde{C} C_{\old} \bar{C}_{\ell^\#} M \cB_{z_B,\diam(B)}$ for any $B,B' \in \cW_0$ with $(\frac{6}{5})B \cap (\frac{6}{5})B' \neq \emptyset$. Furthermore, $\overline{\chi} = O(\poly(D))$ and $\widetilde{C} = O(\exp(\poly(D)))$.
\end{lemma}

\begin{proof}[Proof of Lemma \ref{FMCJ_lem}]
We fix the constants $\epsilon_0$ and $\chi$ via Lemma \ref{main_lem}, and let $\overline{\chi} = \chi + 5$. Suppose $\ell^\# \in \N$ is picked so that $\ell^\# \geq \ell_{\old} + \overline{\chi}$, and $B,B' \in \cW_0$ satisfy $\frac{6}{5} B \cap \frac{6}{5}B' \neq \emptyset$.

Consider the following two cases for the Whitney cover $\cW_0 \subseteq \cW$. 

\emph{Case 1:} $\diam(B) > \epsilon_0 \diam(B_0)$ for all $B \in \cW_0$.

\emph{Case 2:} There exists $\widehat{B} \in \cW_0$ with $\diam(\widehat{B}) \leq \epsilon_0 \diam(B_0)$.

Suppose $\cW_0$ is as in Case 1. By the second containment in condition 6 of Lemma \ref{mdl}, we obtain
\begin{equation}\label{eqn:801}
\begin{aligned}
P_B - P_{B'} &= (P_B - P_0) + (P_0 - P_{B'}) \\
&\in \bar{C}_{\ell^\#} M \cB_{z_B,\diam(B_0)} + \bar{C}_{\ell^\#} M \cB_{z_{B'},\diam(B_0)}.
\end{aligned}
\end{equation}
Because $z_B,z_{B'} \in 2B_0$, we have $|z_B-z_{B'}| \leq 2 \diam(B_0)$. So by Lemma \ref{lem:poly2}, for a controlled constant $C$,
\begin{equation}\label{eqn:802}
\cB_{z_{B'},\diam(B_0)} \subseteq C 2^{m-1} \cB_{z_B,\diam(B_0)}.
\end{equation}
By \eqref{eqn:ball_scale}, because $\diam(B) > \epsilon_0 \diam(B_0)$, we conclude that 
\begin{equation}\label{eqn:803}
\cB_{z_B,\diam(B_0)} \subseteq (\epsilon_0)^{-m} \cB_{z_B,\diam(B)}.
\end{equation}
When put together, \eqref{eqn:801}, \eqref{eqn:802}, \eqref{eqn:803} give that 
\[
P_B - P_{B'} \in  \bar{C}_{\ell^\#} M (\epsilon_0)^{-m} C2^{m} \cB_{z_B,\diam(B)}.
\]
Note that $C' =  (\epsilon_0)^{-m} C 2^{m}$ is a controlled constant. We obtain the conclusion of Lemma \ref{FMCJ_lem} in Case 1, for any choice of $\widetilde{C} \geq C'$.

Now suppose $\cW_0$ is as in Case 2. By condition 7 in Lemma \ref{mdl},
\[
P_B - P_{B'} = (P_B - P_0) + (P_0 - P_{B'}) \in V.
\]
By the first part of condition 6 of Lemma \ref{mdl}, $P_{B'} \in \Gamma_{\ell^\#-3}(z_{B'},f,\bar{C}_{\ell^\#} M)$. Because $z_B \in \frac{6}{5} B$, $z_{B'} \in \frac{6}{5} B'$, $\frac{6}{5} B \cap \frac{6}{5} B' \neq \emptyset$, and $\diam(B') \leq 8 \diam(B)$ (see condition $(3)$ in Definition \ref{defn:whit_cover} of a Whitney cover) we have
\[
|z_B - z_{B'} | \leq 16 \diam(B).
\]
There exists $\widetilde{P}_B \in \Gamma_{\ell^\#-4}(z_{B}, f , \bar{C}_{\ell^\#} M)$ with $\widetilde{P}_B - P_{B'} \in C_T \bar{C}_{\ell^\#} M \cB_{z_{B}, 16 \diam(B)}$, thanks to Lemma \ref{lem:gamma_trans}. By \eqref{eqn:ball_scale}, $\widetilde{P}_B - P_{B'} \in 16^m C_T \bar{C}_{\ell^\#} M  \cB_{z_{B}, \diam(B)}$.

By condition 6 in Lemma \ref{mdl}, 
\[
P_{B} \in \Gamma_{\ell^\#-3}(z_{B},f,\bar{C}_{\ell^\#} M) \subseteq \Gamma_{\ell^\#-4}(z_{B},f,\bar{C}_{\ell^\#} M),
\]
so, because $\widetilde{P}_B \in \Gamma_{\ell^\#-4}(z_{B}, f , \bar{C}_{\ell^\#} M)$, by Lemma \ref{lem:sigma_gamma_rel},
\[
\widetilde{P}_B - P_B \in 2\bar{C}_{\ell^\#} M \cdot \sigma_{\ell^\#-4}(z_{B}).\]
Thus,
\[
\begin{aligned}
P_B - P_{B'} &= (P_B - \widetilde{P}_B) + (\widetilde{P}_B - P_{B'}) \\
&\in 2\bar{C}_{\ell^\#} M \cdot \sigma_{\ell^\#-4}(z_{B}) +  16^m C_T \bar{C}_{\ell^\#} M \cdot \cB_{z_{B}, \diam(B)} \\
&\subseteq C \bar{C}_{\ell^\#} M \cdot (  \sigma_{\ell^\#-4}(z_{B}) + \cB_{z_{B}, \diam(B)} ),
\end{aligned}
\]
and hence
\[
P_B - P_{B'} \in C \bar{C}_{\ell^\#} M \cdot (  \sigma_{\ell^\#-4}(z_{B}) + \cB_{z_{B}, \diam(B)} ) \cap V,
\]
for a controlled constant $C$.

Note that $\ell^\# - 5 \geq \ell_{\old} + \overline{\chi} - 5 = \ell_{\old} + \chi$, by definition of $\overline{\chi}$. We apply Lemma \ref{main_lem} (with $\ell = \ell^\# - 5$) to deduce that 
\[
(  \sigma_{\ell^\#-4}(z_{B}) + \cB_{z_{B}, \diam(B)} ) \cap V \subseteq C C_{\old} \cB_{z_B,\diam(B)}.
\]
Therefore, $P_B - P_{B'} \in C'' C_{\old} \bar{C}_{\ell^\#} M \cdot \cB_{z_B,\diam(B)}$ for a controlled constant $C''$. We obtain the conclusion of Lemma \ref{FMCJ_lem} in Case 2, for any choice of $\widetilde{C} \geq C''$. This concludes the proof of Lemma \ref{FMCJ_lem}.

\end{proof}

\subsection{Completing the Main Induction Argument} \label{sec_glue}

We complete the induction argument started in Section \ref{sec_ind:setup} by proving the Main Lemma for $K$. Thus, we fix data $(B_0, x_0, E, K, f, \ell^\#, M, P_0)$. In Section \ref{sec_ind:setup} we gave a proof of the Main Lemma for $K$ under the assumption $\#(E \cap B_0) \leq 1$. Thus, we may assume $\#(E \cap B_0) \geq 2$. See \eqref{twopoints}. Recall our task is to  construct a linear map $T : C(E) \times \Po \rightarrow C^{m-1,1}(\R^n)$ and prove it satisfies \eqref{main-task}.

The constant $\overline{\chi}$ in the Main Lemma for $K$ is taken to be $\overline{\chi}$ in Lemma \ref{FMCJ_lem}. Note that $\overline{\chi} \geq 5$ is a constant determined by $m$ and $n$, and $\overline{\chi} = O(\poly(D))$. Let $\ell^\#$, $C^\#$  satisfy \eqref{eqn:ind_consts} for $\overline{\chi}$ defined above and $\Lambda$ to be defined momentarily.

Given $P_0 \in \Gamma_{\ell^\#}(x_0,f,M)$, we apply the Main Decomposition Lemma (Lemma \ref{mdl}) to the data $(B_0, x_0, E, K, f, \ell^\#, M, P_0)$ to obtain a Whitney cover $\cW$ of $2B_0$, a DTI subspace $V \subseteq \Po$, and families $\{P_B\}_{B \in \cW}$ and $\{z_B\}_{B \in \cW}$. We defined in \eqref{defn:W0} the subfamily $\cW_0 = \{ B \in \cW : B \cap B_0 \neq \emptyset \}$ of $\cW$, so that $\cW_0$ is a Whitney cover of $B_0$. 

We apply Lemma \ref{old_fp_lem} with $x=z_B \in (6/5)B$ for $B \in \cW$. Thus, there exists a linear map $T_B : C(E) \times \Po \rightarrow C^{m-1,1}(\R^n)$ satisfying conditions 1,2,3 of Lemma \ref{old_fp_lem}, for $x=z_B$. 

Lemma \ref{mdl} (condition 6) asserts that $P_B \in \Gamma_{\ell^\# - 3}(z_B,f,\bar{C}_{\ell^\#} M)$ for $B \in \cW$. Because $\ell^\# - 3 \geq \ell^\# - \overline{\chi} = \ell_\old$, we have $P_B \in \Gamma_{\ell_\old}(z_B,f,\bar{C}_{\ell^\#} M)$. Thus, by Lemma \ref{old_fp_lem}, the function $F_B := T_B(f,P_B) \in C^{m-1,1}(\R^n)$ satisfies
\begin{equation} \label{aaa1}
\left\{
\begin{aligned}
&F_B = f \mbox{ on } E \cap (6/5)B,\\
&J_{z_B} F_B = P_B, \mbox{ and } \| F_B \| \leq C_{\old}  \bar{C}_{\ell^\#} M 
\end{aligned}
\right. \qquad (B \in \cW).
\end{equation} 
Since $\ell^\# \geq \ell_{\old} + \overline{\chi}$, we can apply Lemma \ref{FMCJ_lem} to conclude that
\begin{equation}
\label{aaa2}
|J_{z_B} F_B - J_{z_{B'}} F_{B'}|_{z_B,\diam(B)} = |P_B - P_{B'} |_{z_B, \diam(B)} \leq \widetilde{C} C_{\old} \bar{C}_{\ell^\#} M,
\end{equation}
for $B,B' \in \cW_0$ with $(6/5)  B \cap (6/5) B' \neq \emptyset$, and a controlled constant $\widetilde{C}$.

Let $\{\theta_B\}_{B \in \cW_0}$ be a partition of unity on $B_0$ adapted to the Whitney cover $\cW_0$ of $B_0$, satisfying the properties in Lemma \ref{lem: p of u}. Define $F : B_0 \rightarrow \R$ by
\[
F = \sum_{B \in \cW_0} F_B \theta_B \mbox{ on }B_0.
\]

We describe the basic properties of the function $F$. By Lemma \ref{lem:glue} and the conditions \eqref{aaa1}, \eqref{aaa2}, $F \in C^{m-1,1}(B_0)$  satisfies  $\| F \|_{C^{m-1,1}(B_0)} \leq C C_{\old} \bar{C}_{\ell^\#} M$ and $F=f$ on $E \cap B_0$, where $C$ is a controlled constant.

Because each $F_B$ depends linearly on $(f,P_B)$, and each $P_B$ depends linearly on $(f,P_0)$ (see condition 8 in Lemma \ref{mdl}), $F$ depends linearly on $(f,P_0)$.

By conditions 5 and 6 in Lemma \ref{mdl}, $z_B = x_0$ and $P_B = P_0$ if $x_0 \in (6/5) B$. By the support properties of $\theta_B$ (see Lemma \ref{lem: p of u}), $J_{x_0} \theta_B \neq 0 \implies x_0 \in (6/5) B$. Thus, $J_{x_0} F_B = P_0$ if $J_{x_0} \theta_B \neq 0$. Therefore, using that $\sum_{B \in \cW_0} \theta_B = 1$ on $B_0$,
\[
\begin{aligned}
J_{x_0} F &= \sum_{B \in \cW_0 : x_0 \in \frac{6}{5}B} J_{x_0} (F_B \theta_B) = \sum_{B \in \cW_0 : x_0 \in \frac{6}{5}B} J_{x_0} F_B \odot_{x_0} J_{x_0} \theta_B \\
& =  \sum_{B \in \cW_0 : x_0 \in \frac{6}{5}B} P_0 \odot_{x_0} J_{x_0} \theta_B = P_0 \odot_{x_0} 1 = P_0.
\end{aligned}
\]

We extend $F :B_0 \rightarrow \R$ to all of $\R^n$ using Lemma \ref{lem:dom_ext} (an outcome of the classical Whitney extension theorem). This guarantees the existence of a function $\widehat{F} \in C^{m-1,1}(\R^n)$, depending linearly on $F$, with $\widehat{F}|_{B_0} = F$, and
\[
\| \widehat{F} \|_{C^{m-1,1}(\R^n)} \leq  C \| F \|_{C^{m-1,1}(B_0)} \leq C' C_{\old} \bar{C}_{\ell^\#} M.
\]
Here, $C, C'$ are controlled constants. By the properties of $F$, stated above, and since $\widehat{F}|_{B_0} = F$, we deduce that $\widehat{F} = F =  f$ on $E \cap B_0$ and $J_{x_0} \widehat{F} = J_{x_0} F = P_0$ (recall $x_0 \in B_0$). Therefore, we have shown:
\begin{equation}\label{eqn:main_cond_barF}
\left\{
    \begin{aligned}
    &\widehat{F} = f \mbox{ on } E \cap B_0 \\
    &J_{x_0} \widehat{F} = P_0 \\
    &\| \widehat{F} \|_{C^{m-1,1}(\R^n)} \leq C' C_{\old} \bar{C}_{\ell^\#} M.
    \end{aligned}
\right.
\end{equation}

We choose $\Lambda$ in \eqref{eqn:ind_consts}, now, to ensure the inequality $C^\# \geq C'  C_{\old} \bar{C}_{\ell^\#}$. From Lemma \ref{mdl} recall that $\bar{C}_{\ell^\#} = C \cdot (D+1)^{\ell^\#}$ for a controlled constant $C \geq 1$. From \eqref{eqn:ind_consts}, $C_{\old} = C^\#(K-1)$, $\ell^\# = \ell^\#(K)$, and $C^\#=C^\#(K)$ have the form  $\ell^\# =  \overline{\chi} \cdot (K+1)$, $C^\# = \Lambda^{(K+1)^2 + 1}$ and $C_{\old} = \Lambda^{K^2 + 1}$. Thus, the desired inequality is equivalent to 
\[
\frac{C^\#}{C_{\old}} = \Lambda^{2K + 1} \geq C' \cdot C \cdot (D+1)^{\overline{\chi} \cdot (K+1)}.
\]
Fix a controlled constant $\Lambda$ satisfying the earlier condition \eqref{eqn:Lam_ass1}, in addition to $\Lambda \geq C'  C  (D+1)^{\overline{\chi}}$ so that the preceding inequality is valid, and $C^\# \geq C'  C_{\old} \bar{C}_{\ell^\#}$. Therefore, \eqref{eqn:main_cond_barF} implies
\[
\| \widehat{F} \|_{C^{m-1,1}(\R^n)} \leq C^\# M.
\] 
Because $\widehat{F}$ depends linearly on $F$ and $F$ depends linearly on $(f,P_0)$, we have that $\widehat{F} = T(f,P_0)$ for some linear map $T : C(E) \times \Po \rightarrow C^{m-1,1}(\R^n)$.

Thus we have defined a linear map $T : C(E) \times \Po \rightarrow C^{m-1,1}(\R^n)$ and verified the conditions in \eqref{main-task} (see \eqref{eqn:main_cond_barF}). This completes the proof of the Main Lemma for $K$ (Lemma \ref{lml}).

\section{Proofs of the Main Results}\label{sec:mainproofs}

\subsection{Proof of Theorem \ref{thm:sharpfinitenessprinciple}}\label{subsec:mainproofs1}

We give the proof of Theorem \ref{thm:sharpfinitenessprinciple}. Recall that Lemma \ref{lml} specifies a family of constants $\ell^\#(K)$ and $C^\#(K)$ ($K \in \{-1,0,\dots\}$).

Let $E \subseteq \R^n$ be finite. Fix a closed ball $B_0 \subseteq \R^n$ containing $E$, and a point $x_0 \in B_0$. Set $K_0 := 4 m D^2$, $\ell^\# := \ell^\#(K_0)$, and $C^\# := C^\#(K_0)$.

By Corollary \ref{cor:ballCompBound}, we have $\cC(E|5B_0) \leq K_0$. Lemma \ref{lml} guarantees the existence of a linear mapping $T : C(E) \times \Po \rightarrow C^{m-1,1}(\R^n)$ satisfying, for any $(f,P) \in C(E) \times \Po$, if $P \in \Gamma_{\ell^\#}(x_0,f,M)$ then
\begin{enumerate}
    \item $T(f,P) = f$ on $E$.
    \item $J_{x_0} T(f,P) = P$.
    \item $\| T(f,P) \|_{C^{m-1,1}(\R^n)} \leq C^\# M$.
\end{enumerate}

For the proof of part (A) of Theorem \ref{thm:sharpfinitenessprinciple}, set $k^\# := (D+1)^{\ell^\# + 3}$. We are given that $f$ satisfies the finiteness hypothesis $\cF\cH(k^\#, M)$ for some $M > 0$. According to Lemma \ref{lem:gamma_trans}, $\Gamma_{\ell^\#}(x_0,f,M) \neq \emptyset$. Let $P \in \Gamma_{\ell^\#}(x_0,f,M)$. Set $F = T(f,P)$. According to the above conditions, $F = f$ on $E$ and $\| F \|_{C^{m-1,1}(\R^n)} \leq C^\# M$. Thus, $\| f \|_{C^{m-1,1}(E)} \leq C^\# M$. This establishes part (A) of Theorem \ref{thm:sharpfinitenessprinciple}.

We next prove part (B) of Theorem \ref{thm:sharpfinitenessprinciple}. By Lemma \ref{lem: linear select} there exists a linear map $P^{x_0}_{\ell^\#} : C(E) \rightarrow \Po$ such that if $f$ satisfies $\cF\cH(k^\#, M)$ then $P^{x_0}_{\ell^\#}(f) \in \Gamma_{\ell^\#}(x_0,C_{\ell^\#}M)$, with $C_{\ell^\#} = C' (D+1)^{\ell^\#}$ for a controlled constant $C'$.

Define a linear map $\widehat{T}: C(E) \rightarrow C^{m-1,1}(\R^n)$ by $\widehat{T}(f) := T(f,P^{x_0}_{\ell^\#}(f))$.

Suppose $f \in C(E)$ and let $ M > \| f \|_{C^{m-1,1}(E)}$. Evidently, $f$ satisfies $\cF\cH(k^\#,M)$. Hence, $P^{x_0}_{\ell^\#}(f) \in \Gamma_{\ell^\#}(x_0,C_{\ell^\#} M)$. By property 3 of $T$,
\[
\| \widehat{T}(f) \|_{C^{m-1,1}(\R^n)} = \| T(f, P^{x_0}_{\ell^\#}(f)) \|_{C^{m-1,1}(\R^n)} \leq C^\# C_{\ell^\#} M = C_0^\# M,
\]
with $C_0^\# := C^\# C_{\ell^\#}$. Since $M > \| f \|_{C^{m-1,1}(E)}$ is arbitrary,  $\| \widehat{T}(f) \|_{C^{m-1,1}(\R^n)} \leq C_0^\# \| f \|_{C^{m-1,1}(E)}$, as desired. By property $1$ of $T$, we have $\widehat{T}(f) = f$ on $E$.  This completes the proof of part (B) of Theorem \ref{thm:sharpfinitenessprinciple}.

We remark at last on the form of the constants. Recall that $C^\# = C^\#(K_0) = \Lambda^{(K_0+1)^2 + 1}$, $\Lambda$ is a controlled constant, and $K_0 = 4m D^2$. Thus, $C^\#$ is a controlled constant. Similarly, since $\ell^\# = \ell^\#(K_0) = \overline{\chi} \cdot (K_0+1)$ with $\overline{\chi} = O(\poly(D))$, we have $\ell^\# = O(\poly(D))$, and thus, $C_{\ell^\#} = C' (D+1)^{ \ell^\#}$ is a controlled constant. Therefore, $k^\# = (D+1)^{\ell^\#+1}$ and $C^\#_0 = C^\# C_{\ell^\#}$ are controlled constants. This completes the proof of Theorem \ref{thm:sharpfinitenessprinciple}.

\subsection{Proofs of Theorem \ref{thm: new c sharp} and \ref{thm: lin op}}\label{subsec:mainproofs2}

Let $E \subseteq \R^n$ be an arbitrary set, and let $f : E \rightarrow \R$. We claim that
\begin{equation}\label{eqn:fin_E}
\| f \|_{C^{m-1,1}(E)} = \sup_{\widehat{E} \subseteq E \text{ finite}} \| f|_{\widehat{E}}\|_{C^{m-1,1}(\widehat{E})}.
\end{equation}
To prove \eqref{eqn:fin_E}, we use a compactness argument adapted from the proof of Lemma 18.2 of \cite{F1}.

First note that if $\widehat{E} \subseteq E$ then $ \| f \|_{C^{m-1,1}(E)} \geq \| f|_{\widehat{E}} \|_{C^{m-1,1}(\widehat{E})}$, by definition of the trace seminorm. Therefore, the left-hand side of \eqref{eqn:fin_E} is greater than or equal to the right-hand side of \eqref{eqn:fin_E}.

For the reverse inequality, it suffices to demonstrate that 
\begin{equation}\label{eqn:fin_red1}
\begin{aligned}
& \| f|_{\widehat{E}}\|_{C^{m-1,1}(\widehat{E})} \leq 1 \mbox{ for all finite } \widehat{E} \subseteq E \\
&\implies f \in C^{m-1,1}(E) \mbox{ and } \| f \|_{C^{m-1,1}(E)} \leq 1.
\end{aligned}
\end{equation}

Let $\eta > 0$ be arbitrary. The hypothesis in \eqref{eqn:fin_red1} implies the following:
\begin{equation}\label{eqn:fin_red2}
\begin{aligned}
&\mbox{For all finite } \widehat{E} \subseteq E \mbox{ there exists } F_{\widehat{E}} \in C^{m-1,1}(\R^n) \\
&\mbox{satisfying } F_{\widehat{E}} = f \mbox{ on } \widehat{E} \mbox{ and } \| F_{\widehat{E}} \|_{C^{m-1,1}(\R^n)} \leq 1 + \eta.
\end{aligned}
\end{equation}
We define
\[
\mathcal{D} = \{ F \in C^{m-1,1}(\R^n) : \| F \|_{C^{m-1,1}(\R^n)} \leq 1 + \eta \},
\]
equipped with the local $C^{m-1}$ topology defined by the family of seminorms
\[
\rho_R(F) := \sup_{|x| \leq R} \max_{|\alpha| \leq m-1} |\partial^\alpha F(x)| \qquad (R > 0).
\]
We define
\[
\mathcal{D}(x) = \{ F \in \mathcal{D} : F(x) = f(x) \} \mbox{     for each } x \in E.
\]
Then \eqref{eqn:fin_red2} implies that $ \bigcap_{x \in \widehat{E}} \mathcal{D}(x) \neq\emptyset$ for any finite subset $\widehat{E} \subseteq E$.

On the other hand, each $\mathcal{D}(x)$ is a closed subset of $\mathcal{D}$, and $\mathcal{D}$ is compact by the Arzela-Ascoli theorem. Therefore, the intersection of $\mathcal{D}(x)$ over all $x \in E$ is nonempty. Thus, there exists $F \in C^{m-1,1}(\R^n)$ satisfying $F = f$ on $E$ and $\| F \|_{C^{m-1,1}(\R^n)} \leq 1 + \eta$. Since $\eta > 0$ is arbitrary, by definition of the trace seminorm we have $\| f \|_{C^{m-1,1}(E)} \leq 1$.

This completes the proof of \eqref{eqn:fin_red1}. With this, \eqref{eqn:fin_E} is established.

We take $C^\# \geq 1$ and $k^\# \in \N$ as in Theorem \ref{thm:sharpfinitenessprinciple}. Note that the constants $C^\#$, $k^\#$ in Theorem \ref{thm:sharpfinitenessprinciple} satisfy $C^\# = O(\exp(\poly(D)))$ and $k^\# = O(\exp(\poly(D)))$. Thus, $C^\# , k^\# \leq \exp(\gamma D^k)$ for absolute constants $\gamma,k > 0$ (independent of $m,n, E$).

We first prove Theorem \ref{thm: new c sharp}. Let $E \subseteq \R^n$ be arbitrary, and let $\widehat{E} \subseteq E$ be a finite subset. By hypothesis of Theorem \ref{thm: new c sharp}, we are given $f : E \rightarrow \R$ satisfying: For all $S \subseteq \widehat{E}$ with $\#(S) \leq k^\#$ there exists $F^S \in C^{m-1,1}(\R^n)$ satisfying $F^S = f$ on $S$ and $\| F^S \|_{C^{m-1,1}(\R^n)} \leq 1$. Then $f|_{\widehat{E}}  : \widehat{E} \rightarrow \R$ satisfies the finiteness hypothesis $\cF\cH(k^\#,1)$ (see \eqref{eqn:FH}). Part (A) of Theorem \ref{thm:sharpfinitenessprinciple} ensures that $\| f|_{\widehat{E}} \|_{C^{m-1,1}(\widehat{E})} \leq C^\#$. We deduce that $f \in C^{m-1,1}(E)$ and $\| f \|_{C^{m-1,1}(E)} \leq C^\#$ by \eqref{eqn:fin_E}. This completes the proof of Theorem \ref{thm: new c sharp}.

We will prove Theorem \ref{thm: lin op} for finite $E$. The general case of Theorem \ref{thm: lin op} then follows by a standard argument using Banach limits. See Section 17 of \cite{F2}.

For $E \subseteq \R^n$ finite, we write $C(E)$ to denote the set of all real-valued functions on $E$. Note that $C(E) = C^{m-1,1}(E)$ because $E$ is finite. By part (B) of Theorem \ref{thm:sharpfinitenessprinciple}, there exists a linear map $T : C(E) \rightarrow C^{m-1,1}(\R^n)$ satisfying $Tf = f$ on $E$ and $\| Tf \|_{C^{m-1,1}(\R^n)} \leq C^\# \| f \|_{C^{m-1,1}(E)}$ for all $f \in C(E)$. This completes the proof of Theorem \ref{thm: lin op} for finite $E$.

\bibliography{references}
\bibliographystyle{plain}

\end{document}